\documentclass[12pt]{article}
\usepackage{amsmath, amssymb, amscd, amsthm, amsfonts}
\usepackage[title]{appendix}
\usepackage{graphicx, graphics}
\usepackage{blindtext}
\usepackage{titlesec}
\usepackage{tabularx}
\usepackage{multicol}
\usepackage{enumitem, float}
\usepackage{multirow}
\usepackage{tikz}
\usepackage{tikz-cd}
\usepackage[backend=biber,style=alphabetic]{biblatex}
\usepackage{subfigure, subcaption}
\usepackage[all]{xy}
\usetikzlibrary{shapes.geometric, arrows.meta}
\tikzstyle{startstop} = [rectangle, rounded corners, minimum width=3cm, minimum height=1cm,text centered, text width=3cm, draw=black]
\tikzstyle{arrow} = [thick,->,>=stealth]
\usepackage{hyperref, comment}
\usepackage[margin=0.8in]{geometry}
\usepackage{mathtools}

\title{A K\"{unneth} theorem for the cyclic homology of $A_\infty$ algebras}
\author{Henry Azubuike}
\date{}

\newtheorem{theorem}{Theorem}
\newtheorem{lemma}[theorem]{Lemma}
\newtheorem{conjecture}[theorem]{Conjecture}
\newtheorem{corollary}[theorem]{Corollary}
\newtheorem{proposition}[theorem]{Proposition}

\newtheorem{definition}[theorem]{Definition}

\begin{document}

\newcommand{\AAA}{A}
\newcommand{\BB}{B}
\newcommand{\bee}{\mathfrak{b}}
\newcommand{\Bee}{\mathfrak{B}}
\newcommand{\rr}{\mathbb{R}}
\newcommand{\bfa}{{\bf a}}
\newcommand{\bfb}{{\bf b}}
\newcommand{\ua}{\mathbf{1}_A}
\newcommand{\ub}{\mathbf{1}_B}
\newcommand{\emm}{\mathfrak{m}}
\newcommand{\al}{\alpha}
\newcommand{\bara}{\underline{a}}
\newcommand{\barb}{\underline{b}}
\newcommand{\barc}{\underline{c}}
\newcommand{\barepsilonA}{\underline{\epsilon}^{A}}
\newcommand{\barepsilonB}{\underline{\epsilon}^{B}}

\maketitle

\begin{abstract}

The cyclic homology of a $\mathbb{Z}/2\mathbb{Z}$-graded, smooth and proper $A_\infty$ category satisfying the Hodge-to-de-Rham degeneration property carries the structure of a polarized semi-infinite Hodge structure or the so-called EP-structure. Given two $A_\infty$ algebras $A$ and $B$ with the above conditions, we construct a K\"unneth map from the tensor product of their cyclic homologies to the cyclic homology of the $A_\infty$ tensor product $A\otimes_\infty B$ and show that it respects the EP-structures. As an application, we show that if $A$ and $B$ are equipped with weak Calabi-Yau structures, then $A\otimes_\infty B$ also inherits a weak Calabi-Yau structure. Also, we show that the K\"unneth quasi-isomorphism respects good splittings of the Hodge filtration on $A$ and $B$ compatible with the weak Calabi-Yau structure. Our explicit calculations rely on the combinatorial (tree) description of the tensor product of $A_\infty$ algebras developed in \cite{amorim2016tensor}.
\end{abstract}

\tableofcontents

\section{Introduction}\label{sec:intro}
The notion of a variation of semi-infinite hodge structure was introduced by Barannikov \cite{barannikov2001quantum} as a generalization of {\it variations of Hodge structures} on the moduli spaces of non-commutative deformations of complex manifolds. Mirror symmetry predicts the existence of certain mirror pairs of Calabi–Yau manifolds $X$ and $Y$, so that the Gromov–Witten invariants of $X$, the $A$-side, can be extracted from certain Hodge-theoretic invariants of $Y$, the $B$-side. The mirror relationship between the enumerative invariants of $X$ and $Y$ can be stated as an isomorphism of their variations of semi-infinite Hodge structures (VSHS), see \parencite[Section 1]{barannikov2001semiinfinitehodgestructuresmirror}, \parencite[Section 1]{IRITANI20091016}, and \parencite[Section 1]{sheridan2020formulae}.

Kontsevich proposed the homological mirror symmetry conjecture \parencite[Page 18]{kontsevich1995homological} as a generalization of Hodge-theoretic mirror symmetry. It predicts a quasi-equivalence of $A_\infty$ categories  - the (derived) Fukaya category of the $A$-side, and the derived category of coherent sheaves on the $B$-side. One question is how to get from this very sophisticated statement to the classical one. The expectation is that the negative cyclic homology of a smooth, proper, Calabi-Yau $A_\infty$-category carries a VSHS. So taking the negative cyclic homology of the Fukaya category should recover the $A$-side VSHS and taking the negative cyclic homology of the bounded derived category should recover the $B$-side VSHS. This was made precise by Barannikov \cite{barannikov2002non} and Katzarkov-Kontsevich-Pantev \parencite[Section 2]{katzarvovKonPan}, as well as the works of Gantara-Perutz-Sheridan \cite{ganatra2015mirror}, and \cite{sheridan2020formulae}.
 
We study the tensor product of two smooth and proper, weak proper (right) Calabi-Yau $A_\infty$ algebras $A$ and $B$, and their Hochschild invariants. The underlying VSHS in this case is the so-called EP-structure $\left(HC^-_*(A), \nabla_{d/du}, \langle-,-\rangle_{\text{hres}}\right)$ where the  higher residue pairing and $u$-connection are maps 
    \begin{align*}
        \langle-,-\rangle_{\text {hres}}: H C_{*}^{-}(A) \otimes H C_{*}^{-}(A) \rightarrow k[[u]] \\
        \nabla_{\frac{d}{d u}}: H C_{*}^{-}(A) \rightarrow u^{-2} H C_{*}^{-}(A)
    \end{align*}
defined in (\ref{defnhres}) and (\ref{defnconnection}), (and similarly for $B$). We construct a K\"unneth map from the tensor product of the $A_\infty$ negative cyclic homologies to the negative cyclic homology of their $A_\infty$ tensor product $A\otimes_\infty B$ and show that it is a morphism of the EP-structures.

Our main result is stated below.

\begin{theorem}\label{mainintrothm}
    Let $A$ and $B$ be two $\mathbb{Z}/2\mathbb{Z}$-graded, finite-dimensional, smooth, proper, (weak) Calabi-Yau $A_\infty$ algebras over a field $\mathrm{k}$ of characteristic zero. There is a chain map of the negative cyclic complexes 

$$ k+uK: CC_*^-(A) \otimes_{k[[u]]} CC_*^-(B) \rightarrow CC_{*}^-(A\otimes_\infty B) $$ defined in (\ref{kunnethmap_defn}, \ref{cycKunneth_defn}), such that

(a) $k+uK$ is a quasi-isomorphism of complexes. {[= Proposition (\ref{cycqism})]}

(b) $k+uK$ is compatible with the higher residue pairings and $u$-connections, and hence, a morphism of EP-structures. { [= Theorem (\ref{mainthm})]}

\end{theorem}

The K\"unneth maps $k, K: C_*(\AAA) \otimes C_*(\BB) \rightarrow C_{*}(\AAA \otimes_\infty \BB)$ are constructed by extending the well-known shuffle and cyclic shuffle maps in the associative algebra case (see \parencite[Sections 4.2 and 4.3]{loday2013cyclic}). Shklyarov in \parencite[Theorem 2]{shklyarov2016hodge} carried out this extension in the $dg$ algebra case and proved that the cyclic K\"unneth quasi-isomorphism respects the $u$-connections on periodic cyclic homology. Theorem \ref{mainintrothm} generalizes Shklyarov's work to the $A_\infty$ world, and further shows that under the finiteness conditions - smooth and proper, the (cyclic) K\"unneth map also respects the Mukai and higher residue pairings on negative cyclic homology.

A major challenge to extending results concerning tensor products to $A_\infty$ algebras (or categories) is that the operation of tensor product for $A_\infty$ algebras is not canonical, and may not preserve extra structure. It is well-studied by many authors however. In particular, Saneblidze and Umble \cite{saneblidze2000diagonal, saneblidze2004diagonals}, and Markl and Shnider \cite{markl2006associahedra} formulated it using the language of operads. More recently, Amorim \cite{amorim2016tensor} generalized the construction to filtered $A_\infty$ algebras, without the language of operads. In this paper, we work with the version of Amorim's construction for classical (unfiltered) $A_\infty$ algebras, where the $\emm_n$ operations are given by trees. However, this tensor product framework does not preserve cyclic structures. For a cyclic structure-preserving tensor product, see \parencite[Theorem 1.1]{amorim2019tensor}. To overcome this hurdle, we replace the cyclic structure with a weak Calabi-Yau structure on $A$. 

We use the notion of a {\it weak proper (right)} Calabi-Yau structure in the sense of Amorim-Tu \parencite[Definition 2.13]{amorim2022morita}. Thus, for our smooth and proper $A_\infty$ algebra $A$, a weak proper (right) Calabi-Yau structure on $A$ is a non-degenerate element $\phi \in H H_{*}(A)^{\vee}$, see Definition \ref{defnWCYstr} for more details. For the related notion of a weak smooth (left) Calabi-Yau structure, see e.g. \parencite[Section 2.5]{holstein2024koszul} or \parencite[Section 3]{Brav_Dyckerhoff_2019}. In particular, throughout this paper, the term `` weak Calabi-Yau structure'' refers to this proper (right) Calabi-Yau structure.  Such a weak Calabi-Yau structure gives rise to the Poincar\'e-type isomorphism:

\begin{proposition}[= Proposition (\ref{poincare})] A weak Calabi-Yau structure $\phi$ determines an isomorphism
   $ D: HH^*(A) \longrightarrow HH_*(A)$ defined in equation (\ref{eq D}). This is an isomorphism of $HH^*(A)$-modules, that is, $$D(\theta \cup \psi) = \theta \cap D(\psi)$$ for Hochschild cochains $\theta, \psi \in HH^*(A)$
\end{proposition} 

The non-degeneracy of the Mukai pairing implies that the map $\langle-,-\rangle_{Muk}: HH_*(A)\to HH_*(A)^\vee$ defined by $\alpha\mapsto \langle \alpha,-\rangle_{Muk}$ is an isomorphism. Thus, the weak Calabi-Yau structure $\phi$ may be represented by a unique Hochschild homology class $\omega\in HH_*(A)$ such that we can write $\phi=\langle \omega,-\rangle_{Muk}$.

\begin{proposition}[= Lemma \ref{lemma40} and Corollary \ref{cor40}]\label{prop:intro}
    $\omega=D(1)$. Hence $\omega$ generates $HH_*(A)$ as a $HH^*(A)$-module.
\end{proposition}

Let $\phi_A=\langle \omega_A,-\rangle_{Muk}$, and $\phi_B=\langle \omega_B,-\rangle_{Muk}$ be weak Calabi-Yau structures on $A$ and $B$ respectively. We obtain a weak Calabi-Yau structure on $A\otimes_\infty B$. 

\begin{theorem}[= Theorem \ref{thm42}]
    $\phi_{A\otimes_\infty B}:=\langle\omega^\otimes, \ -\rangle_{Muk}$ where $\omega^\otimes := k(\omega_A \otimes \omega_B)$ is a weak Calabi-Yau structure on $A\otimes_\infty B$.
    
\end{theorem}

Extending Proposition \ref{prop:intro} to the tensor product of $A$ and $B$, we obtain the following.

\begin{proposition} [= Corollary \ref{cor43}]
    $\omega^\otimes$ generates $HH_*(A\otimes_\infty B)$ as $HH^*(A\otimes_\infty B)$-module.
\end{proposition}

The negative cyclic homology of $A$, $HC_*^-(A)$ carries the noncommutative Hodge filtration $F^pHC_*^-(A):=u^pHC_*^-(A)$. A {\it splitting of the Hodge filtration} is a choice of a $\mathrm{k}$-linear morphism $s:HH_*(A)[[u]]\xrightarrow{} HC_*^-(A)$ whose reduction modulo $u$ is the canonical identification $HH_*(A)\cong HC_*^-(A)/uHC_*^-(A)$. Following Amorim-Tu \parencite[Definition 3.7]{amorim2022categorical}, we call such a splitting {\it good} and {\it $\omega$-compatible} when, in addition, it satisfies the required compatibility with the $u$-connection, and the higher residue pairing induced by the weak Calabi-Yau structure. The complete definition is given in Definition \ref{goodsplittings}. 

We apply the results of Theorem \ref{mainintrothm} to obtain good splittings of the Hodge filtration on $A\otimes_\infty B$ compatible with the Calabi-Yau structure $\omega^{\otimes}$.

  \begin{theorem} [= Theorem \ref{thmgoodsplitting}]
    Let $s_A$ and $s_B$ be splittings of the Hodge filtration on $A$ and $B$ respectively. Then $s_{A\otimes_\infty B}: = (k+uK)\circ(s_A \otimes s_B)\circ k^{-1}$ is a splitting of the Hodge filtration on $A\otimes_\infty B$. Further, if $s_A, s_B$ are $\omega_A, \omega_B$-compatible good splittings of the Hodge filtrations on $A,$ and $B$ respectively, then $s_{A\otimes_\infty B}$ is a $\omega^{\otimes }$-compatible good splitting of the Hodge filtration of $A\otimes_\infty B$, where $\omega^\otimes=k(\omega_A\otimes\omega_B)$.
\end{theorem}

The existence of $\omega$-compatible good splittings is very non-trivial. Amorim-Tu \parencite[Corollary 3.8]{amorim2022categorical} showed that they exist in the case where the Hochschild cohomology ring $HH^*(\AAA)$ is semi-simple. They further showed that in the case of a versal deformation of $A$, the set of $\omega$-compatible good splittings of the Hodge filtration on the central fiber is in bijection with the set of primitive forms of the corresponding variational semi-infinite Hodge structure, \parencite[Theorem 4.2]{amorim2022categorical}. With the data of a primitive form on a VSHS, one may construct a formal Frobenius manifold structure on the versal deformation of the $A_\infty$-category. This is the categorical construction of Frobenius manifolds from $A_\infty$ categories, analogous to Saito-Takahashi \parencite[Section 7]{saito2008primitive}'s theory of primitive forms in singularity theory.\\

\subsection{Notation and conventions} We shall work with $\mathbb{Z}/2\mathbb{Z}$-graded strictly unital $A_\infty$ algebras $A$ over a field $\mathrm{k}$ of characteristic zero. We let $|a|$ denote the degree of $a \in A$, and $|a|' = |a| - 1$ its shifted degree in $A[1]$. All our degrees will be shifted unless explicitly specified. A $\mathbb{Z}/2\mathbb{Z}$-graded $A_\infty$ algebra is a $\mathbb{Z}/2\mathbb{Z}$-graded vector space \[ A = \bigoplus_{p\in \mathbb{Z}/2\mathbb{Z}} A^p\] with graded homogeneous $\mathrm{k}$-linear maps 
\[
\emm_n: A[1]^{\otimes n} \longrightarrow A[1], \  \ \ n\geq 1\] of degree $1$ (mod $2$) satisfying the following relations: 
\[
\sum (-1)^{|a_1|'+...+|a_{i-1}|'}\emm_{n-j+1} \left(a_1,...,a_{i-1}, \emm_j (a_i,...,a_{i+j-1}),a_{i+j},...,a_n\right) = 0.
\]

The strict unitality condition means we have $\emm_n\left(\ldots, \ua, \ldots\right)=0$ for all $n \neq 2$, and $\emm_2\left(\ua, a\right)=a=(-1)^{|a|} \emm_2\left(a, \ua\right)$.

We use bold letters to denote elements of the Hochschild chains $\bfa: = a_0 \otimes a_1\otimes a_2\otimes \ldots \otimes a_n$ or $\bfa= a_0|a_1,a_2,...,a_n = a_0|a$, where $|\bfa|'=|a_0|+\sum_{i=1}^n\ |a_i|'$. The commas, as well as $|$ and $\otimes$ will be used interchangeably to mean the same thing. Also, denote $\bara_i = a_i \otimes \ub$, and $\barb_i = \ua \otimes b_i$ in $A\otimes_\infty B$, where $|\bara_i|'=|a_i\otimes 1|'=|a_i|'$. 

Furthermore, for $a=a_1\otimes \cdots \otimes a_n \in A[1]^{\otimes n}$, we use Sweedler notation for the coproduct on the tensor coalgebra, namely: 
\begin{equation}
    \sum_{(a)} a^{(1)} \otimes a^{(2)} = \sum_{i=1}^n (a_1 \otimes \ldots \otimes a_i)\otimes (a_{i+1} \otimes \ldots \otimes a_n)
\end{equation}
and in this case, we denote, say, $|a^{(1)}|'=|a_1|'+\cdots+|a_i|'$.

\subsection*{Acknowledgement} I would like to thank my Ph.D Advisor Lino Amorim for helpful comments and suggestions throughout this work. We also thank the referees for their careful reading of the manuscript and for their valuable comments and suggestions.

\section{Hochschild invariants of A-infinity algebras}
\label{makereference1}
We review the basic definitions that we will need, following conventions from \parencite[Section 2]{amorim2022categorical}. 

Let $A$ be an $A_\infty$-algebra.

\begin{definition}
The space of reduced Hochschild cochains of length $n$ is defined as:
$$
C^*(\AAA)^n:= \operatorname{Hom}_\mathrm{k}\left(\left(\frac{A}{k \cdot \ua}[1]\right)^{\otimes n}, A[1]\right)[-1]
$$

\end{definition}
The reduced Hochschild cochain complex is defined as $C^{*}(A)=\prod_{n \geq 0} C^{*}(A)^n$. 

Let $\tilde{C}^*(A)= \prod_{n\geq 0} \operatorname{Hom^*(A[1]^{\otimes n}, A[1])}$ denote the unreduced cochain complex of $A$. On the unreduced cochain complex one defines the Gerstenhaber product
$$
\varphi \circ \psi(a)=\sum(-1)^{|a^{(1)}|'|\psi|^{\prime}} \varphi\left(a^{(1)}, \psi\left(a^{(2)}\right), a^{(3)}\right)
$$
The corresponding bracket $[\varphi, \psi]:=\varphi \circ \psi-(-1)^{|\varphi|^{\prime}|\psi|^{\prime}} \psi \circ \varphi$ is a graded Lie bracket, and can be used to define the Hochschild differential on $ \tilde{C}^{*}(A)$ as $\delta(-):=[\emm,-]$, where $\emm=\prod_k \emm_k$ are the $A_{\infty}$-operations. 

Now the relation $\emm_2(1_A, a)= a\neq 0$ implies $\emm_2\notin C^*(A)$, and thus $\emm\notin C^*(A)$. However, the strict unitality of $A$ guarantees that $\delta(-)$ preserves the subcomplex of reduced chains $C^*(A)$. Therefore, the unreduced Hochschild differential $\delta$ descends to a differential in the reduced complex $C^*(A)$. The corresponding cohomology is called the Hochschild cohomology of $A$ and is denoted by $H H^{*}(A)$.

\begin{definition}
The (reduced) Hochschild chain complex of $A$ is defined as the vector space
    $$
C_*(\AAA):=\bigoplus_{n=0}^{\infty} \AAA \otimes\left(\frac{\AAA[1]}{\mathrm{k} \cdot \ua}\right)^{\otimes n} .
$$

with differential $\bee$ of degree $+1$ (mod $2$) defined as
$$
\bee\left(\bfa\right)=\bee_0(\bfa)+\bee_1(\bfa)=\sum(-1)^{|a_0|'+|a^{(1)}|'} a_0 \otimes a^{(1)} \otimes \emm\left(a^{(2)}\right) \otimes a^{(3)} + \sum(-1)^@ \mathfrak{m}\left(a^{(3)}, a_0, a^{(1)}\right) \otimes a^{(2)}
$$ 

$$\text{where } @=\left(\left|a_0\right|'+\left|a^{(1)}\right|^{\prime}+\left|a^{(2)}\right|^{\prime}\right)\left|a^{(3)}\right|^{\prime}$$
\end{definition} 

\begin{definition}
The negative cyclic chain complex of $A$ is defined by
$$
CC_*^{-}(A):=\left( C_*\left(A \right)[[u]], \bee +u \Bee\right)
$$
where $u$ is a formal variable of degree $2$, and $\Bee:C_{*}(A) \to C_{*}(A)$ is the so-called Connes' differential of degree $+1$ (mod $2$) defined by:
$$
\Bee(\bfa)=\sum(-1)^{|a^{(2)}|'\left(\left|a_0\right|'+|a^{(1)}|'\right)} \ \mathbf{1} \otimes a^{(2)}\otimes a_0\otimes a^{(1)}
$$
and satisfies $\Bee^2=\bee ^2=\Bee\bee +\bee \Bee=0$.
The cohomology of the negative cyclic chain complex is called the negative cyclic homology $\operatorname{denoted} H C_*^{-}(A)$.
\end{definition}

Define the degree $0$ bilinear maps $\bee^{1,1}$ and $ 
\Bee^{1,1}: C^*(A) \otimes C_*(A) \rightarrow  C_*(A)$ by:
\begin{equation}
    \bee^{1,1}(\phi , \bfa)=\sum(-1)^{\dagger} \mathfrak{m}_*\left(a^{(3)} \otimes \phi\left(a^{(4)}\right) \otimes a^{(5)} \otimes a_0 \otimes a^{(1)}\right)\otimes a^{(2)} \label{signbee}
\end{equation}
where $\dagger=\left(|a^{(3)}|'+|a^{(4)}|'+|a^{(5)}|'\right)\left(\left|a_0\right|'+|a^{(1)}|'+|a^{(2)}|'\right)+|\phi|^{\prime} |a^{(3)}|'$, and
\begin{equation}
    \Bee^{1,1}(\phi , \bfa)=\sum(-1)^{|\phi|^{\prime} |a^{(2)}|'+\left(|a^{(1)}|'+\left|a_0\right|'\right)\left(|a^{(2)}|'+|a^{(3)}|'+|a^{(4)}|'\right)} \mathbf{1} \otimes a^{(2)}\otimes\phi\left(a^{(3)}\right)\otimes a^{(4)}\otimes a_0\otimes a^{(1)} \label{signBee}
\end{equation}
We use the above operators to recall some homotopy calculus structures on the Hochschild (co)homology  pair $(C^*(A), C_*(A))$ that will be relevant to us.

\begin{definition}
    Contraction $\iota: C^*(A)\otimes C_*^-(A)\to C_*^-(A)$ of degree $0$ is defined by $$\iota\left\{\varphi\right\}(-) := \bee\{\varphi\}(-) + u \Bee\{\varphi\}(-)$$ 
where $\bee\{\varphi\}(-) = \bee^{1,1}(\varphi, -)$ and $\Bee\{\varphi\}(-) = \Bee^{1,1}(\varphi, -)$.
\end{definition}

\begin{definition}\label{def:liederivative}
    Lie derivative, $\mathcal{L}:  C^*(A) \otimes  C_*(A) \rightarrow  C_*(A)$ is defined by 
$$
\begin{aligned}
\mathcal{L}_\phi(\bfa) = \sum &(-1)^{|\phi|^{\prime} \left(|a^{(1)}|'+ |a_0|^{\prime}\right)} a_0 | a^{(1)}\otimes \phi_*\left(a^{(2)}\right) \otimes a^{(3)}\\
&+\sum(-1)^{|a^{(3)}|'\left(|a^{(2)}|'+|a^{(1)}|'+\left|a_0\right|^{\prime}\right)} \phi_*\left(a^{(3)} \otimes a_0 \otimes a^{(1)}\right) | a^{(2)} \\
\end{aligned}
$$
\end{definition} 

For any unreduced cochain $\phi \in \tilde{C}^*(A)$, Definition \ref{def:liederivative} constructs the Lie derivative operator $\mathcal{L}_\phi \in \mathrm{End}(C_*(A))$. Although the full $A_\infty$-structure $\mathfrak{m} = \prod_{n \ge 1} \mathfrak{m}_n$ lies in the unreduced cochain complex $\tilde{C}^1(A)$ rather than the reduced cochain complex $C^*(A)$, it nevertheless induces a well-defined operator $\mathcal{L}_{\mathfrak{m}} \in \mathrm{End}(C_*(A))$. By strict unitality, this operator coincides exactly with the Hochschild boundary operator $\mathfrak{b}$ on reduced chains: $$\mathcal{L}_{\mathfrak{m}} = \mathfrak{b} \quad \text{in } \mathrm{End}(C_*(A)).$$
The Lie derivative $\mathcal{L}$ and contraction $\iota$ are related by the Cartan homotopy formula
\begin{equation}
    [\bee+u \Bee, \iota\{\varphi\}]=-u \cdot \mathcal{L}_{\varphi}-\iota\{[\emm, \varphi]\}\label{cartanmagicform}
\end{equation}
See Getzler \parencite[Section 2]{getzler1993cartan} for more related properties.

The Hochschild cochain complex $ C^{*}(A)$ acts on $ C_{*}(A)$ via the cap product $$\varphi \cap(-):=(-1)^{|\varphi|} \bee\{\varphi\}(-)$$ and defines a left module structure, with respect to the cup product $\cup$ on $H H^{*}(A)$ defined on the chain-level by
$$\phi \cup \psi = \sum (-1)^{|\phi|+|\psi|'(|a^{(1)}|'+|a^{(2)}|'+|a^{(3)}|')+|\phi|'|a^{(1)}|'} \ \emm_*(a^{(1)},\phi(a^{(2)}),a^{(3)}, \psi(a^{(4)}), a^{(5)})$$
for cochains $\phi,\psi\in C^*(A)$.

\subsection{Mukai and higher residue pairings} We recall that an $A_\infty$ algebra $A$ is called smooth (or homologically smooth) if the diagonal bimodule $A_\Delta$ is perfect, that is, split-generated by tensor products of Yoneda modules (see \parencite[Definition 8.1.2]{kontsevich2006notes}). Also, we say that an $A_\infty$ algebra $A$ is proper if the cohomology $H^*(A, \emm_1)$ is finite-dimensional.

Assume $A$ is proper. There is a degree-zero pairing known as the Mukai pairing naturally defined on Hochschild homology:
\[
\langle -,-\rangle_{Muk}: HH_*(A) \otimes HH_*(A) \to \mathrm{k}
\]
Shklyarov \parencite[Theorem 4]{shklyarov2013hirzebruch} proves that this pairing is non-degenerate  if $A$ (as a DG algebra) is in addition smooth. When $A$ is finite dimensional (which can be achieved after passing to the minimal model of $A$), Sheridan \parencite[Proposition 5.22]{sheridan2020formulae} shows that the Mukai-pairing can be described at the chain level in the following way: for $\bfa = a_0|a_1,...,a_n$ and $\bfb =b_0|b_1,...,b_m$, 
\begin{align}
    \langle \bfa, \bfb \rangle_{Muk}= \sum tr[c \mapsto (-1)^{\#} \emm_{*}\left(a^{(3)}, a_0, a^{(1)} \otimes \emm_{*}(a^{(2)}, c, b^{(3)}, b_0, b^{(1)})\otimes b^{(2)}\right)] \label{eq:Mukaipairing}
\end{align}
where  $\#= 1+ |c||\bfb|' + (|a^{(3)}|' +|a_0|' + |a^{(1)}|') + ( |a_0|' + |a^{(1)}|' + |a^{(2)}|')|a^{(3)}|' + ( |b_0|' + |b^{(1)}|' + |b^{(2)}|')|b^{(3)}|'
$.

Assuming $A$ is finite dimensional, we can define the so-called {\it higher residue pairing} on $H C_{*}^{-}(A)$ \parencite[Proposition 5.36]{sheridan2020formulae}:
\begin{equation}
   \langle-,-\rangle_{\text {hres}}:  CC_{*}^{-}(A) \otimes CC_{*}^{-}(A) \rightarrow \mathrm{k}[[u]], \label{defnhres} 
\end{equation}
by extending the chain-level Mukai pairing (\ref{eq:Mukaipairing}) sesquilinearly. That is $\langle u \bfa, \bfb \rangle_{\text {hres }}=-\langle\bfa, u \bfb \rangle_{\text {hres }}=u\langle\bfa, \bfb\rangle_{\text {Muk }}$, for Hochschild chains $\bfa, \bfb$. It commutes with the cyclic differential $\bee+u\Bee$ and so descends to the negative cyclic homology. \\

\subsection{The u-connection on negative cyclic homology}
Let $\mathfrak{m}^{\prime}$ be the cocycle in $C^{*}(A)$ defined as $\mathfrak{m}^{\prime}:=\prod_{k \geq 0}(2-k) \mathfrak{m}_k$, and $\Gamma$ the length operator defined by $\Gamma\left(a_0| a_1,\ldots, a_n\right):=-n \cdot a_0| a_1,\ldots, a_n$.
The meromorphic $u$-connection $\nabla_{\frac{d}{d u}}$ on $HC_*^-(A)$ was first introduced by Kontsevich-Soibelman and Shklyarov (see \parencite[Section 3.1]{cualduararu2021categorical} \parencite[Section 2.2.5]{katzarvovKonPan} \parencite[Sections 11.4-11.5]{kontsevich2006notes} \parencite[Sections 1-2]{shklyarov2016matrix}).
\begin{definition}
The $u$-connection $\nabla_{\frac{d}{d u}}: H C_{*}^{-}(A) \rightarrow u^{-2} H C_{*}^{-}(A)$ is defined by the formula 
\begin{equation}
    \nabla_{\frac{d}{d u}}:=\frac{d}{d u}+\frac{\Gamma}{2 u}+\frac{\iota\left\{\mathfrak{m}^{\prime}\right\}}{2 u^2} \label{defnconnection}
\end{equation}
\end{definition}
Amorim-Tu \parencite[Proposition 2.6]{amorim2022categorical} showed the higher residue pairing is parallel with respect to the $u$ connection. That is,
$$
\frac{d}{d u}\langle\bfa, \bfb \rangle_{\text {hres }}=\left\langle\nabla_{\frac{d}{d u}} \bfa, \bfb\right\rangle_{\text {hres }}-\left\langle\bfa, \nabla_{\frac{d}{d u}} \bfb\right\rangle_{\text {hres }}
$$
The DG version was showed earlier by Shklyarov \parencite[Proposition 3.7]{shklyarov2016matrix}.

\subsection{Formal EP-structure}\label{section-intersection}
The negative cyclic homology of a smooth and proper $A_\infty$-algebra $A$ carries a non-commutative variation of semi-infinite Hodge structure, consisting of its $u$-adic filtration together with the $u$-connection and higher residue pairing. These structures provide a non-commutative analogue of the Hodge-theoretic data appearing in classical geometry. In the context of mirror symmetry, they allow algebraic invariants to be compared with corresponding Hodge-theoretic structures. We use the notion of a formal EP-structure to package the pairing and polarization conditions on these data in a convenient formal framework. This is the formal analogue of the TERP-structure introduced by Hertling \parencite[Section 2.5]{Hertling}. We recall the definition in the form needed here.

\begin{definition}
A formal EP-structure

consists of a $\mathbb{Z}/2\mathbb{Z}$-graded $\mathrm{k}[[u]]$-module $\mathcal{E}$ that is free and of finite rank; a meromorphic connection of even degree $$\nabla_{\frac{d}{d u}}: \mathcal{E}\to u^{-2} \mathcal{E}$$ with at most a second order pole at $u=0$; and an even $\mathrm{k}$-linear pairing $\langle-,-\rangle : \mathcal{E} \otimes \mathcal{E} \to \mathrm{k}[[u]]$ subject to the following conditions: for $s_1,s_2\in \mathcal{E}$ 
\begin{itemize}
    \item the pairing $\langle-,-\rangle $ is $\mathrm{k}[[u]]$-sesquilinear, that is: \[
    \langle f\cdot s_1, g\cdot s_2\rangle = fg^*\langle s_1,s_2\rangle ; \ \ f,g\in \mathrm{k}[[u]], \text{ and } g^*(u)=g(-u) 
    \]
    \item the pairing is covariantly constant with respect to the $u$-connection, that is: \[
    \frac{d}{d u}\langle s_1, s_2 \rangle = \langle \nabla_{\frac{d}{d u}}s_1, s_2 \rangle - \langle s_1, \nabla_{\frac{d}{d u}} s_2 \rangle
    \]
    \item the pairing is symmetric in the sense that 
    \[ \langle s_1, s_2 \rangle = (-1)^{|s_1||s_2|} \langle s_2, s_1 \rangle^* \]
    \item a polarization: that is, the induced pairing $\mathcal{E}/u\mathcal{E} \otimes \mathcal{E}/u\mathcal{E} \to \mathrm{k}$ is non-degenerate.
\end{itemize}
\end{definition}

\begin{definition}
    Let $(\mathcal{E}_1, \nabla_{d/du}^1, \langle -, - \rangle_1)$ and $(\mathcal{E}_2, \nabla_{d/du}^2, \langle -, - \rangle_2)$ be formal EP-structures over $k[[u]]$. A $k[[u]]$-linear map $f: \mathcal{E}_1 \to \mathcal{E}_2$ is a morphism of formal EP-structures if it satisfies \begin{itemize}
        \item connection compatibility: $f \left( \nabla^1_{\frac{d}{du}} (s) \right) = \nabla^2_{\frac{d}{du}} \left( f(s) \right)$ for all $s \in \mathcal{E}_1$, and
        \item pairing compatibility: $\langle f(s_1), f(s_2) \rangle_2 = \langle s_1, s_2 \rangle_1$ for all $s_1, s_2 \in \mathcal{E}_1$.
    \end{itemize}
    An isomorphism of formal EP-structures is an invertible morphism whose inverse is also a morphism.
\end{definition}

Thus, in the terminology of Sheridan \parencite[Section 2]{sheridan2020formulae}, the formal EP-structure above is the structure we will refer to as a {\it polarized semi-infinite Hodge structure}. In particular, throughout this paper, the phrase ``polarized semi-infinite Hodge structure'' refers to the formal EP-structure defined above.
We may omit the term {\it formal} and refer to it as an EP-structure.\\

The $u$-adic filtration on the negative cyclic homology $C_*(A)$ gives rise to a spectral sequence, termed the {\it non-commutative Hodge-to-de Rham spectral sequence}: $$E_1 \cong HH_{*}(A)[[u]] \implies HC_{*}^-(A)$$
where the $E_1$-page consists of the Hochschild homology $HH_*(A)$ (analogous to classical Hodge cohomology $\bigoplus_q H^{p}(-, \Omega_{-}^q)$) and converges to the cyclic homology $HC_*(A)$ (analogous to de Rham cohomology $H_{\mathrm{dR}}^*(-)$). A fundamental conjecture of non-commutative Hodge theory, as formulated by Kontsevich-Soibelman \parencite[Conjecture 9.1.2]{kontsevich2006notes} - the so-called Hodge-to-de-Rham degeneration property in this case is: 
\begin{conjecture}\label{conjecture}
    Let $A$ be $\mathbb{Z}/2\mathbb{Z}$-graded smooth and proper $A_\infty$ algebra. Then, the Hodge-to-de-Rham spectral sequence 
    degenerates at the $E_1$-page. So we have the isomorphism $HH_*(A)[[u]] \cong HC_*^-(A)$
\end{conjecture}

The conjecture has been proved by Kaledin in the case that $A$ is $\mathbb{Z}$-graded (see \parencite[Theorem 5.4]{kaledin2017spectral}). In the general case, to the knowledge of the author, the conjecture remains open if $A$ is, for example, $\mathbb{Z}/2\mathbb{Z}$-graded. 
\begin{theorem}[Cf. {\parencite[Theorem A]{sheridan2020formulae}}]
    Let $A$ be $\mathbb{Z}/2\mathbb{Z}$-graded smooth, proper and weak Calabi-Yau $A_\infty$ algebra such that the Hodge-to-de-Rham spectral sequence degenerates. Then the triple 
    $$\left(HC_*^-(A), \nabla_{\frac{d}{d u}}, \langle-,-\rangle_{\text {hres}}\right)$$ is an EP-structure.
\end{theorem}

\section{Tensor product of A-infinity algebras} \label{sec-tensorprod}


\subsection{A combinatorial (tree) description of the tensor product}
Classical constructions of the tensor product of general $A_\infty$- algebras were carried out by Saneblidze and Umble \cite{saneblidze2000diagonal, saneblidze2004diagonals}, and Markl and Shnider \cite{markl2006associahedra} in the language of operads using cellular diagonals on Stasheff associahedra. Amorim’s construction  \cite{amorim2016tensor} reformulates the operadic diagonal of Markl and Shnider into an explicit, combinatorial tree sum over the poset of binary trees $G_n^{\text{bin}}$, where higher-order terms are obtained by grafting corollas and collapsing right-leaning internal edges. While Amorim's original framework was designed for filtered $A_\infty$-algebras over Novikov rings, it translates naturally to classical, unfiltered $A_\infty$-algebras over a field $\mathrm{k}$ of characteristic zero.  In this section, we adopt the formulation due to Amorim \cite{amorim2016tensor}. 

First, we recall the tree conventions and notation. A planar {\it tree} $T=(V(T), E(T))$ will denote a finite connected graph without cycles embedded in the plane, and equipped with a distinguished set of vertices with a unique incident edge. Exactly one of these distinguished vertices is designated as the {\it root}, while the others will be called {\it leaves}. The set of internal vertices, that is, all vertices excluding the root and leaf vertices, will be denoted by $V(T)$. The set $E(T)$ will denote the set of all edges of $T$. An edge $e \in E(T)$ is called an {\it internal edge} if it is not incident to the root or any leaf vertex. We denote the set of internal edges by $E_{int}(T) \subset E(T)$. The edges in $E(T) \setminus E_{int}(T)$ are called {\it external edges} (or {\it legs}). An edge incident to a leaf vertex is called a {\it leaf edge} (an input leg), and the unique edge incident to the root vertex is the {\it root edge} (the output leg). For any internal vertex $v \in V(T)$, there exists a unique path in $T$ from $v$ to the root. The unique edge incident to $v$ along this path is the outgoing edge at $v$; all other edges incident to $v$ are incoming edges. The {\it valency} $val(v)$ of an internal vertex $v$ is defined specifically as the number of incoming edges in $v$. Each internal vertex $v \in V(T)$ will correspond to an operation $\emm_{val(v)}$ in the $A_\infty$-algebra. Let $G_n$ denote the set of planar rooted trees with $n$ leaves such that the valency of incoming edges $\operatorname{val}(v) \geqslant 2$ for all vertices $v \in V(T)$. Given a vertex $v$ we will denote by $v(i)$ the $i^{t h}$ incident edge at $v$. 

For each pair $x, y \in V(T) \cup E(T)$ we say $x$ is under $y$ if the unique path $P_x$ in $T$ from $x$ to the root of $T$ contains $y$. Note that this defines a partial ordering. Suppose $x$ and $y$ are not comparable and denote by $w$ the first (furthest from the root) vertex that $P_x$ and $P_y$ have in common. Then $P_x$ and $P_y$ contain edges $w(i)$ and $w(j)$ respectively. We say that $x$ is to the left of $y$ if $i<j$. 
Given $x \in V(T) \cup E(T)$ we denote by $p_x$ (respectively $q_x$) the left (respectively right) most leaf of $T$ under $x$. Additionally, given $e \in E(T)$ we define integers $J_e, I_e$ and $\tilde{I}_e$ to be zero when $e$ is a leaf edge, and when $e$ is not a leaf edge, $J_e$ is the number of leaves under $e, \tilde{I}_e$ is the number of internal edges under $e$ and $I_e=\tilde{I}_e+1$. We will let $b(T)$ and $nb(T)$ denote the number of binary and non-binary vertices of $T$, respectively.

Let $G_n$ denote the set of planar trees with $n$ leaves, and $G_n^{b i n}$ denote the subset of $G_n$ consisting of trees $T$ with $\operatorname{val}(v)=2$ for all $v \in V(T)$. Elements of $G_n^{b i n}$ are called binary trees. There is a partial order on the set $G_n^{b i n}$, which was introduced in \parencite[Figure 1]{markl2006associahedra}. It corresponds to a local operation within a single vertex that switches an internal edge: specifically, losing a right-leaning internal edge while gaining a left-leaning internal edge as shown in Figure (\ref{figordertree}). The partial order on binary trees is the transitive closure of this relation in the sense that one tree precedes another if and only if it can be reached through a finite sequence of such elementary edge-switching moves. 

\begin{figure}[h]
\begin{center}
 \setlength{\unitlength}{3pt}
\begin{picture}(60,20)(0,0)

\linethickness{0.3mm}
\put(10,10){\line(0,1){10}}
\linethickness{0.3mm}
\put(0,0){\line(1,1){10}}
\linethickness{0.3mm}
\put(20,0){\line(-1,1){10}}
\linethickness{0.3mm}
\put(10,0){\line(-1,1){5}}

\linethickness{0.3mm}
\put(50,10){\line(0,1){10}}
\linethickness{0.3mm}
\put(40,0){\line(1,1){10}}
\linethickness{0.3mm}
\put(60,0){\line(-1,1){10}}
\linethickness{0.3mm}
\put(50,0){\line(1,1){5}}

\put(30,10){\makebox(0,0)[cc]{$\leq$}}

\end{picture}
\end{center}
\caption{Generator relation for the partial order on binary trees}
\label{figordertree}
\end{figure}
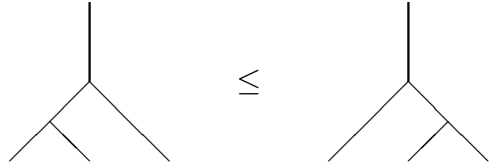
There is an absolute minimum and maximum for this order. The minimum (or maximum), which we denote by $\underline{M}_n$ (or  $\overline{M}_n)$, is the binary tree with strictly right- (or left-) leaning internal edges. Given $T \in G_n$ we define $T_{\max } \in G_n^{b i n}$ as the maximal (with respect to the order above) binary tree that resolves $T$. That is, $T$ can be obtained by collapsing several edges in $T_{\max }$. An example is shown in Figure \ref{figmaxtree}. Similarly, we define $T_{\min }$ as the minimal binary tree that resolves $T$. The tree in $G_n$ with exactly one internal vertex, called the $n$-corolla, is denote by $c_n$.

We introduce two subsets $L_n, R_n\subset G_n$. Let $L_n$ denote the subset of $G_n$ obtained by grafting corollas in all possible ways while avoiding the last leaf on any vertex. Similarly, Let $R_n$ denote the subset of $G_n$ obtained by grafting corollas in all possible ways while avoiding the first leaf on any vertex. More formally, $T \in L_n$ if either $T=c_n$ or there exist integers $1 \leqslant i_1<\ldots<i_l<m$ and trees $T_j \in L_{n_j}$ such that
$$
T=c_m \circ_{i_1, \ldots, i_l}\left(T_1, \ldots, T_l\right),
$$

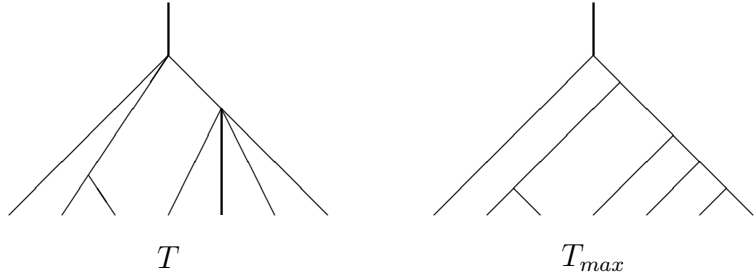
\begin{figure}[h]
\begin{center}
\setlength{\unitlength}{2pt}
\begin{picture}(140,50)(0,0)

\linethickness{0.3mm}
\put(30,40){\line(0,1){10}}
\linethickness{0.3mm}
\put(0,10){\line(1,1){30}}
\linethickness{0.3mm}
\put(10,10){\line(2,3){20}}
\linethickness{0.3mm}
\put(20,10){\line(-2,3){5}}
\linethickness{0.3mm}
\put(30,10){\line(1,2){10}}
\linethickness{0.3mm}
\put(40,10){\line(0,1){20}}
\linethickness{0.3mm}
\put(50,10){\line(-1,2){10}}
\linethickness{0.3mm}
\put(60,10){\line(-1,1){30}}

\linethickness{0.3mm}
\put(110,40){\line(0,1){10}}
\linethickness{0.3mm}
\put(80,10){\line(1,1){30}}
\linethickness{0.3mm}
\put(90,10){\line(1,1){25}}
\linethickness{0.3mm}
\put(100,10){\line(-1,1){5}}
\linethickness{0.3mm}
\put(110,10){\line(1,1){15}}
\linethickness{0.3mm}
\put(120,10){\line(1,1){10}}
\linethickness{0.3mm}
\put(130,10){\line(1,1){5}}
\linethickness{0.3mm}
\put(140,10){\line(-1,1){30}}

\put(30,2){\makebox(0,0)[cc]{$T$}}
\put(110,2){\makebox(0,0)[cc]{$T_{max}$}}

\end{picture}
\end{center}
\caption{A maximal tree resolving $T$}
\label{figmaxtree}
\end{figure}
where $\circ_{i_1, \ldots, i_l}$ tells us to graft the root of $T_j$ to the leaf $i_j$ of $c_m$.

Now suppose that $T \in L_n$ has $k$ internal edges. 
Let $R(S)$ denote the set of right-leaning internal edges of a binary tree $S$ and let $|R(S)|$ denote its order. We define
$$
\alpha(T):=\sum_{\substack{S \in G_n^{b i n} \\ S \geqslant T_{\text {max }} \\|R(S)|=k}} S / R(S),
$$
where $S / R(S)$ is the tree obtained by collapsing all the edges of $S$ in $R(S)$. The sum taken here lives in $\mathrm{k} G_n$, the free $\mathrm{k}$-vector space spanned by the set $G_n$ of planar trees with $n$ leaves. Note that $\alpha(T)$ is non-zero as it always includes $\alpha_0(T):=T_{\max } / R\left(T_{\max }\right)$. For an example see Figure \ref{figexalpha} below. \\

\begin{figure}[h]
\begin{center}
 \setlength{\unitlength}{1.8pt}
\begin{picture}(130,30)(0,0)

\put(3,10){\makebox(0,0)[cc]{$\alpha\Bigg($}}

\linethickness{0.3mm}
\put(25,20){\line(0,1){10}}
\linethickness{0.3mm}
\put(10,0){\line(3,4){15}}
\linethickness{0.3mm}
\put(20,0){\line(0,1){13,3}}
\linethickness{0.3mm}
\put(30,0){\line(-3,4){10}}
\linethickness{0.3mm}
\put(40,0){\line(-3,4){15}}
\linethickness{0.3mm}

\put(45,10){\makebox(0,0)[cc]{$\Bigg)$}}
\put(55,10){\makebox(0,0)[cc]{$=$}}

\linethickness{0.3mm}
\put(75,10){\line(0,1){20}}
\linethickness{0.3mm}
\put(60,0){\line(3,4){15}}
\linethickness{0.3mm}
\put(70,0){\line(1,2){5}}
\linethickness{0.3mm}
\put(80,0){\line(-1,2){5}}
\linethickness{0.3mm}
\put(90,0){\line(-3,4){15}}
\linethickness{0.3mm}

\put(95,10){\makebox(0,0)[cc]{$+$}}

\linethickness{0.3mm}
\put(115,20){\line(0,1){10}}
\linethickness{0.3mm}
\put(100,0){\line(3,4){15}}
\linethickness{0.3mm}
\put(110,0){\line(3,4){10}}
\linethickness{0.3mm}
\put(120,0){\line(0,1){13,3}}
\linethickness{0.3mm}
\put(130,0){\line(-3,4){15}}
\linethickness{0.3mm}

\put(130,10){\makebox(0,0)[cc]{$.$}}
\end{picture}
\end{center}
\caption{Example of $\alpha$} 
\label{figexalpha}
\end{figure}

For an $A_\infty$-algebra $A$ and a tree $T\in G_n$, we define a map $T^A: A[1]^{\otimes n}\to A[1]$ by assigning to each vertex $v \in V(T)$ the map $\mathfrak{m}_{v a l(v)}^A$.

\begin{theorem}[{\parencite[Theorem 4.1]{amorim2016tensor}}]
    
 Let $\AAA$ and $\BB$ be unital $A_{\infty}$-algebras.  Then the tensor product $\AAA \otimes_{\infty} \BB$ is quasi-isomorphic to the $A_{\infty}$-algebra $\left(\AAA \otimes \BB, \mathfrak{m}^{\otimes}\right)$ with operations given by 
$$
\begin{aligned}
\mathfrak{m}_{1}^{\otimes} &= \mathfrak{m}_{1}^\AAA \otimes id + id \otimes \mathfrak{m}_{1}^\BB \\
\mathfrak{m}_{n}^{\otimes}\left(a_1 \otimes b_1, \ldots, a_n \otimes b_n\right) &=\sum_{T \in L_n}(-1)^\epsilon T^A\left(a_1, \ldots a_n\right) \otimes \alpha(T)^B\left(b_1, \ldots, b_n\right),
\end{aligned}
$$
extended linearly.

The sign exponent is defined as $\epsilon=\theta(T)+\triangleleft+|a|\left|E_{\text {int }}(T)\right|+\gamma_{a}+\gamma_{b}$, with
\begin{align}
\triangleleft &=\sum_{p>q}|b_q||a_p|\notag\\
\gamma_{a}&=\sum_{v \in V(T)} \sum_{i<p_v}\left|a_i\right|, \quad \gamma_{b}=\sum_{v \in V(\alpha(T))} \sum_{i<p_v}\left|b_i\right| \text { and } \notag\\
\theta(T)&=\sum_{v \in V(T)}\left(\sum_i\left((\operatorname{val}(v)-i-1) I_{v(i)}+\tilde{I}_{v(i)}+J_{v(i)}\right)+\sum_{i<j} I_{v(i)}\left(\tilde{I}_{v(j)}+J_{v(j)}\right)\right) \label{signoftrees}
\end{align}

\end{theorem}

We give an illustration of the theorem with explicit formulas for $\mathfrak{m}_3^{\otimes}$, and $\mathfrak{m}_4^{\otimes}$ in Figure \ref{figm3m4} below with signs determined by $\theta(T)$. \\

\begin{figure}[h]
\begin{center}
 \setlength{\unitlength}{2.1pt}
\begin{picture}(130,20)(0,0)
\put(0,15){\makebox(0,0)[cc]{$\emm_3^{\otimes}=$}}

\linethickness{0.3mm}
\put(20,10){\line(0,1){10}}
\linethickness{0.3mm}
\put(10,0){\line(1,1){10}}
\linethickness{0.3mm}
\put(30,0){\line(-1,1){10}}
\linethickness{0.3mm}
\put(20,0){\line(-1,1){5}}

\put(35,10){\makebox(0,0)[cc]{$\otimes$}}

\linethickness{0.3mm}
\put(50,0){\line(0,1){20}}
\linethickness{0.3mm}
\put(40,0){\line(1,1){10}}
\linethickness{0.3mm}
\put(60,0){\line(-1,1){10}}

\put(70,10){\makebox(0,0)[cc]{$+$}}

\linethickness{0.3mm}
\put(90,0){\line(0,1){20}}
\linethickness{0.3mm}
\put(80,0){\line(1,1){10}}
\linethickness{0.3mm}
\put(100,0){\line(-1,1){10}}

\put(105,10){\makebox(0,0)[cc]{$\otimes$}}

\linethickness{0.3mm}
\put(120,10){\line(0,1){10}}
\linethickness{0.3mm}
\put(110,0){\line(1,1){10}}
\linethickness{0.3mm}
\put(130,0){\line(-1,1){10}}
\linethickness{0.3mm}
\put(120,0){\line(1,1){5}}

\put(130,10){\makebox(0,0)[cc]{$,$}}

\end{picture}
\vspace{0.3cm}

 \setlength{\unitlength}{1.5pt}
\begin{picture}(270,70)(0,0)
\put(0,55){\makebox(0,0)[cc]{$\emm_4^{\otimes}=$}}

\linethickness{0.3mm}
\put(25,60){\line(0,1){10}}
\linethickness{0.3mm}
\put(10,40){\line(3,4){15}}
\linethickness{0.3mm}
\put(40,40){\line(-3,4){15}}
\linethickness{0.3mm}
\put(20,40){\line(-3,4){5}}
\linethickness{0.3mm}
\put(30,40){\line(-3,4){10}}

\put(45,55){\makebox(0,0)[cc]{$\otimes$}}

\linethickness{0.3mm}
\put(65,60){\line(0,1){10}}
\linethickness{0.3mm}
\put(50,40){\line(3,4){15}}
\linethickness{0.3mm}
\put(80,40){\line(-3,4){15}}
\linethickness{0.3mm}
\put(60,40){\line(1,4){5}}
\linethickness{0.3mm}
\put(70,40){\line(-1,4){5}}

\put(90,55){\makebox(0,0)[cc]{$-$}}

\linethickness{0.3mm}
\put(115,60){\line(0,1){10}}
\linethickness{0.3mm}
\put(100,40){\line(3,4){15}}
\linethickness{0.3mm}
\put(130,40){\line(-3,4){15}}
\linethickness{0.3mm}
\put(110,40){\line(0,1){13,33}}
\linethickness{0.3mm}
\put(120,40){\line(-3,4){10}}

\put(135,55){\makebox(0,0)[cc]{$\otimes$}}

\linethickness{0.3mm}
\put(155,50){\line(0,1){20}}
\linethickness{0.3mm}
\put(140,40){\line(3,4){15}}
\linethickness{0.3mm}
\put(170,40){\line(-3,4){15}}
\linethickness{0.3mm}
\put(150,40){\line(1,2){5}}
\linethickness{0.3mm}
\put(160,40){\line(-1,2){5}}

\put(180,55){\makebox(0,0)[cc]{$-$}}

\linethickness{0.3mm}
\put(205,60){\line(0,1){10}}
\linethickness{0.3mm}
\put(190,40){\line(3,4){15}}
\linethickness{0.3mm}
\put(220,40){\line(-3,4){15}}
\linethickness{0.3mm}
\put(200,40){\line(0,1){13,33}}
\linethickness{0.3mm}
\put(210,40){\line(-3,4){10}}

\put(225,55){\makebox(0,0)[cc]{$\otimes$}}

\linethickness{0.3mm}
\put(245,60){\line(0,1){10}}
\linethickness{0.3mm}
\put(230,40){\line(3,4){15}}
\linethickness{0.3mm}
\put(260,40){\line(-3,4){15}}
\linethickness{0.3mm}
\put(240,40){\line(3,4){10}}
\linethickness{0.3mm}
\put(250,40){\line(0,1){13,33}}


\put(0,15){\makebox(0,0)[cc]{$+$}}

\linethickness{0.3mm}
\put(25,10){\line(0,1){20}}
\linethickness{0.3mm}
\put(10,0){\line(3,4){15}}
\linethickness{0.3mm}
\put(40,0){\line(-3,4){15}}
\linethickness{0.3mm}
\put(20,0){\line(1,2){5}}
\linethickness{0.3mm}
\put(30,0){\line(-1,2){5}}

\put(45,15){\makebox(0,0)[cc]{$\otimes$}}

\linethickness{0.3mm}
\put(65,20){\line(0,1){10}}
\linethickness{0.3mm}
\put(50,0){\line(3,4){15}}
\linethickness{0.3mm}
\put(80,0){\line(-3,4){15}}
\linethickness{0.3mm}
\put(60,0){\line(3,4){10}}
\linethickness{0.3mm}
\put(70,0){\line(0,1){13,33}}

\put(90,15){\makebox(0,0)[cc]{$-$}}

\linethickness{0.3mm}
\put(115,20){\line(0,1){10}}
\linethickness{0.3mm}
\put(100,0){\line(3,4){15}}
\linethickness{0.3mm}
\put(130,0){\line(-3,4){15}}
\linethickness{0.3mm}
\put(110,0){\line(-3,4){5}}
\linethickness{0.3mm}
\put(120,0){\line(-1,4){5}}

\put(135,15){\makebox(0,0)[cc]{$\otimes$}}

\linethickness{0.3mm}
\put(155,20){\line(0,1){10}}
\linethickness{0.3mm}
\put(140,0){\line(3,4){15}}
\linethickness{0.3mm}
\put(170,0){\line(-3,4){15}}
\linethickness{0.3mm}
\put(150,0){\line(1,4){5}}
\linethickness{0.3mm}
\put(160,0){\line(3,4){5}}

\put(180,15){\makebox(0,0)[cc]{$+$}}

\linethickness{0.3mm}
\put(205,20){\line(0,1){10}}
\linethickness{0.3mm}
\put(190,0){\line(3,4){15}}
\linethickness{0.3mm}
\put(220,0){\line(-3,4){15}}
\linethickness{0.3mm}
\put(200,0){\line(1,4){5}}
\linethickness{0.3mm}
\put(210,0){\line(-1,4){5}}

\put(225,15){\makebox(0,0)[cc]{$\otimes$}}

\linethickness{0.3mm}
\put(245,20){\line(0,1){10}}
\linethickness{0.3mm}
\put(230,0){\line(3,4){15}}
\linethickness{0.3mm}
\put(260,0){\line(-3,4){15}}
\linethickness{0.3mm}
\put(240,0){\line(3,4){10}}
\linethickness{0.3mm}
\put(250,0){\line(3,4){5}}

\put(270,15){\makebox(0,0)[cc]{$.$}}

\end{picture}
\end{center}
\caption{Formulas for $\emm_3^{\otimes}$ and $\emm_4^{\otimes}$.}
\label{figm3m4}
\end{figure}

\begin{proposition}\label{prop:18UW}
 There is the following alternative description when $n \geqslant 3$
$$
\mathfrak{m}_{n}^{\otimes}=\sum_{\substack{\left|V(U)\right|+\left|V(W)\right|=n \\ U_{\max } \leqslant W_{\text {min }}}} \pm U^A \otimes W^B
$$ 
where $U\in L_n$, and $W$ is a term of $\alpha(U)\in R_n$.
\end{proposition}

Proof. See \parencite[Proposition 4.3]{amorim2016tensor}. \\

\noindent

The following lemma will be used to prove a subsequent result, so we state it here.

\begin{lemma}[{\parencite[Lemma 6.1]{amorim2016tensor}}]{\label{treelemma}} Consider a tree $U \in L_n,$ and $a_1, \ldots a_n \in \AAA$. Let $j_a$ be the number of $a_i$ such that $a_i=\mathbf{1}_\AAA$ and denote by $b(U)$ the number of vertices of $U$ of valency 2 . If $U^A\left(a_1, \ldots, a_n\right) \neq 0$, then $b(U) \geqslant j_a$ unless $j_a=n$, in which case $b(U)=n-1$. Moreover the same is true if $U \in R_n$.
\end{lemma}

\subsection{Properties of the tensor product}
\begin{definition}
Let $\left(\AAA, \mathfrak{m}^\AAA\right)$ and $(C, \mu)$ be strictly unital $A_{\infty}$-algebras, we say $\AAA$ is a subalgebra of $C$ if $\AAA \subseteq C$, $\ua=\mathbf{1}_C$ and for all $k>0$ and $a_1, \ldots a_k \in A$ we have
$$
\begin{aligned}
& \mu_{k}\left(a_1, \ldots, a_k\right)=\mathfrak{m}_{k}^\AAA\left(a_1, \ldots, a_k\right)
\end{aligned}
$$ 
\end{definition}

\begin{definition}{\label{commutingsubalgebradefn}}
Let $\left(\AAA, \mathfrak{m}^\AAA\right)$ and $\left(\BB, \mathfrak{m}^\BB\right)$ be $A_{\infty}$-algebras. Suppose $\AAA$ and $\BB$ are $A_\infty$-subalgebras of $(C, \mu)$. Define a map $K: \AAA \otimes \BB \longrightarrow C$ by $K(a \otimes b)=(-1)^{|a|} \mu_{2}(a, b)$. We say $\AAA$ and $\BB$ are {\it commuting subalgebras} if given $c=K(a \otimes b)$ and $c_1, \ldots, c_k \in C$ such that for each $i$, $c_i=a_i$ or $c_i=b_i$ for some $a_i \in \AAA$ and $b_i \in \BB$, the following conditions hold. \\

\noindent
(a) For $k>0,$ $ \mu_{k}\left(c_1, \ldots, c_k\right)=0$ unless
\begin{itemize}
    \item $k=2$ and $c_1 \in \AAA, c_2 \in \BB$ (or vice-versa) in which case, \begin{equation}
        \mu_{2}\left(c_1, c_2\right)+(-1)^{|c_1|'|c_2|'} \mu_{2}\left(c_2, c_1\right)=0 \label{eq:m2comsub}
    \end{equation}
    \item  $c_i=a_i$ for all $i$
    \item  $c_i=b_i$ for all $i$,
\end{itemize}

\noindent
(b) $\mu_{k+1}\left(c_1, \ldots, c_i, c, c_{i+1}, \ldots, c_k\right)=0$ unless
\begin{itemize}
    \item  $c_i=a_i$ for all $i$, in which case it equals
\begin{equation}
    (-1)^{|b| \sum_{j>i}|a_j|'} K\left(\mathfrak{m}_{k+1}^\AAA\left(a_1 \ldots a_i, a, \ldots, a_k\right) \otimes b\right) \label{eq:mkcomsuba}
\end{equation}
    \item $c_i=b_i$ for all $i$, in which case it equals
\begin{equation}
    (-1)^{|a|\left(\sum_{j \leqslant i}|b_j|'+1\right)} K\left(a \otimes \mathfrak{m}_{k+1}^\BB\left(b_1 \ldots b_i, b, \ldots, b_k\right)\right) \label{eq:mkcomsubb}
\end{equation}
\end{itemize}
\end{definition}

A more conceptual proof of the following proposition for filtered $A_\infty$ algebras can be found in \parencite[Proposition 6.9]{amorim2016tensor}. For completeness, we give a more combinatorial proof using the planar trees introduced above.

\begin{proposition}\label{prop18}
 The $A_\infty$ algebras $(\AAA, \mathfrak{m}^\AAA)$ and $(\BB, \mathfrak{m}^\BB)$ are commuting subalgebras of  $\AAA\otimes \BB$ under the inclusions \begin{align*}
    \AAA \hookrightarrow  \AAA \otimes \BB \quad \qquad & \BB \hookrightarrow  \AAA\otimes \BB \\
    a \mapsto a\otimes \mathbf{1}_\BB \quad \qquad & b \mapsto \mathbf{1}_\AAA \otimes b
\end{align*}

\end{proposition}
\begin{proof}
   First, we need to check that $\AAA$ is a subalgebra. That is, for each $n>0$ and $a_1, \ldots , a_n \in \AAA,$ we need to show that $\mathfrak{m}_n^\otimes (a_1 \otimes \mathbf{1}_\BB, \ldots ,a_n \otimes \mathbf{1}_\BB ) = \mathfrak{m}_n^\AAA (a_1, \ldots , a_n) \otimes \ub.$ \\
By Proposition \ref{prop:18UW}, 
\[
\mathfrak{m}_n^\otimes (a_1 \otimes \mathbf{1}_\BB, \ldots ,a_n \otimes \mathbf{1}_\BB ) = \sum_{\substack{U \in L_n, \ W \in R_n\\ |V(U)| + |V(W)| = n\\U_{max}\leq W_{min}}} \pm \ U^A(a_1, \ldots, a_n) \otimes W^B(\mathbf{1}_\BB, \ldots , \mathbf{1}_\BB)
\]
The condition $|V(U)| + |V(W)| = n$ implies $b(U) + b(W) \leq n$. Now, $j_b=n$ and so by Lemma \ref{treelemma}, $W^B(\mathbf{1}_\BB, \ldots , \mathbf{1}_\BB) \neq 0$ would mean  $b(W)\geq j_b$ unless $j_b = n,$ which further means $b(W)=n-1$, that is, $W$ is a binary tree, and $|V(W)|=n-1$. Substituting $|V(W)|=n-1$ we have, $|V(U)|+ n-1=n$ or $|V(U)|=1$, hence $U$ is an $n-$corolla. Thus we have \[
\mathfrak{m}_n^\otimes (a_1 \otimes \mathbf{1}_\BB, \ldots ,a_n \otimes \mathbf{1}_\BB ) = \mathfrak{m}_n^\AAA (a_1, \ldots , a_n) \otimes \mathbf{1}_\BB.
\]
Similarly, \[
\mathfrak{m}_n^\otimes (\mathbf{1}_\AAA \otimes b_1, \ldots ,\mathbf{1}_A \otimes b_n ) = \mathbf{1}_\AAA \otimes \mathfrak{m}_n^\BB (b_1, \ldots , b_n).
\]
Next we check the commuting relations in Definition \ref{commutingsubalgebradefn}. Consider $c, c_1, \ldots , c_n \in \AAA\otimes_\infty \BB$ such that for all $i,$ either $c_i=a_i \otimes \mathbf{1}_\BB$ or $c_i=\mathbf{1}_\AAA\otimes b_i$ and $c=a\otimes b$. By Proposition \ref{prop:18UW}, we have

$$
\begin{gathered}
\mathfrak{m}_n^{\otimes}\left(c_1, \ldots, c_n\right)=\sum_{\substack{U \in L_n, \ W \in R_n\\ |V(U)| + |V(W)| = n\\ U_{max}\leq W_{min}}} \pm \ U^A\left(a_1, \ldots, a_n\right) \otimes W^B\left(b_1, \ldots, b_n\right) \\
\end{gathered}
$$
where $a_i=\ua$ for some $i$, and $b_i=\ub$ for some $i$. Clearly, $\quad 0<j_a, j_b<n$ and $j_a+j_b=n$.
By the same argument, $|V(U)|+|V(W)|=n$ or $b(U)+b(W) \leq n$. Suppose $U^A\left(a_1, \ldots, a_n\right) \neq 0$, then by Lemma \ref{treelemma}, $b(U) \geq j_a$. Similarly, $b(W) \geq j_b$.

Now, $\emm_n^{\otimes}\left(c_1, \ldots, c_n\right) \neq 0 $ implies $ n \geq b(U)+b(W) \geqslant j_a+j_b=n$, a contradiction except when $b(U)+b(W)=n$.
But
$$
\begin{aligned}
|V(U)|+|V(W)| & =b(U)+n b(U)+b(W)+n b(W)=n \\
& \Rightarrow n b(U)=n b(W)=0 \\
& \Rightarrow|V(U)|=b(U) \text { and }|V(W)|=b(W)
\end{aligned}
$$
That is, $U$ and $W$ are binary trees with $|V(U)|=n-1=|V(W)|$, where $nb(U)$ means non-binary vertices of $U$. Hence, $|V(U)|+|V(W)|=2 n-2$ contradicting the above, except when $n=2$, in which case we have $c_1=a_1 \otimes \mathbf{1}_\BB$, $c_2=\mathbf{1}_\AAA \otimes b_2$ and
$$
\emm_2^{\otimes}\left(c_1, c_2\right)=
\emm_2^{\otimes}\left(a_1 \otimes \mathbf{1}_\BB, \mathbf{1}_\AAA \otimes b_2\right)=\emm_2^\AAA\left(a_1, \mathbf{1}_\AAA\right) \otimes \emm_2^\BB\left(\mathbf{1}_\BB, b_2\right)=(-1)^{|a_1|} a_1 \otimes b_2
$$
Similarly, $\quad \emm_2^{\otimes}\left(c_2, c_1\right)=\emm_2^{\otimes}\left(\ua \otimes b_2,\left.a_1 \otimes\right.\ub\right)=(-1)^{\left|a_1\right|\left|b_2\right|+\left|b_2\right|} a_1 \otimes b_2$.
Therefore, we conclude
$$
\emm_2^{\otimes}\left(c_1, c_2\right)+(-1)^{\left|c_1\right|^{\prime}\left|c_2\right|^{\prime}} \emm_2^{\otimes}\left(c_2, c_1\right)=0
$$
Finally we compute
$$
\emm_{n+1}^{\otimes}\left(c_1, \ldots, c_i, c, c_{i+1}, \ldots, c_n\right)=\sum_{\substack{|V(U)| + |V(W)| = n+1\\U_{max}\leq W_{min}}} \pm \ U^A \otimes W^B
$$
Clearly, $|V(U)|+|V(W)|=n+1$ implies $b(U)+b(W) \leq n+1$, with $j_a+j_b=n$.
By Lemma \ref{treelemma} above, nontrivial contributions come from when $b(U) \geq j_a$ and $b(W) \geq j_b$. Therefore we have
$$
n=j_a+j_b \leq b(U)+b(W) \leq n+1
$$
giving two possibilities: $U$ is binary and $W$ has one internal vertex, or vice versa.
When $n=0$, both cases contribute and we have:
$$
\emm_1^{\otimes}(a \otimes b)=\emm_1^\AAA(a) \otimes b+(-1)^{|a|} a \otimes \emm_1^\BB(b)
$$
When $n \geq 1$, since $\AAA, \BB$ are unital, the nontrivial contributions come from when all $c_i=\ua \otimes b_i$ or $c_i=a \otimes \ub$. In the first case we have
$$
\emm_{n+1}^{\otimes}\left(c_1, \cdots, c, \ldots, c_n\right)=(-1)^{|a|(1+ \sum_{j\leq i}\left|b_j\right|^{\prime})} a\otimes \emm_{n+1}^\BB\left(b_1, \ldots, b_i, b, b_{i+1}, \ldots, b_n\right) 
$$
and in the second case we have
$$
\emm_{n+1}^{\otimes}\left(c_1, \cdots, c, \ldots, c_n\right)=(-1)^{|b| \sum_{i+1 \leq j} \left|a_j\right|^{\prime}} \emm_{n+1}^\AAA\left(a_1, \ldots, a_i, a, a_{i+1}, \ldots, a_n\right) \otimes b
$$
as required.
\end{proof}

The commuting subalgebras relation will be enough to show that the K\"unneth map on Hochschild complexes (to be defined in the next section) is a chain map. Our first main result (Theorem \ref{thm:gencommutingsub}) is an extension of the commuting subalgebras relation. It will turn out to be the key in showing that the K\"unneth map respects Mukai pairings (and hence, the higher residue pairings on negative cyclic homology.)
\begin{theorem}
    \label{thm:gencommutingsub}
Let $c_i=a_i\otimes \ub=\bara_i$ or $c_i=\ua\otimes b_i=\barb_i$ for each $i$ and 
let $u_\AAA \otimes u_\BB, \ v_\AAA \otimes v_\BB$ be elements of $\AAA\otimes_{\infty} \BB$. Then for $n \geq 2$, 
\begin{align}
        \emm_{n+2}^\otimes \left(c_1,...,c_i, u_\AAA \otimes u_\BB, c_{i+1},...,c_{j}, v_\AAA \otimes v_\BB, c_{j+1},...,c_{n}\right) = 0 \notag
    \end{align}
except when 
\begin{itemize}
    \item  $c_1,...,c_i = \bara_1,...,\bara_p,\barb_1,...,\barb_q=:\bara^{(1)},\barb^{(1)}$ \\
    \item $c_{i+1}, ..., c_j=\barb_{q+1},...,\barb_r,\bara_{p+1},...,\bara_s=:\barb^{(2)},\bara^{(2)}$ \\
    \item  $c_{j+1},...,c_n=\bara_{s+1},...,\bara_t,\barb_{r+1},...,\barb_w=:\bara^{(3)},\barb^{(3)}$ 
\end{itemize}
in which case we have:
\begin{align}
     \emm_{n+2}^\otimes &\left(\bara^{(1)}, \barb^{(1)}, u_\AAA \otimes u_\BB, \barb^{(2)}, \bara^{(2)}, v_\AAA \otimes v_\BB, \bara^{(3)}, \barb^{(3)}\right) = \notag \\
     & (-1)^{ \epsilon}\ U^\AAA \left(a^{(1)}, u_\AAA, a^{(2)}, v_\AAA, a^{(3)}\right) \otimes W^\BB \left(b^{(1)}, u_\BB, b^{(2)}, v_\BB, b^{(3)}\right) \label{eq:com2subalg}
\end{align}
with the tree given in Figure \ref{fig:UW}
and $\epsilon$ is 
explicitly computed in equation (\ref{eqsignofmytree}).
\end{theorem}
\begin{proof}
By Proposition \ref{prop:18UW}, 
$$
\begin{gathered}
\mathfrak{m}_{n+2}^{\otimes}\left( -\otimes - \right) =\sum_{\substack{U \in L_{n+2}, \ W \in R_{n+2}\\ |V(U)| + |V(W)| = n+2\\ U_{\max} \leq W_{\min}}} \pm \ U^A\left(-\right) \otimes W^B\left(-\right) \\
\end{gathered}
$$

The condition $|V(U)| + |V(W)| = n+2$ implies $b(U)+ b(W) \leq n+2$. Clearly, $j_a + j_b = n$, where $j_a, j_b \geq 0$. By Lemma \ref{treelemma}, if $U^A(\cdots) \neq 0$, then $b(U)\geq j_a$ (unless $j_a=n+2$, which is impossible). So we have $b(U)\geq j_a$, and similarly, $b(W) \geq j_b$. Therefore, we have \begin{equation}
    n=j_a +j_b \leq b(U) + b(W) \leq n+2
\end{equation} However, the condition $|V(U)| + |V(W)| = n+2$ can be rewritten as $b(U) + nb(U)+ b(W) + nb(W) = n+2$. Together with the above inequality yields the following three cases:
\begin{enumerate}
    \item $b(U) + b(W) = n+2$; and $nb(U) + nb(W) = 0$
    \item $b(U) + b(W) = n+1$; and $nb(U) + nb(W) = 1$
    \item $b(U) + b(W) = n$; and $nb(U) + nb(W) = 2$
\end{enumerate}

In case (1): the equation $nb(U) + nb(W) = 0$ implies $nb(U) = nb(W) = 0$, or $U$ and $W$ have no non-binary vertices. Hence they are binary trees, and so $b(U) = b(W) = n+1$. We obtain \[
n+2 = |V(U)| + |V(W)| = b(U) + b(W) = n+1+n+1 =2n+2
\]
if and only if $n=0$. This corresponds to $\emm_2^\otimes(u_\AAA \otimes u_\BB, v_\AAA \otimes v_\BB) = (-1)^{|v_\AAA|'|u_\BB|'} \emm_2^\AAA(u_\AAA, v_\AAA) \otimes \emm_2^\BB(u_\BB, v_\BB)$, as shown in Figure \ref{figCase1} below.
    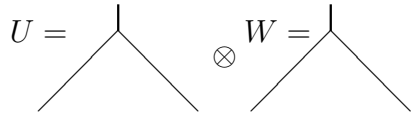
\begin{figure}[h]
		\setlength{\unitlength}{1pt}
		\begin{picture}(140,50)(-175,10)
		
		\linethickness{0.3mm}
		\put(30,40){\line(0,1){10}}
		\linethickness{0.3mm}
		\put(0,10){\line(1,1){30}}
		\linethickness{0.3mm}
		\put(60,10){\line(-1,1){30}}
		
		\linethickness{0.3mm}
		\put(110,40){\line(0,1){10}}
		\linethickness{0.3mm}
		\put(80,10){\line(1,1){30}}
		\linethickness{0.3mm}
		\put(140,10){\line(-1,1){30}}

		\put(70,30){\makebox(0,0)[cc]{$\otimes$}}
		\put(0,40){\makebox(0,0)[cc]{$U =$}}
		\put(90,40){\makebox(0,0)[cc]{$W =$}}
		\end{picture}
        \caption{Case 1: $U$ and $W$ have only binary vertices}
         \label{figCase1}
\end{figure}

In case (2): the equation $nb(U) + nb(W) = 1$ implies $nb(U) =0$ and $nb(W) = 1$, or $nb(U) =1$ and $nb(W) = 0$. Suppose $nb(U) =0$ and $nb(W) = 1$. Then $U$ is  a binary tree with $b(U) = n+1$, so we have $n+1 + |V(W)| = n+2$ or $|V(W)| = 1$. Hence, $W$ is the $n+2$-corolla, $n\geq 1$. The  corresponding operation is only non-zero when $c_i = \barb_i$ for all $i$. Rewrite $\barb^{(1)}=c_1,...,c_i$; $\barb^{(2)}=c_{i+1},...,c_j$; and $\barb^{(3)}=c_{j+1},...,c_n$, then this case becomes
\begin{equation}
    \emm_{n+2}^\otimes(\barb^{(1)}, u_\AAA\otimes u_\BB, \barb^{(2)}, v_\AAA\otimes v_\BB, \barb^{(3)}) = (-1)^{\epsilon} \emm_2^\AAA(u_\AAA, v_\AAA) \otimes \emm_{n+2}^\BB(b^{(1)}, u_\BB,b^{(2)}, v_\BB, b^{(3)})
\end{equation} 
as shown in Figure \ref{figCase2a}.

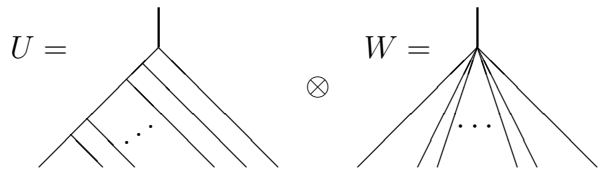
\begin{figure}[h]
	\setlength{\unitlength}{1.5pt}
	\begin{picture}(140,50)(-100,10)
	
	\linethickness{0.3mm}
	\put(30,40){\line(0,1){10}}
	\linethickness{0.3mm}
	\put(0,10){\line(1,1){30}}
	\linethickness{0.3mm}
	\put(52,10){\line(-1,1){26}}
	\linethickness{0.3mm}
	\put(44,10){\line(-1,1){22}}
	\linethickness{0.3mm}
	\put(24,10){\line(-1,1){12}}
	\linethickness{0.3mm}
	\put(16,10){\line(-1,1){8}}
	\linethickness{0.3mm}
	\put(10,10){\line(-1,1){5}}
	\linethickness{0.3mm}
	\put(60,10){\line(-1,1){30}}
	
	\linethickness{0.3mm}
	\put(110,40){\line(0,1){10}}
	\linethickness{0.3mm}
	\put(80,10){\line(1,1){30}}
	\linethickness{0.3mm}
	\put(95,10){\line(1,2){15}}
	\linethickness{0.3mm}
	\put(125,10){\line(-1,2){15}}
	\linethickness{0.3mm}
	\put(100,10){\line(1,3){10}}
	\linethickness{0.3mm}
	\put(120,10){\line(-1,3){10}}
	\linethickness{0.3mm}
	\put(140,10){\line(-1,1){30}}

	\put(70,30){\makebox(0,0)[cc]{$\otimes$}}
	\put(0,40){\makebox(0,0)[cc]{$U =$}}
	\put(90,40){\makebox(0,0)[cc]{$W=$}}
	\put(28,20){\makebox(0,0)[cc]{$. $}}
	\put(25,18){\makebox(0,0)[cc]{$. $}}
	\put(22,16){\makebox(0,0)[cc]{$. $}}
	\put(110,20){\makebox(0,0)[cc]{$\cdots$}}
	\end{picture}
    \caption{Case 2a: $U$ has only binary vertices and $W$ is a corolla}
    \label{figCase2a}
\end{figure}

Similarly, if $nb(U) =1$ and $nb(W) = 0$, then $W$ is a binary tree and $U$ is the $n+2$-corolla. This similarly corresponds to $c_i = \bara_{i}$ for all $i$, that is, \begin{equation}
    \emm_{n+2}^\otimes(\bara^{(1)}, u_\AAA\otimes u_\BB, \bara^{(2)}, v_\AAA\otimes v_\BB, \bara^{(3)}) = (-1)^{\epsilon} \emm_{n+2}^\AAA(a^{(1)}, u_\AAA, a^{(2)}, v_\AAA, a^{(3)}) \otimes \emm_2^\BB(u_\BB, v_\BB) \end{equation} 
as shown in Figure \ref{figCase2b}.

\begin{figure}[h]
	\setlength{\unitlength}{1.5pt}
	\begin{picture}(140,50)(-100,10)
	
	\linethickness{0.3mm}
	\put(30,40){\line(0,1){10}}
	\linethickness{0.3mm}
	\put(0,10){\line(1,1){30}} 
	\linethickness{0.3mm}
	\put(15,10){\line(1,2){15}}
	\linethickness{0.3mm}
	\put(45,10){\line(-1,2){15}}
	\linethickness{0.3mm}
	\put(20,10){\line(1,3){10}}
	\linethickness{0.3mm}
	\put(40,10){\line(-1,3){10}}
	
	\linethickness{0.3mm}
	\put(60,10){\line(-1,1){30}}
	
	\linethickness{0.3mm}
	\put(110,40){\line(0,1){10}}
	\linethickness{0.3mm}
	\put(80,10){\line(1,1){30}} 
	\linethickness{0.3mm}
	\put(88,10){\line(1,1){26}}
	\linethickness{0.3mm}
	\put(96,10){\line(1,1){22}}
	\linethickness{0.3mm}
	\put(116,10){\line(1,1){12}}
	\linethickness{0.3mm}
	\put(122,10){\line(1,1){9}}
	\linethickness{0.3mm}
	\put(130,10){\line(1,1){5}}
	
	\linethickness{0.3mm}
	\put(140,10){\line(-1,1){30}}

	\put(70,30){\makebox(0,0)[cc]{$\otimes$}}
	\put(0,40){\makebox(0,0)[cc]{$U =$}}
	\put(90,40){\makebox(0,0)[cc]{$W =$}}
	\put(30,20){\makebox(0,0)[cc]{$\cdots$}}
	\put(110,20){\makebox(0,0)[cc]{$. $}}
	\put(113,18){\makebox(0,0)[cc]{$. $}}
	\put(116,16){\makebox(0,0)[cc]{$. $}}
	\end{picture}
    \caption{Case 2b: $U$ is the corolla, and $W$ has only binary vertices}
    \label{figCase2b}
\end{figure}
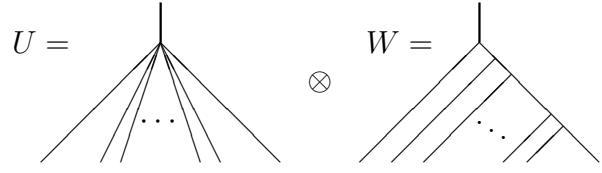

In case (3): the equation $nb(U) + nb(W) = 2$ further yields the following subcases: (a) $nb(U)=2$ and  $nb(W) = 0$; (b) $nb(U)=0$ and $nb(W) = 2$; (c) $nb(U) = nb(W) = 1$. In subcase (a): $nb(W) = 0$ implies $W$ is a binary tree and so $b(W) = n+1 = |V(W)|$. On the other hand, $U$ has 2 non-binary vertices. We therefore obtain $|V(U)| = 1$. But $|V(U)|= b(U) + nb(U)$ which forces $b(U)=-1$, a contradiction. A similar result also follows from subcase (b).

Finally we consider the case $nb(U)=1=nb(W)$. By Lemma \ref{treelemma}, $U^A(\cdots) = 0$ unless $b(U) \geq j_a$ and similarly, $W^B(\cdots) = 0$ unless $b(W) \geq j_b$, where $j_a, j_b \geq 0$. Further, from $|V(U)| + |V(W)| = n+2$, we obtain $b(U) + b(W) =n$, or \begin{equation}
    n = j_a + j_b \leq b(U) + b(W) = n
\end{equation}
which implies $j_a + j_b = b(U) + b(W)$ or $b(U) = j_a$ and $b(W) = j_b$, that is, on $U$ the number of binary vertices corresponds to $j_a$ (the number of $b$'s) and on $W$ the number of binary vertices corresponds to $j_b$ (the number of $a$'s). 

Note that if $j_a=0, \ (j_b>0)$ then $U$ is the $n+2$-corolla but $W$ will have only binary vertices, which is fine, except that it does not belong to this case since $nb(W)=0$. A similar thing happens when $j_b=0, \ (j_a>0)$. If $j_a=j_b=0$, we are back to Case 1. So we consider $j_a,j_b>0.$

On $U$, the $a$'s (coming from $\bara$'s) must be on leaves incident with the unique non-binary vertex since otherwise, the right-leaning internal edges created are collapsed in $W$ forcing the corresponding $1$'s to be inserted on a non-binary vertex on $W$, which is zero by unitality. Similarly, on $W$, the $b$'s (coming from $\barb$'s) must be inserted on the unique non-binary vertex.

Again, on $U$, we must have $v_\AAA$ inserted on the unique non-binary vertex. Otherwise, we have the following cases shown in Figure \ref{fig:fig1allcases}.\\

\begin{figure}[h]
\centering
    \scalebox{0.9}{\begin{tikzpicture}[
    every node/.style={font=\textnormal{}},
    T/.style={
        draw,
        ellipse,
        minimum width=2.0cm,
        minimum height=1.15cm
    },
    S/.style={
        draw,
        ellipse,
        minimum width=1.45cm,
        minimum height=0.9cm
    },
    branch/.style={draw, line width=0.7pt},
    e/.style={font=\textnormal{}}
]


\begin{scope}[xshift=0cm]

    \node[T] (T1) at (0,2.8) {$T$};

    \coordinate (v1) at (0,1.65);
    \draw[branch] (T1.south) -- (v1);

    \node[e] at (-0.18,1.9) {$e$};

    \node (c21) at (-0.85,0.75) {$v_A$};
    \node (one1) at (0.85,0.75) {$1$};

    \draw[branch] (v1) -- (c21.north);
    \draw[branch] (v1) -- (one1.north);

    \node at (0,-0.25) {\textnormal{(i)}};

\end{scope}


\begin{scope}[xshift=3.7cm]

    \node[T] (T2) at (0,2.8) {$T$};

    \coordinate (v2) at (0,1.65);
    \draw[branch] (T2.south) -- (v2);

    \node[e] at (-0.18,1.9) {$e$};

    \node (one2) at (-0.85,0.75) {$1$};
    \node (c22) at (0.85,0.75) {$v_A$};

    \draw[branch] (v2) -- (one2.north);
    \draw[branch] (v2) -- (c22.north);

    \node at (0,-0.25) {\textnormal{(ii)}};

\end{scope}


\begin{scope}[xshift=7.4cm]

    \node[T] (T3) at (0,2.8) {$T$};

    \coordinate (v3) at (0,1.65);
    \draw[branch] (T3.south) -- (v3);

    \node[e] at (-0.18,1.9) {$e$};

    \node (c13) at (-0.85,0.75) {$u_A$};
    \node (c23) at (0.85,0.75) {$v_A$};

    \draw[branch] (v3) -- (c13.north);
    \draw[branch] (v3) -- (c23.north);

    \node at (0,-0.25) {\textnormal{(iii)}};

\end{scope}


\begin{scope}[xshift=11.1cm]

    \node[T] (T4) at (0,2.8) {$T$};

    \coordinate (v4) at (0,1.65);
    \draw[branch] (T4.south) -- (v4);

    \node[e] at (-0.18,1.9) {$e$};

    \node[S] (S4) at (-0.85,0.75) {$S$};
    \node (c24) at (0.85,0.75) {$v_A$};

    \draw[branch] (v4) -- (S4.north);
    \draw[branch] (v4) -- (c24.north);

    \node at (0,-0.25) {\textnormal{(iv)}};

\end{scope}


\begin{scope}[xshift=14.8cm]

    \node[T] (T5) at (0,2.8) {$T$};

    \coordinate (v5) at (0,1.65);
    \draw[branch] (T5.south) -- (v5);

    \node[e] at (-0.18,1.9) {$e$};

    \node (c25) at (-0.85,0.75) {$v_A$};
    \node[S] (S5) at (0.85,0.75) {$S$};

    \draw[branch] (v5) -- (c25.north);
    \draw[branch] (v5) -- (S5.north);

    \node at (0,-0.25) {\textnormal{(v)}};

\end{scope}

\end{tikzpicture}} 
    \caption{All 5 subcases: (i) and (ii) shows $v_a$ sharing a binary vertex with $1$ to the right, and left of $v_A$ respectively; (iii) shows $v_A$ sharing a binary vertex with $u_A$; (iv) and (v) shows $v_A$ sharing a binary vertex with a subtree $S$ to the left, and right of $v_A$ respectively.}
    \label{fig:fig1allcases}
\end{figure}
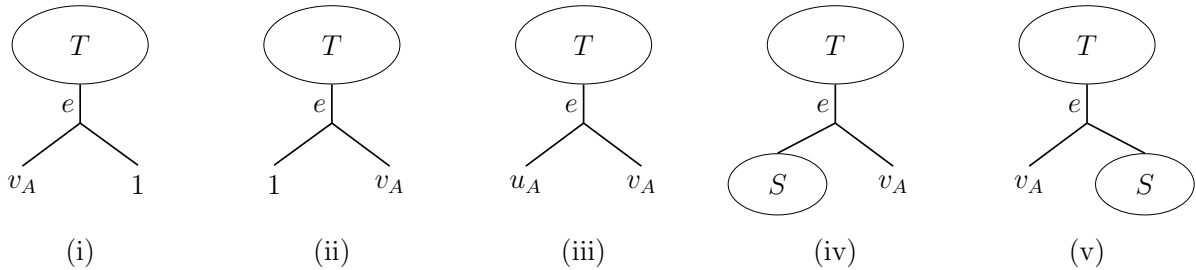

In case (i) since there are certainly leaves to the left of $v_\AAA$, (namely, $u_\AAA$) the edge $e$ could be the last leaf and we would have attached corollas on it leading to a contradiction. Otherwise, $e$ must eventually (up to finitely many binary attachments) join the unique non-binary vertex with $a$ attached to the right as shown in Figure \ref{fig:case1}.

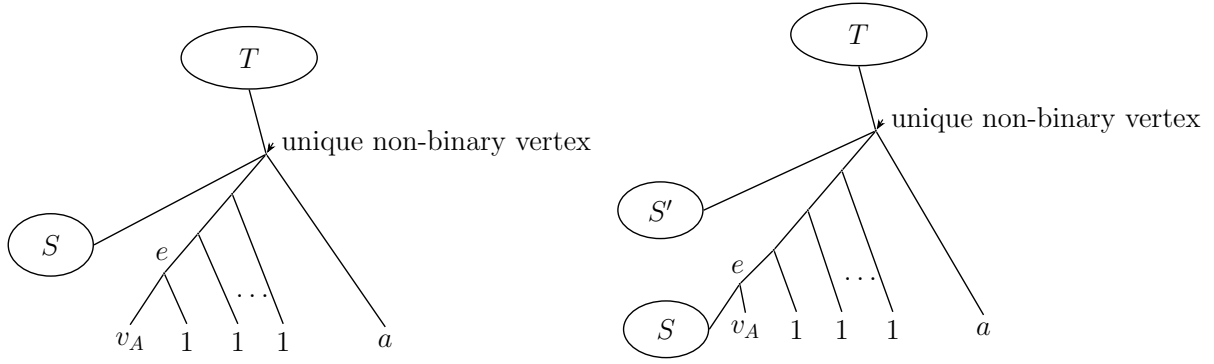
\begin{figure}[h]
    \centering
    \scalebox{0.75}{\begin{tikzpicture}[
    every node/.style={font=\large},
    ellipse_node/.style={
        draw,
        ellipse,
        minimum width=2.4cm,
        minimum height=1.1cm,
        line width=0.7pt
    },
    S_node/.style={
        draw,
        ellipse,
        minimum width=1.5cm,
        minimum height=1.1cm,
        line width=0.7pt
    },
    branch/.style={draw, line width=0.7pt},
    dot/.style={fill, circle, inner sep=0.3pt}
]

    \node[ellipse_node] (T) at (0, 4) {$T$};

    \node[dot] (v_nb) at (0.3, 2.3) {};
    \draw[branch] (T.south) -- (v_nb);

    \node[align=left, font=\large] (annot) at (3.3, 2.5) {unique non-binary vertex};
    \draw[-{Stealth[scale=0.9]}, line width=0.6pt] (annot.west) -- (v_nb);

    \node[S_node] (S) at (-3.5, 0.7) {$S$};
    \draw[branch] (v_nb) -- (S.east);

    \node[dot] (u2) at (-0.3, 1.6) {};
    \node[dot] (u1) at (-0.9, 0.9) {};
    \node[dot] (u0) at (-1.5, 0.2) {};

    \draw[branch] (v_nb) -- (u2);
    \draw[branch] (u2) -- (u1);
    \draw[branch] (u1) -- (u0) node[midway, left=2pt] {$e$};

    \node (vA) at (-2.1, -1.0) {$v_A$};
    \node (one0) at (-1.1, -1.0) {$1$};
    \node (one1) at (-0.2, -1.0) {$1$};
    \node (one2) at (0.6, -1.0) {$1$};

    \draw[branch] (u0) -- (vA.north);
    \draw[branch] (u0) -- (one0.north);
    \draw[branch] (u1) -- (one1.north);
    \draw[branch] (u2) -- (one2.north);

    \node at (0.1, -0.2) {$\dots$};

    \node (a) at (2.4, -1.0) {$a$};

    \draw[branch] (v_nb) -- (a.north);

\end{tikzpicture}}
    \scalebox{0.75}{\begin{tikzpicture}[
    every node/.style={font=\large},
    ellipse_node/.style={
        draw,
        ellipse,
        minimum width=2.4cm,
        minimum height=1.1cm,
        line width=0.7pt
    },
    S_node/.style={
        draw,
        ellipse,
        minimum width=1.5cm,
        minimum height=1.0cm,
        line width=0.7pt
    },
    branch/.style={draw, line width=0.7pt},
    dot/.style={fill, circle, inner sep=0.3pt}
]

    \node[ellipse_node] (T) at (0, 4) {$T$};

    \node[dot] (v_nb) at (0.3, 2.3) {};
    \draw[branch] (T.south) -- (v_nb);

    \node[align=left, font=\large] (annot) at (3.3, 2.5) {unique non-binary vertex};
    \draw[-{Stealth[scale=0.9]}, line width=0.6pt] (annot.west) -- (v_nb);

    \node[S_node] (Sp) at (-3.5, 0.9) {$S'$};
    \draw[branch] (v_nb) -- (Sp.east);

    \node[dot] (u3) at (-0.3, 1.6) {};
    \node[dot] (u2) at (-0.9, 0.9) {};
    \node[dot] (u1) at (-1.5, 0.2) {};
    \node[dot] (u0) at (-2.1, -0.4) {};

    \draw[branch] (v_nb) -- (u3);
    \draw[branch] (u3) -- (u2);
    \draw[branch] (u2) -- (u1);
    \draw[branch] (u1) -- (u0)node[midway, left=2pt] {$e$}; 

    \node[S_node] (S) at (-3.4, -1.2) {$S$};
    \draw[branch] (u0) -- (S.east);

    \node (vA) at (-2.0, -1.2) {$v_A$};
    \node (one0) at (-1.1, -1.2) {$1$};
    \node (one1) at (-0.3, -1.2) {$1$};
    \node (one2) at (0.6, -1.2) {$1$};

    \draw[branch] (u0) -- (vA.north);
    \draw[branch] (u1) -- (one0.north);
    \draw[branch] (u2) -- (one1.north);
    \draw[branch] (u3) -- (one2.north);

    \node at (0.07, -0.3) {$\dots$};

    \node (a) at (2.2, -1.2) {$a$};
    \draw[branch] (v_nb) -- (a.north);

\end{tikzpicture}}
    \caption{Illustration of cases (i) and (iv) respectively showing where edge $e$ has to (up to finitely many binary attachments) join the unique non-binary vertex with $a$ attached.}
    \label{fig:case1}
\end{figure}

However, resolving this tree leads to attaching edge $e$ on the leaf containing $a$. This creates a right-leaning internal edge which is collapsed to obtain the unique non-binary vertex on tree $W$ containing the unit $1$. Hence, $0$ by unitality. A similar argument applies to case (ii).

In case (iii) you automatically create an additional binary vertex and so $b(U)>j_a$, a contradiction.

In case (iv), if edge $e$ is not the last leaf, then the unique non-binary vertex is contained in $T$ with $a$ attached to the right as in case (i) as shown in Figure \ref{fig:case1}. Once again, resolving this tree yields $W=0$ by unitality. A similar argument works for case (v).

So, on tree $U$, we must have all $a's$ and $v_\AAA$ inserted on the unique non-binary vertex. On the other hand, $1$'s must be inserted on binary vertices, while $u_\AAA$ is free in general to join a binary vertex or the non-binary vertex. Similarly on tree $W$, we must have all $b's$ and $u_\BB$ inserted on the unique non-binary vertex, while $v_\BB$ is free in general to join a binary vertex or the non-binary vertex.

We thus have a unique non-trivial arrangement, namely, $\barc_1,...,\barc_i$ must be of the ordered form $(\bara^{(1)}, \barb^{(1)})$, that is, all the $\bara$'s must come before $\barb$'s since any term of the form $\barb ,\bara$ forces one to attach corollas on the last leaf or have $1$ share a binary vertex with $a$ on $U$. Similarly, $\barc_{i+1},...,\barc_j$ must be ordered as $(\barb^{(2)}, \bara^{(2)})$; since any term of the form $\bara,\barb$ forces $1$ to share a binary vertex with $a$ or $v_\AAA$ on $U$. Similarly, we must also have $\barc_{j+1},...,\barc_n$ be of the form $(\bara^{(3)}, \barb^{(3)})$.

These restrictions guarantee the existence of unique $U \in L_{n+2}$ and $W \in R_{n+2}$ such that
\begin{align}
     \emm_*^\otimes &\left(\bara^{(1)}, \barb^{(1)}, u_\AAA \otimes u_\BB, \barb^{(2)}, \bara^{(2)}, v_\AAA \otimes v_\BB, \bara^{(3)}, \barb^{(3)}\right) = \notag \\
     & (-1)^{ \epsilon} \ U^A \left(a^{(1)}, u_\AAA, a^{(2)}, v_\AAA, a^{(3)}\right) \otimes W^B \left(b^{(1)}, u_\BB, b^{(2)}, v_\BB, b^{(3)}\right) \notag
\end{align}
where trees $U$ and $W$ are shown below.

\begin{figure}[h]
    \scalebox{.6}{
		\setlength{\unitlength}{1.5pt}
		\begin{picture}(240,150)(0,0)

		\linethickness{0.3mm}
        \put(120,130){\line(0,1){20}}

		\linethickness{0.3mm}
        \put(0,10){\line(1,1){120}}
	    
	    \linethickness{0.3mm}
	    \put(33,10){\line(2,3){65}} 
	    \linethickness{0.3mm}
	    \put(49,10){\line(1,2){50}} 
	     
	    \linethickness{0.3mm}
	    \put(60,10){\line(-1,2){5.5}}
	    \linethickness{0.3mm}
	    \put(70,10){\line(-1,2){10.5}}
	    \linethickness{0.3mm}
	    \put(83,10){\line(-1,2){17}}
	    \linethickness{0.3mm}
	    \put(93,10){\line(-1,2){22}}
	    \linethickness{0.3mm}
	    \put(103,10){\line(-1,2){27}}
	    \linethickness{0.3mm}
	    \put(123,10){\line(-1,2){37}}
	     
	    \linethickness{0.3mm}
	    \put(132,10){\line(-1,3){33}}	    
	    \linethickness{0.3mm}
     	\put(139,10){\line(-2,5){39}}
	    \linethickness{0.3mm}
	    \put(149,10){\line(-1,2){50}} 
	    
		\linethickness{0.3mm}
	    \put(166,10){\line(-2,3){67}}

		\linethickness{0.3mm}
		\put(200,10){\line(-1,1){100}}		
		\linethickness{0.3mm}
		\put(210,10){\line(-1,1){105}}		
		\linethickness{0.3mm}
		\put(220,10){\line(-1,1){110}}
		\linethickness{0.3mm}
        \put(240,10){\line(-1,1){120}}
        
        \put(0,120){\makebox(0,0)[cc]{$U =$}}

        \put(0,5){\makebox(0,0)[cc]{$a$}}
        \put(16,5){\makebox(0,0)[cc]{$\cdots$}}
        \put(33,5){\makebox(0,0)[cc]{$a$}}
        \put(49,5){\makebox(0,0)[cc]{$1$}}	
        \put(69,5){\makebox(0,0)[cc]{$\cdots$}}
        \put(83,5){\makebox(0,0)[cc]{$1$}}
        \put(93,5){\makebox(0,0)[cc]{$u_A$}}
        \put(103,5){\makebox(0,0)[cc]{$1$}}
        \put(113,5){\makebox(0,0)[cc]{$\cdots$}}
        \put(123,5){\makebox(0,0)[cc]{$1$}}
        \put(132,5){\makebox(0,0)[cc]{$a\ $}}
        \put(135,5){\makebox(0,0)[cc]{$...$}}
        \put(139,5){\makebox(0,0)[cc]{$\ a$}}
        \put(149,5){\makebox(0,0)[cc]{$v_A$}}
        \put(166,5){\makebox(0,0)[cc]{$a$}}
        \put(182,5){\makebox(0,0)[cc]{$\cdots$}}
        \put(200,5){\makebox(0,0)[cc]{$a$}}
        \put(210,5){\makebox(0,0)[cc]{$1$}}
        \put(225,5){\makebox(0,0)[cc]{$\cdots$}}
        \put(240,5){\makebox(0,0)[cc]{$1$}}
        
        \put(50,50){\makebox(0,0)[cc]{$\cdots$}}
        \put(68,20){\makebox(0,0)[cc]{$.$}}
        \put(70,22){\makebox(0,0)[cc]{$.$}}
        \put(72,24){\makebox(0,0)[cc]{$.$}}
        \put(96,40){\makebox(0,0)[cc]{$.$}}
        \put(98,42){\makebox(0,0)[cc]{$.$}}
        \put(100,44){\makebox(0,0)[cc]{$.$}}
        \put(140,60){\makebox(0,0)[cc]{$\cdots$}}
        \put(128,30){\makebox(0,0)[cc]{$...$}}
        \put(158,88){\makebox(0,0)[cc]{$.$}}
	\put(156,86){\makebox(0,0)[cc]{$.$}}
        \put(154,84){\makebox(0,0)[cc]{$.$}}		
			
\end{picture}
} \qquad \quad
    \scalebox{.6}{
		\setlength{\unitlength}{1.5pt}
		\begin{picture}(240,150)(0,0)

		\linethickness{0.3mm}
		\put(120,130){\line(0,1){20}}

		\linethickness{0.3mm}
		\put(0,10){\line(1,1){120}}
        \linethickness{0.3mm}
        \put(10,10){\line(1,1){115}}
        \linethickness{0.3mm}
        \put(25,10){\line(1,1){108}}
        \linethickness{0.3mm}
        \put(40,10){\line(1,1){100}}

         \linethickness{0.3mm}
         \put(73,10){\line(2,3){67}}
         
         \linethickness{0.3mm}
         \put(90,10){\line(1,2){50}} 
         \linethickness{0.3mm}
         \put(100,10){\line(2,5){40}}
         \linethickness{0.3mm}
         \put(107,10){\line(1,3){33}}	    

         \linethickness{0.3mm}
         \put(117,10){\line(1,2){37}}
         \linethickness{0.3mm}
         \put(137,10){\line(1,2){27}}
         \linethickness{0.3mm}
         \put(153,10){\line(1,2){19}}
         \linethickness{0.3mm}
         \put(170,10){\line(1,2){10}}
         
         \linethickness{0.3mm}
         \put(180,10){\line(1,2){5}}

        \linethickness{0.3mm}
        \put(190,10){\line(-1,2){50}}         
        \linethickness{0.3mm}
        \put(207,10){\line(-2,3){65}}

		\linethickness{0.3mm}
        \put(240,10){\line(-1,1){120}}

         \put(0,120){\makebox(0,0)[cc]{$W =$}}

         \put(0,5){\makebox(0,0)[cc]{$1$}}
         \put(13,5){\makebox(0,0)[cc]{$\cdots$}}
         \put(25,5){\makebox(0,0)[cc]{$1$}}
         \put(40,5){\makebox(0,0)[cc]{$b$}}	
         \put(52,5){\makebox(0,0)[cc]{$\cdots$}}
         \put(73,5){\makebox(0,0)[cc]{$b$}}
         \put(90,5){\makebox(0,0)[cc]{$u_B$}}
         \put(100,5){\makebox(0,0)[cc]{$b \ $}}
         \put(103,5){\makebox(0,0)[cc]{$...$}}
         \put(107,5){\makebox(0,0)[cc]{$\ b$}}
         \put(117,5){\makebox(0,0)[cc]{$1$}}
         \put(127,5){\makebox(0,0)[cc]{$\cdots$}}
         \put(137,5){\makebox(0,0)[cc]{$1$}}
         \put(153,5){\makebox(0,0)[cc]{$v_B$}}
         \put(166,5){\makebox(0,0)[cc]{$1$}}
         \put(178,5){\makebox(0,0)[cc]{$\cdots$}}
         \put(190,5){\makebox(0,0)[cc]{$1$}}
         \put(207,5){\makebox(0,0)[cc]{$b$}}
         \put(225,5){\makebox(0,0)[cc]{$\cdots$}}
         \put(240,5){\makebox(0,0)[cc]{$b$}}

         \put(77,70){\makebox(0,0)[cc]{$\cdots$}}
         \put(90,50){\makebox(0,0)[cc]{$\cdots$}}
         \put(107,20){\makebox(0,0)[cc]{$...$}}
         \put(147,50){\makebox(0,0)[cc]{$\cdots$}}
         \put(180,60){\makebox(0,0)[cc]{$\cdots$}}
         \put(176,18){\makebox(0,0)[cc]{$.$}}
         \put(178,16){\makebox(0,0)[cc]{$.$}}
         \put(180,14){\makebox(0,0)[cc]{$.$}}

\end{picture}
}
    \caption{Unique trees $U$ and $W$}
    \label{fig:UW}
\end{figure}
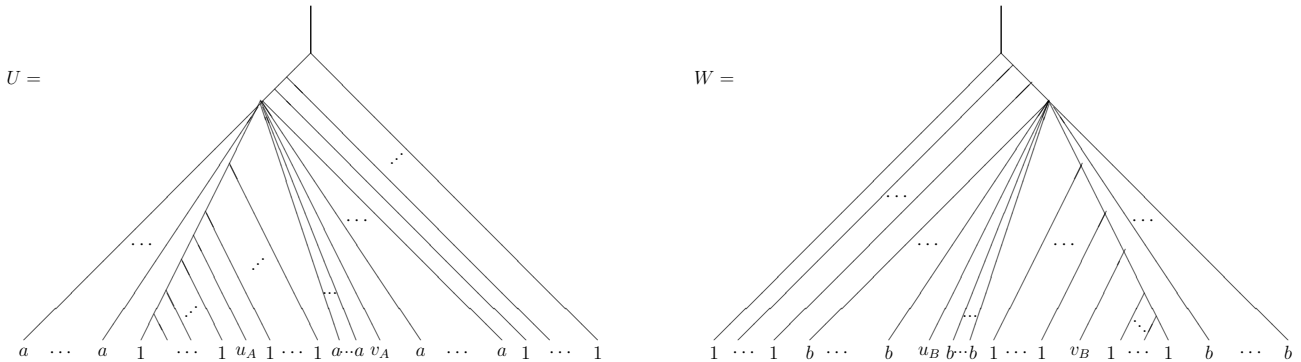

\vspace{-0.5cm}
where 
\vspace{-0.5cm}
\begin{align}
    \epsilon=&(\#a^{(2)}+\#a^{(3)})(\#b^{(1)}+\#b^{(2)})+\#b^{(3)}(\#a^{(1)}+\#a^{(2)}+\#a^{(3)})\notag\\
    +|v_\BB||a^{(3)}| + &(|u_\BB|+|b^{(2)}|)(|a^{(2)}|+|v_A|+|a^{(3)}|) + |b^{(1)}|(|u_A|+|a^{(2)}|+|v_A|+|a^{(3)}|) \notag\\
    +&\#b^{(3)}|a^{(1)}| + (\#b^{(1)}+\#b^{(2)}+\#b^{(3)})(|a^{(2)}|+|a^{(3)}|+|u_A|+|v_A|)\notag\\
    +(\#a^{(2)}&+\#a^{(3)})(|b^{(1)}|+|u_B|+|b^{(2)}|) + (\#a^{(3)}-1)|v_B|
    \label{eqsignofmytree}
\end{align}
and $\#a^{(i)}=m$ if $a^{(i)}=a^{(i)}_1\otimes \cdots \otimes a^{(i)}_m$, similarly defined for $\#b^{(i)}$.
\end{proof}

\section{The K\"unneth maps}
\label{section-applications}
In this section, all our $A_\infty$ algebras will be smooth and proper, strictly unital, and $\mathbb{Z}/2\mathbb{Z}$-graded as defined, unless otherwise stated. We will use the shuffle product to define a K\"unneth map between tensor product of $A_\infty$ algebras as follows.
\subsection{Shuffle and cyclic shuffle maps}
The (cyclic) shuffle product will play an important role in defining the K\"unneth maps. We introduce them here.
\begin{definition}
    Let $m, n \geq 0$ be integers. We say a permutation $\sigma \in S_{n+m}$ is a $(n, m)$-shuffle if $\sigma(i)<\sigma(j)$ for $i<j \leq n$ or $n+1 \leq i<j$.
Consider $a=a_1 \otimes \ldots \otimes a_n \in \AAA^{\otimes n}$ and $ b=b_1 \otimes \ldots \otimes b_m \in \BB^{\otimes m}$ and denote
$$
u_i= \begin{cases}a_i \otimes {\bf1}_B, &  \ 1 \leqslant i \leqslant n, \\{\bf1}_A \otimes b_{i-n}, &  n<1 \leqslant n+m .\end{cases}
$$
Given a $(n, m)$-shuffle $\sigma$, we define $
\sigma(a; b)=(-1)^{\epsilon} u_{\sigma^{-1}(1)} \otimes \ldots \otimes u_{\sigma^{-1}(n+m)} \in (A\otimes B)^{\otimes (n+m)},$ where $ \displaystyle \epsilon=\sum_{\substack{i \leqslant n<j \\ \sigma(j)<\sigma(i)}}|u_j|'|u_i|' .$

With the same notation, define the shuffle product
\begin{equation}
    \operatorname{sh}(a ; b)=\sum_\sigma \sigma(a; b) \label{shuffle_defn}
\end{equation}
where we sum over all $(n, m)$-shuffles $\sigma$. 
\end{definition}
Some properties of the shuffle product $\operatorname{sh}$ that will be useful for us are stated.

\begin{lemma}[{\parencite[Lemma 5.8]{amorim2016tensor}}]\label{lemma21}
We have the following:
$$
(a) \sum_{(\operatorname{sh}(a; b))} \operatorname{sh}(a ; b)^{(1)} \otimes \operatorname{sh}(a;b)^{(2)} =\sum_{(a)(b)}(-1)^{|a^{(2)}|' |b^{(1)}|'} \operatorname{sh}\left(a^{(1)} ; b^{(1)}\right) \otimes \operatorname{sh} \left(a^{(2)} ; b^{(2)}\right)
$$

(b) $\operatorname{sh}\left( a^{(1)}, x,  a^{(2)} ; b\right)= \sum_{(b)}(-1)^{|b^{(1)}|'(|x|' + |a^{(2)}|')} \operatorname{sh}\left( a^{(1)} ; b^{(1)}\right) \otimes x \otimes \operatorname{sh}\left( a^{(2)} ; b^{(2)}\right)$ 

(c) $\operatorname{sh}\left(a ;  b^{(1)}, y,  b^{(2)}\right)= \sum_{(a)}(-1)^{|a^{(2)}|'(|y|' + |b^{(1)}|')} \operatorname{sh}\left( a^{(1)} ; b^{(1)}\right) \otimes y \otimes \operatorname{sh}\left( a^{(2)} ; b^{(2)}\right)$ 
\end{lemma}

\begin{definition}\label{defcycshuff}
    A $(n, m)$-cyclic shuffle is a permutation
$\{\sigma(1), ... , \sigma(n + m)\}$ of $\{1,...,n+m\}$ obtained as follows. Perform a cyclic permutation of any order on the set $\{1, ... ,n\}$ and a cyclic permutation of any order on the set $\{n+ 1,...,n+m\}$. Then shuffle the two results to obtain $\{\sigma(1), ... , \sigma(n + m)\}$. This is a cyclic shuffle if $\sigma(n+1)$ appears before $\sigma(1)$, denoted as $\sigma(n+1)<\sigma(1)$ in this sequence.
\end{definition}
Note how this compares to the definition in \parencite[Section 4.3.2]{loday2013cyclic} or \parencite[Section 2.2]{shklyarov2016hodge}. In our context with $\bfa=(a_0|...,a_n)$ and $\bfb=(b_0|...,b_m)$, the cyclic shuffle of $\{\bfa, \bfb\}$ will involve performing a cyclic permutation of any order on $\bfa$ and $\bfb$ respectively, and then shuffling the result so that $b_0 $ appears before $a_0$. We denote this operation by $\operatorname{Sh}$. 

Let
$$
v_i= \begin{cases}a_i \otimes {\bf1}_B & , \ 1 \leqslant i \leqslant n+1, \\{\bf1}_A \otimes b_{i-n-1} & , n+1<1 \leqslant n+m+2 .\end{cases}
$$
Given a $(n+1, m+1)$-cyclic shuffle $\sigma$, we define $
\sigma(\bfa; \bfb)=(-1)^{\epsilon} v_{\sigma^{-1}(1)} \otimes \ldots \otimes v_{\sigma^{-1}(n+m+2)}\in (A\otimes B)^{\otimes (n+m+2)},$ where $ \displaystyle \epsilon=\sum_{\substack{i \leqslant n+1<j \\ \sigma(j)<\sigma(i)}}|v_j|'|v_i|' .$

Define the cyclic shuffle product
\begin{equation}
    \operatorname{Sh}(\bfa ; \bfb)=\sum_\sigma (-1)^{@_A+@_B} \sigma(\bfa; \bfb) \label{cycshuffle_defn}
\end{equation}
where we sum over all $(n+1, m+1)$-cyclic shuffles $\sigma$, and $@_A$, $@_B$ are the signs of the cyclic permutations of $\bfa$ and $\bfb$ respectively.

\subsection{Hochschild homology}
In the setting of ordinary DG algebras (and strictly associative algebras), the K\"unneth isomorphism for Hochschild homology is established via the classical Eilenberg–Zilber shuffle map $k_{\text{DG}}: C_*(A) \otimes C_*(B) \to C_*(A \otimes B)$ (see, e.g., Loday \parencite[Section 4.2]{loday2013cyclic}, and Shklyarov \parencite[Section 2.2]{shklyarov2016hodge}). Because DG multiplication is strictly associative (i.e., $\emm_n = 0$ for $n \ge 3$), the Hochschild differential $\bee_{\text{DG}}$ in $C_*(A\otimes B)$ splits linearly as $\bee_{\text{DG}} = \bee_A \otimes \text{id} + \text{id} \otimes \bee_B$, making $k_{\text{DG}}$ a strict chain map. Compatibility with the Mukai pairing reduces to standard Koszul sign factorizations across tensor factors.  

When $A$ and $B$ are proper $A_\infty$-algebras, the multiplication $\emm_n^\otimes$ on $A \otimes_\infty B$ is defined via operadic sums over planar binary trees $G_n^{\text{bin}}$ as described in Section \ref{sec-tensorprod}. This introduces a major technical challenge, namely, the classical map $k_{\text{DG}}$ fails to be a chain map because higher multiplication maps $\emm_n$ ($n \ge 3$) generate non-trivial boundary terms under $\bee$. We construct an extended shuffle map $k$ that incorporates these combinatorial tree structures. Furthermore, we show that $k$ is compatible with the chain-level Mukai pairing. This requires demonstrating that higher tree operations cancel in pairs along internal edges when evaluating the trace. This mechanism is not present in the DG setting where higher $\emm_n$ operations vanish.

\begin{definition}
    Let $\AAA, \BB$ be two $A_\infty$ algebras, and let $(C_*(\AAA), \bee), (C_*(\BB), \bee)$ be their respective Hochschild chain complexes. Denote $\bfa = a_0| a$ where $a=a_1\otimes ...\otimes a_n \in \AAA^{\otimes n}$ and $\bfb =b_0|b$ where $b=b_1\otimes ...\otimes b_m \in \BB^{\otimes m}$. Define $
    k: C_*(\AAA) \otimes C_*(\BB) \rightarrow C_{*}(\AAA \otimes_\infty \BB) $ by 
    \begin{equation}
        k(\bfa \otimes \bfb) = \sum_{(n,m)-shuffles} (-1)^{|b_0||a|'} a_0\otimes b_0 | \operatorname{sh} (\bara;\barb) \label{kunnethmap_defn}
    \end{equation}
    
    where $\bara = a_1\otimes \ub , ..., a_n \otimes \ub$ and $\barb = \ua \otimes b_1, ..., \ua\otimes b_m$.
\end{definition}
\begin{proposition}
    \label{lem: kb=bk}
    The map $k$ defined above is a quasi-isomorphism of complexes. 
\end{proposition} 

\begin{proof}
We first show that it is a chain map. That is: $\bee \circ k = k \circ (\bee\otimes \text{id} + \text{id} \otimes \bee)$, where $\bee$ is the Hochschild boundary map.

On the first term of $\bee$ we have
\begin{align}
    \bee_0 ( k(\bfa, \bfb))  &=  \sum (-1)^{\alpha+\gamma} a_0 \otimes b_0 | \operatorname{sh}(\bara; \barb)^{(1)} \emm(\operatorname{sh}(\bara; \barb)^{(2)}) \operatorname{sh}(\bara; \barb)^{(3)}  \notag\\
        &=  \sum (-1)^{\alpha+\beta+\gamma} a_0 \otimes b_0 | \operatorname{sh}(\bara^{(1)}; \barb^{(1)}) \otimes \emm \left(\operatorname{sh}(\bara^{(2)}; \barb^{(2)})\right) \otimes \operatorname{sh}(\bara^{(3)}; \barb^{(3)}) \label{eqlem21}\\
        &=  \sum (-1)^{\alpha+\beta+\gamma} a_0 \otimes b_0 | \operatorname{sh}(\bara^{(1)}; \barb^{(1)}) \otimes \emm (\bara^{(2)}) \otimes \operatorname{sh}(\bara^{(3)}; \barb^{(3)}) \notag\\
        &\quad+ \  \sum (-1)^{\alpha+\beta+\gamma} a_0 \otimes b_0 | \operatorname{sh}(\bara^{(1)}; \barb^{(1)}) \otimes \emm (\barb^{(2)})\otimes \operatorname{sh}(\bara^{(3)}; \barb^{(3)}) \label{eqprop18}\\
        &=  \sum (-1)^{\alpha+\beta +\gamma+\delta} a_0 \otimes b_0 | \operatorname{sh}(\bara^{(1)}, \emm (\bara^{(2)}), \bara^{(3)} ; \barb) \notag\\
        &\quad + \ \sum (-1)^{\alpha+\beta +\gamma+\phi} a_0 \otimes b_0 | \operatorname{sh}(\bara; \barb^{(1)}, \emm (\barb^{(2)}), \barb^{(3)}) \notag\\
     & = k(\bee_0(\bfa); \bfb) + (-1)^{|\bfa|} k(\bfa;\bee_0( \bfb)) \notag
\end{align}
where equations (\ref{eqlem21}) and (\ref{eqprop18}) follow from Lemma \ref{lemma21} and Proposition \ref{prop18} respectively. The signs are given by $\alpha = |b_0||a|'$, $\beta= |b^{(1)}|'(|a^{(2)}|'+|a^{(3)}|') + |b^{(2)}|'|a^{(3)}|'$, $\gamma = |a_0|+| b_0|+1 + |a^{(1)}|'+|b^{(1)}|'$, $\delta=|b^{(1)}|'(|a^{(2)}|'+1+|a^{(3)}|')$, $\phi=|a^{(3)}|' (|b^{(2)}|'+1+|b^{(1)}|').$\\

Similarly, on the second term of $\bee$, we have
\begin{align}
    \bee_1 ( k(\bfa, \bfb))  &= \sum (-1)^{\alpha+\eta} \emm(\operatorname{sh}(\bara; \barb)^{(3)}, a_0 \otimes b_0, \operatorname{sh}(\bara; \barb)^{(1)}) \ | \ \operatorname{sh}(\bara; \barb)^{(2)}  \notag\\
        &= \sum (-1)^{\alpha+\beta+ \eta} \emm \left(\operatorname{sh}(\bara^{(3)}; \barb^{(3)}), a_0 \otimes b_0, \operatorname{sh}(\bara^{(1)}; \barb^{(1)})\right) | \operatorname{sh}(\bara^{(2)}; \barb^{(2)}) \label{eqlem21v2}\\
        &= \sum (-1)^{\alpha+\beta+ \eta} \emm \left(\bara^{(3)}, a_0 \otimes b_0, \bara^{(1)}\right) | \operatorname{sh}(\bara^{(2)}; \barb^{(2)}) \notag\\
        & \quad + \sum (-1)^{\alpha+\beta+ \eta} \emm \left(\barb^{(3)}, a_0 \otimes b_0, \barb^{(1)}\right) | \operatorname{sh}(\bara^{(2)}; \barb^{(2)}) \label{eqprop18v2}\\
        &= \sum (-1)^{\alpha+\beta+ \eta + \theta_1} \emm \left(\bara^{(3)}, a_0 , \bara^{(1)}\right)\otimes b_0 | \operatorname{sh}(\bara^{(2)}; \barb^{(2)}) \notag\\
        & \quad + \sum (-1)^{\alpha+\beta+ \eta + \theta_2} a_0 \otimes \emm \left(\barb^{(3)}, b_0, \barb^{(1)}\right) | \operatorname{sh}(\bara^{(2)}; \barb^{(2)}) \notag\\
        & = k(\bee_1(\bfa); \bfb) + (-1)^{|\bfa|} k(\bfa;\bee_1( \bfb)) 
\end{align}
where equations (\ref{eqlem21v2}) and (\ref{eqprop18v2}) follow from Lemma \ref{lemma21} and Proposition \ref{prop18} respectively.

The signs are given by $\eta= (|a^{(3)}|' + |b^{(3)}|')(|a^{(2)}|'+|b^{(2)}|'+|a^{(1)}|'+|b^{(1)}|' +|a_0|' + |b_0|)$, $\theta_1 = |b_0||b^{(1)}|'$, $\theta_2 = |a_0|(|b^{(3)}|' +1)$.
Hence, $k$ is a chain map.\\

To show that it induces an isomorphism in homology, consider the length filtration on the chain complex of $\AAA$ defined by $$\mathcal{F}^lC_*(\AAA) = \bigoplus_{n\leq l} \AAA[1]\otimes \AAA[1]^{\otimes n}.$$ Clearly, the Hochschild boundary map $\bee$ restricts to a boundary map on each filtered piece making it a chain complex for each $l$. Moreover we have the inclusion $$\ldots \subset \mathcal{F}^{l-1}C_*(\AAA) \subset \mathcal{F}^lC_*(\AAA) \subset \mathcal{F}^{l+1}C_*(\AAA) \subset \ldots \qquad \text{and} \qquad C_*(A)=\bigcup_l \mathcal{F}^l C_*(A)$$
Taking homology of the quotient $E^0 = \mathcal{F}^*C_*(\AAA)/\mathcal{F}^{*-1}C_*(\AAA)$, with respect to the induced differential gives the $E^1$-page as the algebra complex $$E^1 = \left( \bigoplus_l H^*(\AAA)^{\otimes l}, \bee_{\emm_2}\right)$$ where $H^*(\AAA)$ is the cohomology of $\AAA$ with differential $\emm_1$, and $\bee_{\emm_2}$ is the Hochschild differential induced by $\emm_2$ only. Taking homology now gives the Hochschild homology of the associative algebra above. For two $A_\infty$-algebras $\AAA, \BB$ we have the map on $E^2$-pages $$[k]: HH_*(H^*(\AAA), \emm_2^\AAA) \otimes HH_*(H^*(\BB), \emm_2^\BB) \longrightarrow HH_*(H^*(\AAA\otimes_\infty\BB), \emm_2^\otimes)$$ given by the shuffle map which is known to be an isomorphism for associative algebras (see for example \parencite[Theorem 0.1.2]{tsygan2004cyclic}, \parencite[Theorem 4.2.5]{loday2013cyclic}). Hence, by the spectral sequence comparison theorem (\parencite[Theorem 5.2.12]{weibel1994introduction}), the result follows.
\end{proof}

\begin{theorem} \label{thm: mukai}
Assume $A$ and $B$ are finite-dimensional.\footnote{one can always do this by going to the minimal model since $A$ and $B$ are proper.} 
The K\"unneth quasi-isomorphism $$k: C_*(A) \otimes C_*(B) \longrightarrow C_*(A \otimes_{\infty} B)$$ respects the chain-level Mukai pairings, that is, for all $\bfa^1, \bfa^2 \in C_*(A)$ and $\bfb^1, \bfb^2 \in C_*(B)$, the chain-level Mukai pairings satisfy:
\begin{equation}
    \left\langle  k(\bfa^1; \bfb^1),  k(\bfa^2; \bfb^2)\right\rangle_{Muk}= (-1)^{|\bfb^1|'|\bfa^2|'}\langle \bfa^1, \bfa^2 \rangle_{Muk} \cdot \langle \bfb^1, \bfb^2 \rangle_{Muk}. \label{eq:muk_k}
\end{equation}
\end{theorem}

\begin{proof} 
Let $\bfa^i = a_0^i\ | \ a_1^i\otimes...\otimes  a_n^i \in C_*(\AAA)$, for $i=1,2$ and similarly define $ \bfb^j \in  C_*(\BB)$ for $j=1,2.$
    \begin{align}
        k(\bfa^1; \bfb^1) &= \sum (-1)^{|b_0^1||a^1|'} a_0^1 \otimes b_0^1 | \operatorname{sh}(\bara^1;\barb^1)^{(1)}, \operatorname{sh}(\bara^1;\barb^1)^{(2)}, \operatorname{sh}(\bara^1;\barb^1)^{(3)} \notag \\
        =& \sum (-1)^{|b_0^1||a^1|'+\Delta^1} a_0^1 \otimes b_0^1 | \operatorname{sh}(\bara^{1(1)};\barb^{1(1)}), \operatorname{sh}(\bara^{1(2)};\barb^{1(2)}), \operatorname{sh}(\bara^{1(3)};\barb^{1(3)}) \\
        k(\bfa^2; \bfb^2) &= \sum (-1)^{|b_0^2||a^2|'}a_0^2 \otimes b_0^2 | \operatorname{sh}(\bara^2;\barb^2)^{(1)}, \operatorname{sh}(\bara^2;\barb^2)^{(2)}, \operatorname{sh}(\bara^2;\barb^2)^{(3)} \notag\\
        =& \sum (-1)^{|b_0^2||a^2|'+\Delta^2} a_0^2 \otimes b_0^2 | \operatorname{sh}(\bara^{2(1)};\barb^{2(1)}), \operatorname{sh}(\bara^{2(2)};\barb^{2(2)}), \operatorname{sh}(\bara^{2(3)};\barb^{2(3)})
    \end{align}
    where $$\Delta^1= |a^{1(2)}|'|b^{1(1)}|'+|a^{1(3)}|'|b^{1(2)}|' + |a^{1(3)}|'|b^{1(1)}|' \text{ and } \Delta^2 = |a^{2(2)}|'|b^{2(1)}|' + |a^{2(3)}|'|b^{2(2)}|' + |a^{2(3)}|'|b^{2(1)}|'.$$
    
Applying the Mukai pairing, we obtain
\begin{align}
    \left\langle  k(\bfa^1; \bfb^1),  k(\bfa^2; \bfb^2)\right\rangle_{Muk} &= \notag\\
    \sum (-1)^{|b_0^1||a^1|'+ |b_0^2||a^2|'+\Delta^1 +\Delta^2+\#}& \ tr[
     c_\AAA \otimes c_\BB \mapsto  \emm_*^\otimes ( \operatorname{sh}(\bara^{1(3)};\barb^{1(3)}), a_0^1\otimes b_0^1 , \operatorname{sh}(\bara^{1(1)};\barb^{1(1)}), \notag \\
     \emm_*^\otimes(\operatorname{sh}(\bara^{1(2)};\barb^{1(2)}), c_\AAA &\otimes c_\BB, \operatorname{sh}(\bara^{2(3)};\barb^{2(3)}), a_0^2\otimes b_0^2, \operatorname{sh}(\bara^{2(1)};\barb^{2(1)}), \operatorname{sh}(\bara^{2(2)};\barb^{2(2)})]
\end{align}
where $\#$ is the corresponding sign of the Mukai pairing as in equation (\ref{eq:Mukaipairing}), given explicitly by
\begin{align}
    \# =  1 + (|c_A|+&|c_B|)(|a_0^2|'+|b_0^2|'+ |a^{2(1)}|'+|a^{2(2)}|'+|a^{2(3)}|'+|b^{2(1)}|'+|b^{2(2)}|'+|b^{2(3)}|') \notag\\
    &+|a_0^1|+|b_0^1|+1+|a^{1(1)}|'+|b^{1(1)}|'+|a^{1(3)}|'+|b^{1(3)}|'\notag\\
    +(|a^{1(3)}|'+&|b^{1(3)}|')(|a_0^1|+|b_0^1|+1+|a^{1(1)}|'+|b^{1(1)}|'+|a^{1(2)}|'+|b^{1(2)}|')\notag\\
    +(|a^{2(3)}|'+&|b^{2(3)}|')(|a_0^2|+|b_0^2|+1+|a^{2(1)}|'+|b^{2(1)}|'+|a^{2(2)}|'+|b^{2(2)}|') \label{signofmuk}
\end{align}

Applying Theorem \ref{thm:gencommutingsub} on the inner $\emm_*^\otimes$ we obtain \begin{align}
     \sum (-1)^{* +\epsilon_{in}} \ tr[ c_\AAA \otimes c_\BB \mapsto &\emm_*^\otimes ( \operatorname{sh}(\bara^{1(3)};\barb^{1(3)}), a_0^1\otimes b_0^1 , \operatorname{sh}(\bara^{1(1)};\barb^{1(1)}), \alpha \otimes \beta , \operatorname{sh}(\bara^{2(2)};\barb^{2(2)})] \label{eq31muk}
\end{align}
where $*=|b_0^1||a^1|'+ |b_0^2||a^2|'+\Delta^1+\Delta^2+\#$, for $\#$ given in (\ref{signofmuk}) above; $\alpha = \emm_*^\AAA(a^{1(2)}, c_A, a^{2(3)}, a_0^2, a^{2(1)})$; $\beta = \emm_*^\BB(b^{1(2)}, c_\BB, b^{2(3)}, b_0^2, b^{2(1)})$, and $\epsilon_{in}$ is the corresponding sign of the unique tree in Theorem \ref{thm:gencommutingsub}, explicitly computed as
\begin{align}
    \epsilon_{in}=&(\#a^{2(3)}+\#a^{2(1)})(\#b^{1(2)}+\#b^{2(3)})+\#b^{2(1)}(\#a^{1(2)}+\#a^{2(3)}+\#a^{2(1)})\notag\\
    +|b_0^2||a^{2(1)}| &+ (|c_B|+|b^{2(3)}|)(|a^{2(3)}|+|a_0^2|+ |a^{2(1)}|) +|b^{1(2)}|(|c_A|+|a^{2(3)}|+|a_0^2|+|a^{2(1)}|) \notag\\
    +&\#b^{2(1)}|a^{1(2)}| + (\#b^{1(2)}+\#b^{2(3)}+\#b^{2(1)})(|a^{2(3)}|+|a^{2(1)}|+|c_A|+|a_0^2|)\notag\\
    &+(\#a^{2(3)}+\#a^{2(1)})(|b^{1(2)}|+|c_B|+|b^{2(3)}|) + (\#a^{2(1)}-1)|b_0^2|
    \label{eqsignofmytreein}
\end{align}

Further applying Theorem \ref{thm:gencommutingsub} on the $\emm_*^\otimes$ in equation (\ref{eq31muk}), and observing that the trace of a tensor product is the product of traces, we obtain \begin{align}
     \sum (-1)^{* +\epsilon_{in}+\epsilon_{out}}\ tr[ c_\AAA \otimes c_\BB \mapsto \emm_*^\AAA (a^{1(3)},  a_0^1,  a^{1(1)},  &\alpha , a^{2(2)} )\otimes \emm_*^\BB (b^{1(3)}, b_0^1, b^{1(1)}, \beta , b^{2(2)}) \notag\\
    =(-1)^{|\bfb^1|'|\bfa^2|'}\ \sum (-1)^{\#_A }\ tr [c_\AAA \mapsto &\emm_*^\AAA (a^{1(3)}, a_0^1, a^{1(1)}, \alpha , a^{2(2)} )]  \notag \\ \cdot \
\sum (-1)^{\#_B} \ 
tr [c_\BB \mapsto  &\emm_*^\BB (b^{1(3)},  b_0^1, b^{1(1)}, \beta , b^{2(2)} )]  \label{eqtraceof prod}\\
    = (-1)^{|\bfb^1|'|\bfa^2|'}\langle \bfa^1, \bfa^2 \rangle_{Muk} \cdot \langle \bfb^1, \bfb^2 \rangle_{Muk} & \qquad 
\end{align}
where
\begin{align}
    \#_A = & 1+|c_A|(|a_0^2|'+|a^{2(1)}|'+|a^{2(2)}|'+|a^{2(3)}|') + |a_0^1|'+|a^{1(1)}|'+|a^{1(3)}|' \notag\\
    & |a^{1(3)}|'(|a_0^1|'+|a^{1(1)}|'+|a^{1(2)}|') + |a^{2(3)}|'(|a_0^2|'+|a^{2(1)}|'+|a^{2(2)}|'), \notag
\end{align}
\begin{align}
    \#_B = & 1+|c_B|(|b_0^2|'+|b^{2(1)}|'+|b^{2(2)}|'+|b^{2(3)}|') + |b_0^1|'+|b^{1(1)}|'+|b^{1(3)}|' \notag\\
    & |b^{1(3)}|'(|b_0^1|'+|b^{1(1)}|'+|b^{1(2)}|') + |b^{2(3)}|'(|b_0^2|'+|b^{2(1)}|'+|b^{2(2)}|'), \notag
\end{align}

and
\begin{align}
    \epsilon_{out}=&(\#a^{1(1)}+\#a^{2(2)})(\#b^{1(3)}+\#b^{1(1)})+\#b^{2(2)}(\#a^{1(3)}+\#a^{1(1)}+\#a^{2(2)})\notag\\
    +|\beta||a^{2(2)}| + &(|b_0^1|+|b^{1(1)}|)(|a^{1(1)}|+|\alpha|+ |a^{2(2)}|) +|b^{1(3)}|(|a_0^1|+|a^{1(1)}|+|\alpha|+|a^{2(2)}|) \notag\\
    +&\#b^{2(2)}|a^{1(3)}| + (\#b^{1(3)}+\#b^{1(1)}+\#b^{2(2)})(|a^{1(1)}|+|a^{2(2)}|+|a_0^1|+|\alpha|)\notag\\
    &+(\#a^{1(1)}+\#a^{2(2)})(|b^{1(3)}|+|b_0^1|+|b^{1(1)}|) + (\#a^{2(2)}-1)|\beta|
    \label{eqsignofmytreeout}
\end{align}
with $\#a^{i(j)}=m$ if $a^{i(j)}=a^{i(j)}_1\otimes \cdots \otimes a^{i(j)}_m$, and $|a^{i(j)}|'=|a^{i(j)}|+\#a^{i(j)}$, and similarly defined for $\#b^{i(j)}$.

\end{proof}

\begin{corollary}
    The Mukai pairing is non-degenerate on $\AAA \otimes_\infty \BB$.
\end{corollary}
    
\subsection{Negative cyclic homology} 
We extend Shklyarov \parencite[Section 2]{shklyarov2016hodge}’s construction of the non-commutative Hodge-theoretic K\"unneth map from DG algebras to the $A_\infty$ setting. In particular, the main features of his construction carry over, including the compatibility of the K\"unneth map with the $u$-connection. In our setting, we establish in addition that the cyclic K\"unneth map preserves the chain-level higher residue pairings, hence, a morphism of formal EP structures. We use the cyclic shuffle map to define a K\"unneth map for the negative cyclic complexes as follows.
\begin{definition}
    Let $\AAA, \BB$ be two $A_\infty$ algebras, and let $(CC_*^-(\AAA), \bee +u\Bee), (CC_*^-(\BB), \bee +u\Bee)$ be their respective negative cyclic chain complexes. Denote $\bfa = a_0| a$ where $a=a_1, ..., a_n$ and $\bfb =b_0|b$ where $b=b_1,...,b_m$. Define $
    K: C_*(\AAA) \otimes C_*(\BB) \rightarrow C_{*}(\AAA\otimes_\infty \BB) $ by 
    \begin{equation}
        K(\bfa \otimes \bfb) = \sum \ua\otimes \ub | \operatorname{Sh}(\bara_0, \bara_1,\ldots, \bara_n;\barb_0,\barb_1,\ldots,\barb_m) \label{cycKunneth_defn}
    \end{equation}
    
    where $\bara_i = a_i\otimes \ub,$ and $\barb_j = \ua\otimes b_j$, and extend $u$-linearly.
\end{definition}

\begin{proposition}\label{cycqism}
     The 
map $k+uK: CC_*^-(A) \otimes_{k[[u]]} CC_*^-(B) \rightarrow CC_*^-(A\otimes_\infty B)$ is a quasi-isomorphism of complexes.
\end{proposition}
   
\begin{proof}
We first show that it is a chain map. That is: \begin{equation}
       \left( \bee + u\Bee\right) \left(k+uK\right)(\bfa; \bfb) = \left(k+uK\right)\left( (\bee + u\Bee)(\bfa); \bfb \right) + (-1)^{|\bfa|} \left(k+uK\right)\left( \bfa; (\bee + u\Bee)(\bfb) \right)
    \end{equation} 
Comparing coefficients in $u$, this is equivalent to the three equations: 
\begin{align}
    u^0:& \ \bee \ k(\bfa; \bfb) = k\left(\bee (\bfa); \bfb\right) + (-1)^{|\bfa|} k\left(\bfa;\bee(\bfb)\right) \label{a}\\
    u: \ \Bee \ k(\bfa; \bfb) + \bee K(\bfa; & \bfb) =  k\left(\Bee (\bfa); \bfb\right) + K\left(\bee (\bfa); \bfb\right) + (-1)^{|\bfa|} k\left(\bfa;\Bee(\bfb)\right) + (-1)^{|\bfa|} K\left(\bfa;\bee(\bfb)\right) \notag\\
    u^2:& \ \Bee \ K(\bfa; \bfb) = K\left(\Bee (\bfa); \bfb\right) + (-1)^{|\bfa|} K\left(\bfa;\Bee(\bfb)\right) \label{b}
    \end{align}
Equation (\ref{a}) follows from Proposition \ref{lem: kb=bk} since $k$ is a chain map. The left hand side of equation (\ref{b}) vanishes since our chains are reduced. The terms on the right hand side also vanish individually for the same reason since
\begin{align*}
    K\left(\Bee (\bfa); \bfb\right) &= \sum \pm K \left( \ 1 \ | a^{(2)}, a_0, a^{(1)}; \ \barb_0, ...,\barb_m\right) \\
    &= \sum \pm \ 1\otimes1 \ | \ \operatorname{Sh}(1\otimes1, \bara^{(2)}, \bara_0, \bara^{(1)}; \ \barb_0, ...,\barb_m) = 0\\
\text{and } K (\bfa; \Bee(\bfb)) &= \sum \pm \ 1\otimes1 \ | \ \operatorname{Sh}(\bara_0, ...,\bara_n; \ 1\otimes1, \barb^{(2)}, \barb_0, \barb^{(1)}) = 0
\end{align*}
For the coefficient of $u$, this follows from Lemma \ref{lem:AppendixA} below. 
\begin{lemma}[See Appendix \ref{secA.1} for a proof]\label{lem:AppendixA}
    $$\Bee \ k(\bfa; \bfb) + \bee K(\bfa; \bfb) =  k\left(\Bee (\bfa); \bfb\right) + K\left(\bee (\bfa); \bfb\right) + (-1)^{|\bfa|} k\left(\bfa;\Bee(\bfb)\right) + (-1)^{|\bfa|} K\left(\bfa;\bee(\bfb)\right). $$
\end{lemma}

\noindent
To show that it induces an isomorphism in homology, consider the $u$-filtration on the negative cyclic complex of $\AAA$ defined by $$\mathcal{F}^lCC_*^-(\AAA) =  u^l CC_*^-(\AAA)=u^l C_*(A)[[u]]$$ The boundary map $\bee +u\Bee$ restricts to a boundary map on each filtered piece and we have the reverse inclusion $$\ldots \supset \mathcal{F}^{l-1}CC_*^-(\AAA) \supset \mathcal{F}^lCC_*^-(\AAA) \supset \mathcal{F}^{l+1}CC_*^-(\AAA) \supset \ldots, \qquad \text{and }\ \bigcap_l \mathcal{F}^lCC_*^-(A)=0$$
Taking the quotient $E^0 = \mathcal{F}^*CC_*^-(\AAA)/\mathcal{F}^{*+1}CC_*^-(\AAA)$, we have the induced differential as $\bee$ since the term $u\Bee$ increases filtration and hence killed in the quotient. Taking homology we obtain the $E^1$-page of the spectral sequence for the Hochschild homology of $\AAA$. For two $A_\infty$-algebras $\AAA, \BB$ we have the map on $E^1$-pages $$[k]: HH_*(\AAA) \otimes HH_*(\BB) \longrightarrow HH_*(\AAA\otimes_\infty\BB)$$ which is an isomorphism by Proposition \ref{lem: kb=bk}. Hence, by the spectral sequence comparison theorem \parencite[Theorem 5.2.12]{weibel1994introduction}, the result follows.
\end{proof}
We therefore prove our main theorem.
\begin{theorem} \label{mainthm}
     Let $A$ and $B$ be finite dimensional $A_\infty$-algebras. The cyclic Kunneth quasi-isomorphism induces an isomorphism of formal EP-structures on negative cyclic homology.
\end{theorem}
   
\begin{proof} We first show that the cyclic K\"unneth map respects the higher residue pairings. Let $\bfa^i = (a_0^i | a_1^i,..., a_n^i) \in C_*(\AAA)$, for $i=1,2$ and similarly define $ \bfb^j=(b_0^j | b_1^j,..., b_m^j) \in  C_*(\BB)$ for $j=1,2.$ 
\begin{align}
    \left\langle  (k+uK)(\bfa^1; \bfb^1),  (k+uK)(\bfa^2; \bfb^2)\right\rangle_{hres} &= \left\langle  k(\bfa^1; \bfb^1),  k(\bfa^2; \bfb^2)\right\rangle_{Muk} \\
    -u\left\langle  k(\bfa^1; \bfb^1),  K(\bfa^2; \bfb^2)\right\rangle_{Muk} &+ u\left\langle  K(\bfa^1; \bfb^1),  k(\bfa^2; \bfb^2)\right\rangle_{Muk} \label{eqcoeffu}\\
    -u^2\left\langle  K(\bfa^1; \bfb^1),  K(\bfa^2; \bfb^2)\right\rangle_{Muk}& 
\end{align}
The coefficient of $u^2$ is readily seen to be zero since
\begin{align}
    \left\langle  K(\bfa^1; \bfb^1),  K(\bfa^2; \bfb^2)\right\rangle_{Muk} &= \sum \left\langle  1\otimes 1|\operatorname{Sh}(\underline{\bfa}^1; \underline{\bfb}^1),  1\otimes 1|\operatorname{Sh}(\underline{\bfa}^2; \underline{\bfb}^2)\right\rangle_{Muk}\notag\\
    =\sum \pm \ tr[c_A\otimes c_B \mapsto &\emm_*^\otimes(\operatorname{Sh}(\underline{\bfa}^1; \underline{\bfb}^1)^{(3)}, 1\otimes 1, \operatorname{Sh}(\underline{\bfa}^1; \underline{\bfb}^1)^{(1)}, \notag\\
    \emm_*^\otimes(\operatorname{Sh}(\underline{\bfa}^1; \underline{\bfb}^1)^{(2)},& c_A\otimes c_B, \operatorname{Sh}(\underline{\bfa}^2; \underline{\bfb}^2)^{(3)}, 1\otimes 1, \operatorname{Sh}(\underline{\bfa}^2; \underline{\bfb}^2)^{(1)}), \operatorname{Sh}(\underline{\bfa}^2; \underline{\bfb}^2)^{(2)})]
\end{align}
By the commuting subalgebras relation of Definition \ref{commutingsubalgebradefn}, the inner $\emm_*^\otimes$ is only non-zero when $*=2$, forcing all the inner $\operatorname{Sh}$ terms to be empty. The same situation applies to the outer $\emm_*^\otimes$ except that all the $\operatorname{Sh}$ terms cannot be simultaneously empty. Thus, the outer $\emm_*^\otimes$ evaluates to zero, and we are taking the (sum of) traces of the zero map, hence zero.\\

For the coefficient of $u$, the terms vanish independently. We show this for the first term, a very similar computation shows this for the second term.
\begin{align}
    \left\langle  k(\bfa^1; \bfb^1),  K(\bfa^2; \bfb^2)\right\rangle_{Muk} &= \sum \pm \ tr[
     c_\AAA \otimes c_\BB \mapsto  \emm_*^\otimes ( \operatorname{sh}(\bara^1;\barb^1)^{(3)}, a_0^1\otimes b_0^1 , \operatorname{sh}(\bara^1;\barb^1)^{(1)}, \notag \\
     \emm_*^\otimes(\operatorname{sh}(\bara^1;\barb^1)^{(2)}, c_\AAA &\otimes c_\BB, \operatorname{Sh}(\bara^2;\barb^2)^{(3)}, 1\otimes 1, \operatorname{Sh}(\bara^2;\barb^2)^{(1)}), \operatorname{Sh}(\bara^2;\barb^2)^{(2)})]. \label{equ^1}
\end{align}
Now the inner $\emm_*^\otimes$ is non-zero except when $*=2$ by Definition \ref{commutingsubalgebradefn}. So we obtain $\emm_2^\otimes(c_\AAA \otimes c_\BB , 1\otimes 1) = \pm\ c_\AAA\otimes c_\BB$ forcing all the inner shuffle and cyclic shuffle terms to be empty. Equation (\ref{equ^1}) now becomes
\begin{align}
= \sum \pm \ tr[
     c_\AAA \otimes c_\BB \mapsto\emm_*^\otimes ( \operatorname{sh}(\bara^1;\barb^1)^{(2)}, a_0^1\otimes b_0^1 , \operatorname{sh}(\bara^1;\barb^1)^{(1)}, c_\AAA \otimes c_\BB, \operatorname{Sh}(\bara^2;\barb^2)].
\end{align}
From Theorem \ref{thm:gencommutingsub}, the $\operatorname{Sh}(\bara^2;\barb^2)$ term must be ordered such that all $\bara's$ appear before any $\barb's$, otherwise we get zero. But by the definition of cyclic shuffles 
we must have $\barb_0^2$ to the left of $\bara_0^2$. Hence, zero by Theorem \ref{thm:gencommutingsub}.

So we have
\begin{equation}
    \left\langle  (k+uK)(\bfa^1; \bfb^1),  (k+uK)(\bfa^2; \bfb^2)\right\rangle_{hres} = \langle \bfa^1, \bfa^2 \rangle_{Muk} \cdot \langle \bfb^1, \bfb^2 \rangle_{Muk}=\langle \bfa^1, \bfa^2 \rangle_{hres} \cdot \langle \bfb^1, \bfb^2 \rangle_{hres} \label{eq:hresk+uK}
\end{equation}
By sesqui-linearity, the general case follows.

    Next, we need to show that $k+uK: C_*(\AAA)[[u]] \otimes C_*(\BB)[[u]] \to C_*(\AAA\otimes_\infty \BB)[[u]]$ is compatible with the $u$-connection defined as \begin{equation}
    \nabla_{\frac{d}{du}} = \frac{d}{du} + \frac{\bee\{\emm'\}}{2u^2} + \frac{\Bee\{\emm'\} - \Gamma}{2u}
\end{equation}
on homology. We compute
\begin{align*}
    \nabla_{\frac{d}{du}} \circ (k+uK) & (u^p \bfa \otimes \bfb) = \left( \frac{d}{du} + \frac{\bee\{\emm'\}}{2u^2} + \frac{\Bee\{\emm'\} - \Gamma}{2u} \right) \circ \left(u^pk(\bfa ; \bfb) + u^{p+1} K(\bfa ; \bfb)\right)\\
    & = u^p \left( (p+1) K(\bfa ; \bfb) + \frac{1}{2} \Bee\{\emm'\} \circ K(\bfa ; \bfb) -\frac{1}{2} \Gamma \circ K(\bfa ; \bfb) \right) \\
    & + u^{p-1} \left( p k(\bfa ; \bfb) + \frac{1}{2} \bee\{\emm'\} \circ K(\bfa ; \bfb) +\frac{1}{2} \Bee\{\emm'\} \circ k(\bfa; \bfb) - \frac{1}{2} \Gamma\circ k(\bfa ; \bfb) \right) \\
    & + u^{p-2} \left( \frac{1}{2} \bee\{\emm'\} \circ k(\bfa; \bfb) \right)
\end{align*}
On the other hand, we have
\begin{align*}
   (k+uK) & \circ \left( \nabla_\frac{d}{du} \otimes \text{id} + \text{id} \otimes \nabla_\frac{d}{du} 
  \right) (u^p \bfa \otimes \bfb)  = (k+uK) (pu^{p-1} \bfa \otimes \bfb ) \\
   & + (k+uK) \left(\frac{u^{p-2}}{2} (\bee\{\emm'\}(\bfa)\otimes \bfb + \bfa \otimes \bee\{\emm'\}(\bfb) )  \right) \\
   & + (k+uK) \left(\frac{u^{p-1}}{2} \left(\Bee\{\emm'\}(\bfa)\otimes \bfb + \bfa \otimes \Bee\{\emm'\}(\bfb) - \Gamma (\bfa) \otimes \bfb - \bfa \otimes \Gamma (\bfb) \right) \right) \\
   & = u^p \left(p K(\bfa ; \bfb) + \frac{1}{2} K(\Bee\{\emm'\}(\bfa) ; \bfb) + \frac{1}{2} K(\bfa ; \Bee\{\emm'\}(\bfb)) -\frac{1}{2} K(\Gamma (\bfa) ; \bfb) -\frac{1}{2} K(\bfa ; \Gamma(\bfb)) \right) \\
   & + u^{p-1} \bigg(p k(\bfa ; \bfb) + \frac{1}{2} K(\bee\{\emm'\}(\bfa) ; \bfb) + \frac{1}{2} K(\bfa ; \bee\{\emm'\}(\bfb)) + \frac{1}{2} k(\Bee\{\emm'\}(\bfa) ; \bfb) \\
   & \qquad \qquad + \frac{1}{2} k(\bfa ; \Bee\{\emm'\}(\bfb)) -\frac{1}{2} k(\Gamma (\bfa) ; \bfb) -\frac{1}{2} k(\bfa ; \Gamma(\bfb)) \bigg) \\
   & + u^{p-2} \bigg( \frac{1}{2} k(\bee\{\emm'\}(\bfa) ; \bfb) + \frac{1}{2} k(\bfa ; \bee\{\emm'\}(\bfb)) \bigg)
\end{align*}
Comparing coefficients, the commutator $\displaystyle [ k+uK, \nabla_{\frac{d}{du}}]$ equals 
\begin{align}
    u^p & \left( -K + \frac{1}{2}[ K,\Gamma ] - \frac{1}{2}[ K, \Bee\{\emm'\}] \right)\\
    +u^{p-1} & \left( \frac{1}{2}[k,  \Gamma] - \frac{1}{2}[ k, \Bee\{\emm'\}] - \frac{1}{2}[ K, \bee\{\emm'\}] \right)\\
    +u^{p-2} & \left( -\frac{1}{2}[ k, \bee\{\emm'\}]\right)
\end{align}
where $\bee\{\emm'\}(-) = \bee^{1,1}(\emm', -)$ and $\Bee\{\emm'\}(-) = \Bee^{1,1}(\emm', -)$.

Clearly, $2K = [K,\Gamma]$ since by definition of the operators, we have
\begin{align}
    K(\Gamma\bfa; \bfb) &= -n \sum 1\otimes 1|\operatorname{Sh}(\underline{\bfa};\underline{\bfb})\\
    K(\bfa; \Gamma\bfb) &= -m \sum 1\otimes 1|\operatorname{Sh}(\underline{\bfa};\underline{\bfb})\\
    \Gamma K(\bfa;\bfb) &= -(n+m+1+1) \sum 1\otimes 1|\operatorname{Sh}(\underline{\bfa};\underline{\bfb})
\end{align}
The coefficient of $u^p$ is therefore zero by the above calculation, and because
every term of the composition $\Bee\{\emm'\}\circ K$ and $K\circ \Bee\{\emm'\}$ is zero since we are working in the reduced complex. We show that the coefficient of $u^{p-2}$ is $0$. 

The first term in the commutator is
\begin{align}
    k\circ &(\bee\{\emm'\} \otimes \text{id}  + \text{id} \otimes \bee\{\emm'\} ) (\bfa \otimes \bfb) = k \bigg( (\bee\{\emm'\}(\bfa); \bfb ) + (-1)^{|\bfa|}(\bfa ; \bee\{\emm'\}(\bfb) ) \bigg) \notag\\
    &= k\bigg( \sum (-1)^{s_\AAA} \emm_*^\AAA (a^{(3)}, \emm'_\AAA(a^{(4)} ), a^{(5)},a_0, a^{(1)}) | a^{(2)} ; b_0 | b^{(1)}\bigg) \notag\\
    & \qquad \qquad + (-1)^{|\bfa|}k\bigg( \sum (-1)^{s_\BB}  a_0 | a^{(1)}; \emm_*^\BB (b^{(3)}, \emm'_\BB(b^{(4)}), b^{(5)},b_0,b^{(1)}) | b^{(2)} \bigg) \notag\\
    &= \sum (-1)^{s_\AAA +|b_0||a^{(2)}|'} \emm_*^\AAA (a^{(3)}, \emm'_\AAA(a^{(4)} ), a^{(5)},a_0, a^{(1)})\otimes b_0 | \operatorname{sh}(\bara^{(2)}; \barb) \notag\\
    & \qquad \qquad + (-1)^{|\bfa|}\sum (-1)^{s_\BB+ t|a^{(1)}|'} a_0 \otimes  \emm_*^\BB (b^{(3)}, \emm'_\BB(b^{(4)}), b^{(5)},b_0,b^{(1)}) | sh(\bara; \barb^{(2)})\label{eqstar}
\end{align}
where $$s_\AAA= (|a_0|'+|a^{(1)}|'+|a^{(2)}|')(|a^{(3)}|'+|a^{(4)}|'+|a^{(5)}|') + |a^{(3)}|'$$ 
$$ s_\BB=(|b_0|'+|b^{(1)}|'+|b^{(2)}|')(|b^{(3)}|'+|b^{(4)}|'+|b^{(5)}|') + |b^{(3)}|'$$ 
$$\text{ and } t=|b^{(3)}|'+|b^{(4)}|'+|b^{(5)}|'+|b_0|'+|b^{(1)}|'$$

In a similar fashion, the second term in the commutator gives
\begin{align*}
     &\bee\{\emm'\}  \circ k(\bfa ; \bfb)  =  \bee\{\emm'\} \bigg( \sum (-1)^{\sigma} a_0 \otimes b_0 | \operatorname{sh} (\bara ; \barb)^{(1)}, \operatorname{sh} (\bara ; \barb)^{(2)}, \operatorname{sh} (\bara ; \barb)^{(3)}, \operatorname{sh} (\bara ; \barb)^{(4)}, \operatorname{sh} (\bara ; \barb)^{(5)}\bigg) \\
     &=  \sum (-1)^{\sigma+\Delta + s_{\otimes}} \emm_*^\otimes \left(\operatorname{sh} (\bara^{(3)} ; \barb^{(3)}), \emm'_\otimes(\operatorname{sh} (\bara^{(4)} ; \barb^{(4)})), \operatorname{sh} (\bara^{(5)} ; \barb^{(5)}),a_0 \otimes b_0, \operatorname{sh} (\bara^{(1)} ; \barb^{(1)})\right) | \operatorname{sh} (\bara^{(2)} ; \barb^{(2)})
\end{align*}
Since $\emm_2'=0$, it follows from the commuting subalgebras relations of Definition \ref{commutingsubalgebradefn} that the argument in $\emm'_\otimes$, namely, $\operatorname{sh} (\bara^{(4)} ; \barb^{(4)})$ must be strictly $\bara^{(4)}$ or strictly $\barb^{(4)}$ to be non-trivial. Similarly, the terms in $\emm_*^\otimes$ must be strictly $\bara's$ or strictly $\barb's$ (but with $a_0\otimes b_0$ fixed), otherwise the sum is zero. The last sum above thus decomposes as
\begin{align}
    \sum &(-1)^{\sigma+\Delta + s_{\otimes}} \emm_*^\otimes \left(\bara^{(3)}, \emm'_A(\bara^{(4)}), \bara^{(5)},a_0 \otimes b_0, \bara^{(1)}\right) | \operatorname{sh} (\bara^{(2)} ; \barb^{(2)}) \notag\\
    + &\sum (-1)^{\sigma+\Delta + s_{\otimes}} \emm_*^\otimes \left(\barb^{(3)}, \emm'_B(\barb^{(4)}), \barb^{(5)},a_0 \otimes b_0, \barb^{(1)}\right) | \operatorname{sh} (\bara^{(2)} ; \barb^{(2)}) \notag\\
    = \sum &(-1)^{\sigma+\Delta + s_{\otimes}+|b_0||a^{(1)}|'} \emm_*^A \left(a^{(3)}, \emm'_A(a^{(4)}), a^{(5)},a_0, a^{(1)}\right) \otimes b_0 | \operatorname{sh} (\bara^{(2)} ; \barb) \notag\\
     + \sum& (-1)^{\sigma+\Delta + s_{\otimes}+|a_0|'(|b^{(5)}|'+|b^{(4)}|'+|b^{(3)}|')} a_0 \otimes \emm_*^B \left(b^{(3)}, \emm'_B(b^{(4)}), b^{(5)},b_0, b^{(1)}\right)| \operatorname{sh} (\bara ; \barb^{(2)}) \label{eqstarstar}
\end{align}
where
$$\sigma= |b_0||a|', \qquad \Delta=\sum_{i=1}^4 |b^{(i)}|'\left(\sum_{j>i}^5 |a^{(j)}|'\right)$$
$$s_\otimes= \left(1+ |a_0|+|b_0|+\sum_{i=1}^2 |a^{(i)}|'+|b^{(i)}|'\right)\left(\sum_{i=3}^5 |a^{(i)}|'+|b^{(i)}|'\right)+  |a^{(3)}|'+|b^{(3)}|'$$
The equations (\ref{eqstar}) and (\ref{eqstarstar}) cancel each other, and we have  $[k, \bee\{\emm'\}]=0$. \\

\noindent
Lastly, for the coefficient of $u^{p-1}$, clearly $[k, \Gamma] = 0$. We will show that $[k, \Bee\{\emm'\}] + [K, \bee\{\emm'\}]$ is homotopic to $0$ via the degree $+1$ homotopy $N=H+H': CC_*^-(\AAA)\otimes CC_*^-(\BB) \to CC_*^-(\AAA\otimes_\infty \BB)$, where $H$ is defined as (Cf. \parencite[Proof of Theorem 2]{shklyarov2016hodge})
\begin{align}
    H(\bfa, \bfb) &= \sum_{\barb_0 < \bara_0} (-1)^{@_\AAA^1 + @_\BB^1 + |a^{(2)}|'} \ \ua\otimes \ub \ | \ \operatorname{sh}\left(\bara^{(2)}, \emm'_\AAA(\bara^{(3)}), \bara^{(4)},\bara_0,\bara^{(1)}; \barb^{(2)},\barb_0,\barb^{(1)}\right) \notag \\
    + &\sum_{\barb_0 < \bara_0} (-1)^{@_\AAA^2 + @_\BB^2 + |\bfa|' + |b^{(2)}|'} \ \ua\otimes \ub \ | \ \operatorname{sh}\left(\bara^{(2)},\bara_0,\bara^{(1)}; \barb^{(2)}, \emm'_\BB(\barb^{(3)}),\barb^{(4)},\barb_0,\barb^{(1)}\right) \label{eqH}
\end{align} 
where $$@_\AAA^1=\left(|a_0|'+|a^{(1)}|'\right)\left(\sum_{i=2}^4|a^{(i)}|'\right), \qquad @_B^1=|b^{(2)}|'\left(|b_0|'+|b^{(1)}|'\right)$$
$$@_\AAA^2=|a^{(2)}|'\left(|a_0|'+|a^{(1)}|'\right), \qquad @_B^2=\left(|b_0|'+|b^{(1)}|'\right)\left(\sum_{i=2}^4|b^{(i)}|'\right).$$
The condition $\barb_0<\bara_0$ means $\barb_0$ appears to the left of $\bara_0$. We define $H'$ as $$H':= u^{-1}\left(k+uK\right)\circ \left(\iota\{\emm'_A\}\otimes \Bee\right)$$
So we need to show that $[k, \Bee\{\emm'\}] + [K, \bee\{\emm'\}] = [H+H', \bee + u\Bee]$. The RHS equals $$[H, \bee+u\Bee] + [H', \bee+u\Bee] = [H,\bee] + [H',\bee+u\Bee]$$ since $[H,\Bee]=0$ (as we are in the reduced complex). We will show that $[H',\bee+u\Bee]=-k(\mathcal{L}_{\emm'}\otimes \Bee)$
\begin{align} 
H'\circ ((\bee+u\Bee)\otimes \text{id} + & \text{id} \otimes  (\bee+u\Bee)) \notag \\
= & -u^{-1}(k+uK)(\iota\{\emm'_A\}(\bee+u\Bee)\otimes \text{id})(\text{id} \otimes \Bee) \\
&+ u^{-1}(k+uK)(\iota\{\emm'_A\}\otimes \text{id})(\text{id} \otimes \Bee\bee)\label{eq43}
\end{align}
\begin{align}
    (\bee+u\Bee)\circ H' &=(\bee+u\Bee)u^{-1}\left(k+uK\right)\circ \left(\iota\{\emm'_A\}\otimes \Bee\right)\\
    &=u^{-1}(k+uK)\left( (\bee+u\Bee)(\iota\{\emm'_A\})\otimes \Bee + \iota\{\emm'_A\}\otimes (\bee+u\Bee)\Bee\right)\\
    &=u^{-1}(k+uK)\circ(\bee+u\Bee)(\iota\{\emm'_A\}\otimes \text{id})\circ(\text{id} \otimes \Bee)\\
    &\qquad + u^{-1}(k+uK)(\iota\{\emm'_A\}\otimes \text{id})(\text{id} \otimes \bee\Bee) \label{eq41}
\end{align}
Now, terms in (\ref{eq41}) and (\ref{eq43}) cancel out since $\Bee\bee + \bee\Bee=0$, and we are left with
\begin{align}
    [H', \bee+u\Bee] &= u^{-1}(k+uK)\circ(\left[\bee+u\Bee, \ \iota\{\emm'_A\}\right]\otimes \text{id})(\text{id}\otimes \Bee)\\
    &= -u^{-1}(k+uK)(u\mathcal{L}_{\emm'} \otimes \text{id})(\text{id}\otimes \Bee) \label{eq45}\\
    &= -(k+uK)(\mathcal{L}_{\emm'} \otimes \Bee)\\
    &= -k(\mathcal{L}_{\emm'} \otimes \Bee)
\end{align}
where equation (\ref{eq45}) follows from the Cartan homotopy formula (\ref{cartanmagicform}), and since $[\emm, \emm']=0$

Putting it all together, we need to prove the lemma below.
\begin{lemma}[See Appendix \ref{secA.2} for a proof]\label{appendixlemma}
    $[k, \Bee\{\emm'\}] + [K, \bee\{\emm'\}] = [H,\bee] - k(\mathcal{L}_{\emm'} \otimes \Bee).$
\end{lemma}
The result follows.
\end{proof}

\section{Applications to splittings of the Hodge filtration}
\subsection{Weak Calabi-Yau structure and duality isomorphism}
Let $A$ be finite-dimensional, smooth and proper $A_\infty$ algebra over a field $\mathrm{k}$ of characteristic zero. We will generalize Amorim-Tu \parencite[Theorem 2.4]{amorim2022categorical}'s duality isomorphism $D:HH^*(A)\to HH_*(A)$ for cyclic $A_\infty$ structures to weak proper (right) Calabi-Yau structures. We recall some well-known facts about $A_{\infty}$-bimodules, following \parencite[Section 2.1]{amorim2022morita}.
\begin{definition}
    An $A_{\infty}$-bimodule over $A$ (or $A-A$ bimodule) consists of a graded vector space $P$ and a family of maps
$$
\mathfrak{n}_{r, s}: A[1]^{\otimes r} \otimes P \otimes A[1]^{\otimes s} \rightarrow P[1]
$$
satisfying$$
\begin{aligned}
& \sum(-1)^{\star} \mathfrak{n}_{r, s}\left(\mathbf{x}^{(1)}, \mathfrak{m}_i\left(\mathbf{x}^{(2)}\right), \mathbf{x}^{(3)}, \underline{p}, \mathbf{y}\right) \\
& +\sum(-1)^{\star} \mathfrak{n}_{r, s}\left(\mathbf{x}, \underline{p}, \mathbf{y}^{(1)}, \mathfrak{m}_i\left(\mathbf{y}^{(2)}\right), \mathbf{y}^{(3)}\right) \\
& +\sum(-1)^{\star} \mathfrak{n}_{r, s}\left(\mathbf{x}^{(1)}, \mathfrak{n}_{r, s}\left(\mathbf{x}^{(2)}, \underline{p}, \mathbf{y}^{(1)}\right), \mathbf{y}^{(2)}\right)=0
\end{aligned}
$$
where $\star$ is the sum of the degrees to the left of the second multilinear map. For example, in the second line above, $\star=\left|x_1\right|^{\prime}+\ldots+\left|x_r\right|^{\prime}+\left|y_1\right|^{\prime}+\ldots+\left|y_{j-1}\right|^{\prime}+|p|$.
\end{definition}
$A_{\infty}$-bimodules form a dg-category $[A,A]$. Given $A_{\infty}$-bimodules $P$ and $Q$ an element $\rho \in \operatorname{hom}_{[A,A]}^d(P,Q)$ of degree $d$, called a pre-homomorphism consists of maps
$$
\rho_{r, s}: A[1]^{\otimes r} \otimes P  \otimes A[1]^{\otimes s} \rightarrow Q[d]
$$
The differential in $h o m_{[A,A]}^{\bullet}(P,Q)$ is defined as

$$
\begin{aligned}
(\partial \rho)_{r, s}(\mathbf{x}, \underline{p}, \mathbf{y})= & \sum(-1)^{|\rho| \star} \mathfrak{n}_{r_1, s_1}\left(\mathbf{x}^{(1)}, \underline{\rho_{r_2, s_2}\left(\mathbf{x}^{(2)}, \underline{p}, \mathbf{y}^{(1)}\right)}, \mathbf{y}^{(2)}\right) \\
& +\sum(-1)^{1+|\rho|+\star} \rho_{r_1, s_1}\left(\mathbf{x}^{(1)}, \underline{\mathfrak{n}_{r_2, s_2}\left(\mathbf{x}^{(2)}, \underline{p}, \mathbf{y}^{(1)}\right)}, \mathbf{y}^{(2)}\right) \\
& +\sum(-1)^{1+|\rho|+\star} \rho_{r_1, s}\left(\mathbf{x}^{(1)} \mathfrak{m}_i\left(\mathbf{x}^{(2)}\right) \mathbf{x}^{(3)}, \underline{p}, \mathbf{y}\right) \\
& +\sum(-1)^{1+|\rho|+\star} \rho_{r, s_1}\left(\mathbf{x}, \underline{p}, \mathbf{y}^{(1)} \mathfrak{m}_i\left(\mathbf{y}^{(2)}\right), \mathbf{y}^{(3)}\right)
\end{aligned}
$$

\begin{definition}
    Given $A_{\infty}$-bimodules $P$ and $Q$, a pre-homomorphism $\rho \in h o m_{[A,A]}^{\bullet}(P,Q)$ with $\partial \rho=0$ is called a homomorphism. A homomorphism is called a quasi-isomorphism if $\rho_{0,0}: P \rightarrow Q$ induces an isomorphism on cohomology.
\end{definition}

We define two important bimodules for our work.
\begin{definition}
    Let $A_\Delta$ be the bimodule defined as $A[1]$, with operations
$$
\mathfrak{n}_{r, s}(\mathbf{x}, \underline{p}, \mathbf{y}):=\mathfrak{m}_{r+s+1}(\mathbf{x}, p, \mathbf{y})
$$
We will refer to $A_\Delta$ as the (shifted) diagonal bimodule.
\end{definition}
\begin{definition}
    Let $A_\Delta^{\vee}$ be the bimodule defined as $A^{\vee}[1]$, where $\vee$ stands for the $\mathrm{k}$-linear dual. We define the operations
$$
\mathfrak{n}_{r, s}^{\vee}(\mathbf{x}, \underline{\pi}, \mathbf{y})(p):=(-1)^{|\mathbf{x}|^{\prime}+|\mathbf{y}|^{\prime}+|p|^{\prime}} \pi\left(\mathfrak{m}_{r+s+1}(\mathbf{y}, p, \mathbf{x})\right)
$$
We will refer to $A_\Delta^{\vee}$ as the (shifted) dual diagonal bimodule.
\end{definition}
In order to define the Calabi-Yau structure we need the following lemma.
\begin{lemma}[{cf. \parencite[Lemma 2.11]{amorim2022morita}}]\label{lem:weakCY}
There is a quasi-isomorphism of chain complexes
$$
\Psi: C_{*}(A)^{\vee} \rightarrow \operatorname{hom}_{[A,A]}^{*}\left(A_\Delta, A_\Delta^{\vee}\right)
\vspace{-0.5cm}
$$
given explicitly as
$$
\Psi(\varphi)(\boldsymbol{x}, \underline{v}, \boldsymbol{y})(w)=\sum(-1)^@ \varphi\left(\mathfrak{m}\left(\boldsymbol{y}^{(3)}, w, \boldsymbol{x}, v, \boldsymbol{y}^{(1)}\right), \boldsymbol{y}^{(2)}\right)
$$
\end{lemma}
We are now able to define a weak Calabi-Yau structure on $A$.
\begin{definition} \label{defnWCYstr}
    An element $\phi \in H H_{*}(A)^{\vee}$ is called non-degenerate if $\Psi(\phi)$ is a quasi-isomorphism.

A weak Calabi-Yau structure on $A$ is a non-degenerate element $\phi \in H H_{*}(A)^{\vee}$.
\end{definition}
Notice that the non-degeneracy of the Mukai pairing implies that the map $\langle-,-\rangle_{Muk}: HH_*(A)\to HH_*(A)^\vee$ defined by $\alpha\mapsto \langle \alpha,-\rangle_{Muk}$ is an isomorphism. Hence, we can write $\phi=\langle \omega,-\rangle_{Muk}$ for some unique $\omega\in HH_*(A)$.

Assume further that $A$ is minimal, i.e., $\emm_1=0$. Let $\phi \in H H_{*}(A)^{\vee}$ be a weak Calabi-Yau structure on $A$. That is, the map $\Psi(\phi):A_\Delta \to A_\Delta^\vee$ is an isomorphism. This means for every $u\in A_\Delta$, there exists $v\in A_\Delta$ such that
$$
\Psi(\phi)_{0,0}(u)(v) = (-1)^{|v|'|u|'}\phi(\emm_2(v,u))=(-1)^{|v|'|u|'}\langle\omega, \emm_2(v,u)\rangle_{Muk} \neq 0
$$

Let $\rho \in hom_{[A,A]}(A_\Delta,A_\Delta^\vee[d])$ denote a chain-level representative of $\Psi(\phi)$. 

As in \parencite[Definition 2.28]{amorim2022morita} we define the contraction map $\iota_{-} \rho: C^{*}(A) \rightarrow C_{*+d}(A)^{\vee}$
as follows
$$
\iota_Z \rho\left(a_0 \otimes a_1 \cdots \otimes a_n\right)=\sum(-1)^{@+\dagger} \rho\left(a_1, \ldots \underline{Z\left(a_i, \ldots, a_j\right)} \ldots a_n\right)\left(a_0\right) . \label{eqiota}
$$ $\dagger:=|Z|^{\prime}\left(|\rho|+\left|a_1\right|^{\prime}+\ldots+\left|a_{i-1}\right|^{\prime}\right)$
and $Z \in C^{*}(A)$ is a Hochschild cochain.

Further, we have the Mukai pairing defined in (\ref{eq:Mukaipairing}).
Since $\AAA$ is smooth and proper, the Mukai pairing is non-degenerate, hence it induces an isomorphism $H H_{*}(\AAA) \cong H H_{*}(\AAA)^{\vee}$. Putting together, we obtain a chain of isomorphisms:

\begin{equation}
    HH^*(A) \ \xrightarrow{\iota_-\rho} \ HH_{*}(A)^\vee \xleftarrow{\langle-,-\rangle_{Muk}}\ HH_*(A) \label{eq D}
\end{equation}

\begin{lemma}[{cf. \parencite[Proposition A.1]{amorim2022categorical}}]\label{appA.1}
    Given a Hochschild cochain class $\theta\in HH^*(A)$ and chains $\alpha,\beta \in HH_*(A)$, we have the identity $$\langle\theta\cap\alpha,\beta\rangle_{Muk}=(-1)^{|\phi||\alpha|}\langle\alpha,\theta\cap\beta\rangle_{Muk}$$
\end{lemma}
We state our main proposition. This is a generalization of \parencite[See Theorem 2.4, and the Appendices A.1 and A.2]{amorim2022categorical} which shows this for cyclic structures as opposed to Calabi-Yau structures.
\begin{proposition} \label{poincare} The isomorphism
   $ D: HH^*(A) \longrightarrow HH_*(A)$ in equation (\ref{eq D}) is an isomorphism of modules. That is, $$D(\phi \cup \psi) = \phi \cap D(\psi)$$ for Hochschild cochains $\phi, \psi \in HH^*(A)$
\end{proposition}

\begin{proof}
    Since $\langle-,-\rangle_{Muk}$ is non-degenerate, it suffices to show that $$\langle D(\phi \cup \psi), \bfa \rangle_{Muk} = \langle \phi \cap D(\psi),\bfa\rangle_{Muk}, \qquad \text{for any } \bfa\in HH_*(\AAA)$$
    The LHS by definition of $D$ and $\iota_-\rho$ is equal to $(\iota_{\phi\cup\psi}\rho)(\bfa)$ or $(-1)^{|\phi||\psi|}(\iota_{\psi\cup\phi}\rho)(\bfa)$ by the commutativity of $\cup$. On the other hand, the right hand side is equal to $(-1)^{|\phi||D(\psi)|}\langle  D(\psi),\phi \cap\bfa\rangle_{Muk}$ by Lemma \ref{appA.1}, which is further equal to $(-1)^{|\phi||D(\psi)|}(\iota_\psi \rho)(\phi\cap\bfa)$ by definition of $D$. So we need to show \begin{equation}
        (-1)^{|\phi||\psi|}(\iota_{\psi\cup\phi}\rho)(\bfa) = (-1)^{|\phi||D(\psi)|}(\iota_\psi \rho)(\phi\cap\bfa) \label{eq39}
    \end{equation}
Write $\bfa = a_0|a_1,...,a_n$, then the LHS of equation (\ref{eq39}) becomes
\begin{align}
    (\iota_{\psi\cup\phi}\rho)(\bfa) &=\sum(-1)^{@+\dagger}\rho\left(a^{(1)}, \underline{\psi\cup\phi(a^{(2)},a^{(3)},a^{(4)},a^{(5)},a^{(6)})},a^{(7)})\right)(a_0)\\
    &=\sum(-1)^{@+\dagger+|\cup|}\rho\left(a^{(1)}, \underline{\emm(a^{(2)},\psi(a^{(3)}),a^{(4)},\phi(a^{(5)}),a^{(6)})},a^{(7)}\right)(a_0)
\end{align}
Similarly, we compute the RHS as \begin{align}
    (\iota_\psi \rho)(\phi\cap\bfa)=\sum (-1)^{@+\dagger+|\cap|} \rho\left(a^{(2)}, \underline{\psi(a^{(3)})}, a^{(4)}\right)\left(\emm(a^{(5)}, \phi(a^{(6)}), a^{(7)}, a_0, a^{(1)})\right) \label{eq 43}
\end{align}
where $|\cup|=|\phi|(|a^{(2)}|+|a^{(3)}|+|a^{(4)}|) + |\psi||a^{(2)}|$, \ $|\cap|=|\phi|'(|a^{(2)}|+|a^{(3)}|+|a^{(4)}|+ |a^{(2)}|)$ and $\dagger$ is the sign from the definition of $\iota_-\rho$, for example,  in equation (\ref{eq 43}), $\dagger:=|\psi|^{\prime}\left(|\rho|+\left|a^{(2)}\right|^{\prime}\right)$

Define \begin{equation}
    \iota_{\psi,\phi}( \theta)(\bfa):= \sum (-1)^{@+\dagger'} \ \theta\left(a^{(1)}, \underline{\psi(a^{(2)})}, a^{(3)}, \phi(a^{(4)}), a^{(5)}\right)(a_0) \label{eqdef}
\end{equation} where $\dagger' = |\phi|'(|\theta|+|a^{(1)}|+|a^{(2)}|+|a^{(3)}|) + |\psi|^{\prime}(|\theta|+|a^{(1)}|)$.
The computation below shows that \begin{equation}
\iota_{\psi,\phi}(\partial \rho)(\bfa)= \iota_{\psi\cup\phi}( \rho)(\bfa) + (-1)^{|\rho||\phi|}\iota_{\psi} \rho(\phi\cap\bfa) + \iota_{\psi,\phi}( \rho)(\bee(\bfa))
\end{equation}
\begin{align}
    \iota_{\psi,\phi}(\partial\rho)(\bfa) &= \sum (-1)^{@+\dagger'}\ \partial\rho\left(a^{(1)}, \underline{\psi(a^{(2)})}, a^{(3)}, \phi(a^{(4)}), a^{(5)}\right)(a_0)\\
    &=\sum (-1)^{@+\dagger'+ 1+|\rho|+|\emm_<|}\ \rho\left(a^{(2)}, \underline{\psi(a^{(3)})}, a^{(4)}, \phi(a^{(5)}), a^{(6)}\right)\left(\emm (a^{(7)},a_0, a^{(1)})\right)\\
    &+\sum (-1)^{@+\dagger'+ 1+|\rho|+|\emm_<|}\ \rho\left(a^{(2)}, \underline{\psi(a^{(3)})}, a^{(4)}\right)\left(\emm (a^{(5)}, \phi(a^{(6)}),a^{(7)},a_0, a^{(1)})\right) \label{eq47}\\
    &+\sum (-1)^{@+\dagger'+ 1+|\rho|+|\emm_<|}\ \rho\left(a^{(1)}, \underline{\emm(a^{(2)},\psi(a^{(3)}), a^{(4)})}, a^{(5)}, \phi(a^{(6)}), a^{(7)}\right)(a_0)\label{eq48}\\
    &+\sum (-1)^{@+\dagger'+ 1+|\rho|+|\emm_<|}\ \rho\left(a^{(1)}, \underline{\emm(a^{(2)},\psi(a^{(3)}), a^{(4)}), \phi(a^{(5)}), a^{(6)}}), a^{(7)}\right)(a_0) \label{eq49}\\
    &+\sum (-1)^{@+\dagger'+ 1+|\rho|+|\emm_<|}\ \rho\left(a^{(1)}, \emm(a^{(2)}),a^{(3)},\underline{\psi(a^{(4)})}, a^{(5)}, \phi(a^{(6)}), a^{(7)}\right)(a_0)\\
    &+\sum (-1)^{@+\dagger'+ 1+|\rho|+|\emm_<|}\ \rho\left(a^{(1)}, \underline{\psi(a^{(2)})}, a^{(3)}, \emm(a^{(4)}), a^{(5)},\phi(a^{(6)}), a^{(7)}\right)(a_0)\\
    &+\sum (-1)^{@+\dagger'+ 1+|\rho|+|\emm_<|}\ \rho\left(a^{(1)}, \underline{\psi(a^{(2)})}, a^{(3)}, \emm(a^{(4)},\phi(a^{(5)}), a^{(6)}), a^{(7)}\right)(a_0)\label{eq52}\\
    &+\sum (-1)^{@+\dagger'+ 1+|\rho|+|\emm_<|}\ \rho\left(a^{(1)}, \underline{\psi(a^{(2)})}, a^{(3)},\phi(a^{(4)}), a^{(5)}, \emm(a^{(6)}), a^{(7)}\right)(a_0)
\end{align}
where $\dagger'$ is the appropriate Koszul sign as in equation (\ref{eqdef}), and $|\emm_<|$ is the sum of (shifted) degrees of terms to the left of $\emm$. Clearly, equations (\ref{eq47}) and (\ref{eq49}) correspond to $\iota_{\psi} \rho(\phi\cap\bfa)$ and $\iota_{\psi\cup\phi}( \rho)(\bfa)$ respectively. Since $\psi$ and $\phi$ are closed cochains, the terms in \ref{eq48} and \ref{eq52} can be rewritten as $$(-1)^{|\psi|+|a^{(2)}|'(|\psi|+1)}\rho\left(a^{(1)},\underline{\psi(a^{(2)},\emm(a^{(3)}), a^{(4)})}, a^{(5)}, \phi(a^{(6)}), a^{(7)}\right)(a_0)$$ and $$(-1)^{|\phi|+|a^{(4)}|'(|\phi|+1)}\rho\left(a^{(1)}, \underline{\psi(a^{(2)})}, a^{(3)}, \phi(a^{(4)},\emm(a^{(5)}), a^{(6)}), a^{(7)}\right)(a_0)$$ respectively. Combined with the rest of the terms, these now correspond to $\iota_{\psi,\phi}( \rho)(\bee(\bfa))$. Since $\bfa$ and $\rho$ are closed, we have $\iota_{\psi\cup\phi}( \rho)(\bfa) =- (-1)^{|\rho||\phi|}\iota_{\psi} \rho(\phi\cap\bfa)$.
\end{proof}

\begin{lemma} \label{lemma40}
Let $\omega\in HH_*(A)$ be the unique Hochschild homology class representing a fixed weak Calabi-Yau structure $\phi$, that is, $\phi=\langle \omega,-\rangle_{Muk}$. Then,   $D(1)=\omega.$
\end{lemma}
\begin{proof}
    \begin{align}
        \iota_{1}\rho&(a_0|a_1,...,a_n)=\sum (-1)^{@_\rho} \rho^\otimes(a^{(1)}, \underline{1}, a^{(2)}, a^{(3)}, a^{(4)})(a_0) \notag\\
        &= \sum (-1)^{@_\rho+@_\phi} \phi(\emm_*(a^{(4)}, a_0, a^{(1)}, 1,a^{(2)}), a^{(3)}), \ \  \text{ (since $\rho =\Psi(\phi)$)} \notag\\
        &= (-1)^{|\rho|+|a_0|(|a_1|'+...+|a_n|')}\phi(a_0|a_1,...,a_n)\notag\\
        &= (-1)^{|\rho|+|a_0|(|a_1|'+...+|a_n|')}\langle \omega, a_0|a_1,...,a_n\rangle_{Muk}\notag
     \end{align}
where $@_\rho=|\rho|+|a^{(1)}|'+|a_0|'(|a^{(1)}|'+|a^{(2)}|'+|a^{(3)}|'+|a^{(4)}|')$ and $@\rho=(|a_0|'+|a^{(4)}|')(|a^{(1)}|'+|a^{(2)}|'+1)$.
But by construction of $D$, we have $(\iota_1\rho)(a_0|a_1,...,a_n)=\langle D(1), a_0|a_1,...,a_n\rangle_{Muk}$. Hence, $\langle \omega, a_0|a_1,...,a_n\rangle_{Muk}=\langle D(1), a_0|a_1,...,a_n\rangle_{Muk}$ or $D(1)=\omega$.
\end{proof}

\begin{corollary}\label{cor40}
    $\omega=D(1)$ generates $HH_*(A)$ as a $HH^*(A)$-module.
\end{corollary}
\begin{proof}
    Since $D(\psi \cup 1)= \psi \cap D(1)$, let $\theta\in HH_*(A)$. Then $$\theta=D(D^{-1}(\theta)\cup 1) = D^{-1}(\theta) \cap D(1)$$
\end{proof}


\subsection{Weak Calabi-Yau structure on the tensor product}
We will obtain a weak Calabi-Yau structure on $A\otimes_\otimes B$ given weak Calabi-Yau structures on $A$ and $B$ respectively. We will abuse notation and write ``$\omega_A$ is a weak Calabi-Yau structure on $A$'' if $\phi_A:=\langle\omega_A, -\rangle_{Muk}$ is a non-degenerate element of $HH_*(A)$.

\begin{theorem}\label{thm42}
    $\phi_{A\otimes_\infty B}:=\langle\omega^\otimes, \ -\rangle_{Muk}$ where $\omega^\otimes := k(\omega_A \otimes \omega_B)$ is a weak Calabi-Yau structure on $A\otimes_\infty B$ if $\omega_A, \omega_B$ are weak Calabi-Yau structures on $A$ and $B$ respectively.
\end{theorem}
\begin{proof}
    We need to show that $\phi_A\otimes\phi_B:=\langle\omega^\otimes, \ -\rangle_{Muk}$ is a non-degenerate element of $HH_*(A\otimes B)^\vee$. By Lemma \ref{lem:weakCY}, we have
    \begin{align}
        \Psi(\phi_{A\otimes_\infty B})&(u_A\otimes u_B)(v_A\otimes v_B)= (-1)^@ \phi_A\otimes\phi_B(\emm_2^\otimes(v_A\otimes v_B, u_A\otimes u_B))\\
        &=(-1)^@\langle k(\omega_A \otimes \omega_B), k(\emm_2^A(v_A, u_A)\otimes \emm_2^B(v_B,u_B))\rangle_{Muk}\\
        &=(-1)^@\langle\omega_A, \emm_2^A(v_A,u_A)\rangle_{Muk}\cdot \langle\omega_B, \emm_2^B(v_B,u_B)\rangle_{Muk} \ \ \text{ (by Theorem \ref{thm: mukai})}
    \end{align}
Since $\omega_A, \omega_B$ are weak Calabi-Yau structures on $A$ and $B$ respectively, the result follows.
\end{proof}

\begin{corollary} \label{cor43}
    $\omega^\otimes$ generates $HH_*(A\otimes_\infty B)$ as $HH^*(A\otimes_\infty B)$-module.
\end{corollary}
\begin{proof}
    By Theorem \ref{thm42}, $\omega^\otimes$ is a weak Calabi-Yau structure on $A\otimes_\infty B$. Hence, by Corollary \ref{cor40}, $\omega^\otimes$ is a generator of $HH_*(A\otimes_\infty B)$ as $HH^*(A\otimes_\infty B)$-module.
\end{proof}
\subsection{Splittings of the Hodge filtration}
As stated in the introduction, the negative cyclic homology $HC_*^-(A)$ of $A$ carries the $u$-filtration $F^pHC_*^-(A)=u^pHC_*^-(A)$, also known as the non-commutative Hodge filtration. A {\it splitting} of this Hodge filtration is a $\mathrm{k}$-linear map $s=s_A:HH_*(A)\xrightarrow{} HC_*^-(A)$ whose reduction modulo $u$ is the canonical isomorphism $HH_*(A)\cong HC_*^-(A)/uHC_*^-(A)$. We state the following definition, formulated in terms of the semi-infinite Hodge structure (see \parencite[Theorem A]{sheridan2020formulae}) of $\AAA$ given by the triple $\left(H C_{\bullet}^{-}(\AAA), \nabla_{\frac{d}{d u}},\langle-,-\rangle_{hres }\right)$.
\begin{definition}\label{goodsplittings}
 Let $\omega_A$ be a generator of  $H H_{*}(\AAA)$. A $\mathrm{k}$-linear map $s=s_A: H H_{*}(\AAA) \rightarrow H C_{*}^{-}(\AAA)$ is called a splitting of the Hodge filtration of $\AAA$ if it satisfies
\begin{itemize}
    \item[S1] (Splitting condition) $s$ splits the canonical map $\pi: H C_{*}^{-}(\AAA) \rightarrow H H_{*}(\AAA)$, defined as $$\pi\left(\sum_{n \geq 0} \alpha_n u^n\right)=\alpha_0$$
    \item[S2] (Lagrangian condition) $\langle s(\alpha), s(\beta)\rangle_{hres }=\langle\alpha, \beta\rangle_{Muk }, \forall \alpha, \beta \in H H_{*}(\AAA)$.
\end{itemize}
\noindent
 A splitting $s$ is called a good splitting if it satisfies

S3. (Homogeneity) $L^s:=\bigoplus_{l \in \mathbb{N}} u^{-l} \cdot \operatorname{Im}(s)$ is stable under the u-connection $\nabla_{u \frac{d}{d u}}$.

This is equivalent to requiring $\nabla_{u \frac{d}{d u}} s(\alpha) \in u^{-1} \operatorname{Im}(s)+\operatorname{Im}(s), \forall \alpha$. \\
A splitting $s$ is called $\omega$-compatible if

S4. ($\omega$-Compatibility) $\nabla_{u \frac{d}{d u}} s(\omega) \in r \cdot s(\omega)+u^{-1} \cdot \operatorname{Im}(s)$ for some $r \in \mathrm{k}$.
\end{definition}

In Amorim-Tu \parencite[Section 3.2, and Theorem 3.10]{amorim2022categorical}, $\omega$-compatible good splittings are shown to exist in the case where the Hochschild cohomology ring $HH^*(\AAA)$ is semi-simple. They further classified them in terms of grading operators on Hochschild homology. In that setting, $HH_*(A)$ decomposes into weight spaces $\bigoplus_i HH_*(A)_{\lambda_i}$. The stability condition of $s$ under the $u$-connection becomes equivalent to requiring $\nabla_{u\frac{d}{du}} \left( s_A(\alpha) \right) = u^{-1} s_A(\xi_A(\alpha)) + s_A(\mu_A(\alpha)), \ \forall \alpha\in HH_*(A)$, where the diagonal endomorphism $\xi_A$ acts as scalar multiplication by the eigenvalues $\lambda_i \cdot \mathrm{id}$ on $HH_*(A)_{\lambda_i}$, and $\mu_A: HH_*(A)\to HH_*(A)$ is a grading operator. 


In our framework, we do not assume semi-simplicity. Instead, we show that for any smooth and proper $A_\infty$-algebra $A$ equipped with a good splitting $s_A$, the connection $\nabla_{u\frac{d}{du}}$ defines unique endomorphisms $\xi_A,$ and $ \mu_A$ without requiring $\xi_A$ to be diagonal.
\begin{lemma}\label{lem:endo_mu}
    For any smooth and proper $A_\infty$-algebra $A$ equipped with a good splitting $s=s_A: HH_*(A) \to HC_*^-(A)$, the action of the connection operator $\nabla_{u\frac{d}{du}}$ on $\operatorname{Im}(s_A)$ is given by: $$\nabla_{u\frac{d}{du}} \left( s_A(\alpha) \right) = u^{-1} s_A(\xi_A(\alpha)) + s_A(\mu_A(\alpha)), \ \forall \alpha\in HH_*(A),$$ for uniquely determined $\mathrm{k}$-linear endomorphisms $\xi_A, \mu_A \in \mathrm{End}(HH_*(A))$.
\end{lemma}
\begin{proof}
    Given the canonical filtration on negative cyclic homology, the map $s_A: HH_*(A) \to HC_*^-(A)$ induces an isomorphism of $\mathrm{k}[[u]]$-modules: $$HC_*^-(A) \cong \bigoplus_{l=0}^\infty u^l \cdot s_A(HH_*(A))=\bigoplus_{l=0}^\infty u^l \cdot \text{Im}(s_A).$$ Since $s_A$ is a good splitting, we have the stability condition $\nabla_{u \frac{d}{d u}} s(\alpha) \in u^{-1} \operatorname{Im}(s)+\operatorname{Im}(s), \ \forall \alpha\in HH_*(A)$, which forces all positive powers of $u$ in the module expansion of $\nabla_{u\frac{d}{du}}(s_A(\alpha))$ to vanish identically. This stability condition implies the existence of unique elements $x,y\in HH_*(A)$ such that $$\nabla_{u\frac{d}{du}}\left(s_A(\alpha)\right) = u^{-1} s_A(x) + s_A(y).$$ Setting $\xi_A(\alpha) := x$ and $\mu_A(\alpha) := y$, we obtain linear maps $\xi_A, \mu_A: HH_*(A) \to HH_*(A)$. Because $A$ is smooth and proper, $HH_*(A)$ is finite-dimensional, hence $\xi_A$ and $\mu_A$ are well-defined endomorphisms in $\mathrm{End}(HH_*(A))$. We explicitly compute $\xi_A(\alpha)$ and $\mu_A(\alpha)$ for any $\alpha\in HH_*(A)$. 

Multiply $\nabla_{u\frac{d}{du}} \left( s_A(\alpha) \right)$ by $u$ and evaluate modulo $u HC_*^-(A)$: $$u \cdot \nabla_{u\frac{d}{du}} \left( s_A(\alpha) \right) = s_A(\xi_A(\alpha)) + u s_A(\mu_A(\alpha)).$$ Applying $\pi_A$, we obtain $$\xi_A(\alpha) = \pi_A \left( u \cdot \nabla_{u\frac{d}{du}} \left( s_A(\alpha) \right) \right).$$ Since $u \cdot \nabla_{u\frac{d}{du}} \left( s_A(\alpha) \right)$ is a well-defined element of $HC_*^-(A)$, its image under $\pi_A$ is unique. 

Similarly, compute the difference: $$\nabla_{u\frac{d}{du}} \left( s_A(\alpha) \right) - u^{-1} s_A(\xi_A(\alpha)) = s_A(\mu_A(\alpha))\in HC_*^-(A).$$ Applying $\pi_A$ yields: $$\mu_A(\alpha) = \pi_A \left( \nabla_{u\frac{d}{du}} \left( s_A(\alpha) \right) - u^{-1} s_A(\xi_A(\alpha)) \right).$$ Since $s_A$, $\pi_A$, and $\nabla_{u\frac{d}{du}}$ are fixed, $\xi_A(\alpha)$ and $\mu_A(\alpha)$ are canonically defined for every $\alpha \in HH_*(A)$. 
\end{proof}

We apply the previous results to obtain a $\omega^\otimes$-compatible good splitting of the Hodge filtration on $A\otimes_\infty B$ in terms of the splittings on $A$ and $B$.

\begin{theorem}\label{thmgoodsplitting}
Let $s_A, s_B$ be splittings of the hodge filtration on $A$, and $B$ respectively. Then $s_{A\otimes_\infty B}: = (k+uK)\circ(s_A \otimes s_B)\circ k^{-1}$ is a splitting of the Hodge filtration on $A\otimes_\infty B$. Further, if $s_A, s_B$ are $\omega_A, \omega_B$-compatible good splittings of the Hodge filtration on $A,$ and $B$ respectively, then $s_{A\otimes_\infty B}$ is a $\omega^{\otimes }$-compatible good splitting of the Hodge filtration of $A\otimes_\infty B$, where $\omega^\otimes=k(\omega_A\otimes\omega_B)$. We have the following commutative diagram in Figure \ref{fig:commdiagram}.
\end{theorem}
\begin{figure}[h]
    \centering
   \begin{tikzcd}[row sep=4em, column sep=6em]
HC_*^-(A) \otimes_{k[[u]]} HC_*^-(B) \arrow[r, "k + uK", "\simeq"'] \arrow[d, shift left=6pt, "\pi_A \otimes \pi_B"] & HC_*^-(A \otimes_\infty B) \arrow[d, shift right=6pt, "\pi^\otimes"'] \\
HC_*(A) \otimes HC_*(B) \arrow[r, "k", "\simeq"'] \arrow[u, shift left=6pt, "s_A \otimes s_B"] & HH_*(A \otimes_\infty B) \arrow[u, shift right=6pt, "s_{A \otimes_\infty B}"']
\end{tikzcd}
    \caption{Commutative diagram showing compatibility of K\"unneth isomorphisms and splittings}
    \label{fig:commdiagram}
\end{figure}

\begin{proof}
Clearly, $s_{\AAA\otimes_\infty\BB}$ satisfies the splitting condition since we have \begin{align}
    \pi^\otimes \circ s_{\AAA\otimes_\infty\BB}&= (\pi^\otimes \circ (k+uK))\circ(s_\AAA \otimes s_\BB)\circ k^{-1} \\
    &=k\circ (\pi_\AAA\otimes\pi_\BB) \circ (s_\AAA \otimes s_\BB)\circ k^{-1} \\
    &=k \circ (\text{id}_{HH_*(\AAA)}\otimes\text{id}_{HH_*(\BB)})\circ k^{-1} \\
    &= \text{id}_{HH_*(\AAA \otimes_\infty\BB)}
\end{align}
For the Lagrangian condition, let $\alpha, \beta \in HH_*(\AAA \otimes_\infty\BB)$. Denote $k^{-1}(\alpha)=\sum_{(\alpha)} \alpha^{(1)} \otimes \alpha^{(2)}$, and $k^{-1}(\beta)=\sum_{(\beta)} \beta^{(1)} \otimes \beta^{(2)}$. \\
Then \begin{align}
    \langle s&_{\AAA\otimes_\infty\BB}(\alpha), \ s_{\AAA\otimes_\infty\BB}(\beta)\rangle_{hres} \\
    &= \left\langle (k+uK)\sum_{(\alpha)} s_{\AAA}(\alpha^{(1)}) \otimes s_\BB(\alpha^{(2)}), (k+uK)\sum_{(\beta)} s_{\AAA}(\beta^{(1)}) \otimes s_\BB(\beta^{(2)})\right\rangle_{hres} (\text{by definition})\\
    &= \sum \langle s_{\AAA}(\alpha^{(1)}), \ s_{\AAA}(\beta^{(1)})\rangle_{hres} \cdot \sum \langle s_{\BB}(\alpha^{(2)}), \ s_{\BB}(\beta^{(2)})\rangle_{hres} \ \ \ (\text{by equation (\ref{eq:hresk+uK})})\\
    &= \sum \langle \alpha^{(1)}, \ \beta^{(1)}\rangle_{Muk} \cdot \sum \langle \alpha^{(2)}, \ \beta^{(2)}\rangle_{Muk}\\
    &= \left\langle k\left(\sum \alpha^{(1)} \otimes \alpha^{(2)}\right), \ k\left(\sum \beta^{(1)} \otimes \beta^{(2)}\right)\right\rangle_{Muk} \ \ \ \ (\text{by equation (\ref{eq:muk_k})})\\
    = & \ \langle \alpha, \ \beta\rangle_{Muk}
\end{align}
We will use the endomorphisms $\xi, \mu$ in Lemma \ref{lem:endo_mu} to prove homogeneity and $\omega^\otimes$-compatibility. Take $\alpha, \beta \in HH_*(\AAA\otimes_\infty\BB)$. We will show that $s_{\AAA\otimes_\infty\BB}$ satisfies the homogeneity condition.
\begin{align}
    \nabla_{u \frac{d}{du}}(s_{\AAA\otimes_\infty\BB} &\circ k(\alpha \otimes \beta)) =  \nabla_{u \frac{d}{du}} (k+uK)(s_\AAA(\alpha) \otimes s_\BB(\beta)) \\
    &= (k+uK)( \nabla_{u \frac{d}{du}} \otimes id + id \otimes  \nabla_{u \frac{d}{du}})(s_\AAA(\alpha) \otimes s_\BB(\beta))\\
    &= (k+uK)( \nabla_{u \frac{d}{du}} (s_\AAA(\alpha)) \otimes s_\BB(\beta) + s_\AAA(\alpha) \otimes  \nabla_{u \frac{d}{du}} (s_\BB(\beta)))\\
    &= (k+uK) \left((u^{-1}s_\AAA(\xi_\AAA (\alpha)) + s_\AAA (\mu_\AAA(\alpha))) \otimes s_\BB(\beta)\right) \notag \\ &\quad + (k+uK)\left(s_\AAA(\alpha) \otimes (u^{-1}s_\BB(\xi_\BB (\beta)) + s_\BB (\mu_\BB(\beta))\right)\\
    &= u^{-1}(k+uK)\left( s_\AAA \otimes s_\BB) (\xi_\AAA(\alpha) \otimes \beta + \alpha \otimes \xi_\BB (\beta) \right) \notag \\
    & \quad + (k+uK)\left( s_\AAA \otimes s_\BB) (\mu_\AAA(\alpha) \otimes \beta + \alpha \otimes \mu_\BB (\beta) \right) \\
    &= s_{\AAA\otimes_\infty \BB}\circ k \left( u^{-1} (\xi_\AAA \otimes id + id \otimes \xi_\BB) + (\mu_\AAA \otimes id + id \otimes \mu_\BB)\right)(\alpha\otimes\beta)\\
    \nabla_{u \frac{d}{du}}\left( s_{\AAA\otimes_\infty\BB}\right) &= s_{\AAA\otimes_\infty \BB}\circ k \left( u^{-1} (\xi_\AAA \otimes id + id \otimes \xi_\BB) + (\mu_\AAA \otimes id + id \otimes \mu_\BB)\right)\circ k^{-1}\\
    &= s_{\AAA\otimes_\infty \BB} \left( u^{-1} \xi_\otimes + \mu_\otimes \right) = u^{-1}s_{\AAA\otimes_\infty \BB} (\xi_\otimes) + s_{\AAA\otimes_\infty \BB} (\mu_\otimes) 
\end{align}
where $\xi_\otimes = k\circ(\xi_\AAA \otimes id + id \otimes \xi_\BB)\circ k^{-1}$ and $\mu_\otimes = k\circ (\mu_\AAA \otimes id + id \otimes \mu_\BB)\circ k^{-1}$

To show that $s_{\AAA\otimes_\infty\BB}$ is $\omega^{\otimes}$-compatible, where $\omega^{\otimes} = k(\omega_\AAA \otimes \omega_\BB)$,
\begin{align}
    \nabla_{u \frac{d}{du}}(s_{\AAA\otimes_\infty\BB} &\circ \omega^\otimes)  
    = (k+uK)( \nabla_{u \frac{d}{du}} \otimes id + id \otimes  \nabla_{u \frac{d}{du}})(s_\AAA(\omega_\AAA) \otimes s_\BB(\omega_\BB))\\
    &= (k+uK) \left((u^{-1}s_\AAA(\xi_\AAA (\omega_\AAA)) + r_\AAA s_\AAA (\omega_\AAA)) \otimes s_\BB(\omega_\BB)\right) \notag \\ &\quad + (k+uK)\left(s_\AAA(\omega_\AAA) \otimes (u^{-1}s_\BB(\xi_\BB (\omega_\BB)) + r_\BB s_\BB (\omega_\BB)\right)\\
    &= s_{\AAA\otimes_\infty \BB}\circ k \left( u^{-1} (\xi_\AAA \otimes id + id \otimes \xi_\BB) + (r_\AAA \otimes id + id \otimes r_\BB)\right)(\omega_\AAA\otimes\omega_\BB)\\
    \nabla_{u \frac{d}{du}}\left( s_{\AAA\otimes_\infty\BB}(\omega^\otimes)\right) &= s_{\AAA\otimes_\infty \BB} \left( u^{-1} \xi_\otimes + r_\otimes \right) = u^{-1}s_{\AAA\otimes_\infty \BB} (\xi_\otimes(\omega^\otimes)) + r_\otimes s_{\AAA\otimes_\infty \BB}(\omega^\otimes)
\end{align}
where $\xi_\otimes = k\circ(\xi_\AAA \otimes id + id \otimes \xi_\BB)\circ k^{-1}$ and $r_\otimes = 
r_\AAA + r_\BB$
\end{proof}


\newpage
\section{Appendix A}


\label{section-appendix}
In the calculations that follow, we will let $N(\bfa)=\sum(-1)^@\ a^{(2)}\otimes a_0\otimes a^{(1)}$  denote the sum of all cyclic permutations on $\bfa$,  where $@$ denotes the corresponding sign. Then by definition of the cyclic shuffle, we have \[
\operatorname{Sh}(\underline{\bfa}, \underline{\bfb}) = \sum_{} \operatorname{sh}_0(N\underline{\bfa}, N\underline{\bfb})
\]
where $\operatorname{sh}_0$ denotes the shuffles satisfying $\barb_0<\bara_0$.
\subsection{Proof of Lemma \ref{lem:AppendixA}}\label{secA.1} 
We compute term by term. The left hand side gives
\begin{align}
\Bee \circ k(\bfa; \bfb) &= \sum (-1)^{|b_0||a|' + @_{\AAA\otimes \BB}} \ 1\otimes 1\ | \ \operatorname{sh}(\bara;\barb)^{(2)}, a_0\otimes b_0, \operatorname{sh}(\bara;\barb)^{(1)}\notag\\
&= \sum (-1)^{|b_0||a|' + |a^{(2)}|'|b^{(1)}|'+@_{\AAA\otimes \BB}} \ 1\otimes 1\ | \ \operatorname{sh}(\bara^{(2)};\barb^{(2)}), a_0\otimes b_0, \operatorname{sh}(\bara^{(1)};\barb^{(1)})\label{eq:14}\\
\text{where } &\ @_{\AAA\otimes \BB}= (|a^{(2)}|'+|b^{(2)}|')(|a^{(1)}|'+|b^{(1)}|'+|a_0|+|b_0|)\notag
\end{align}
\begin{align}
\bee \circ K(\mathbf{a};\mathbf{b}) 
&= \sum_{\barb_0 < \bara_0} (-1)^{<+\dagger}\ 1\otimes 1 | \ \operatorname{Sh}(\underline{\bfa};\underline{\bfb})^{(1)}, \emm_*^\otimes(\operatorname{Sh}(\underline{\bfa};\underline{\bfb})^{(2)}), \operatorname{Sh}(\underline{\bfa};\underline{\bfb})^{(3)} \notag \\
 \qquad+ \sum_{\barb_0 < \bara_0} &(-1)^{\dagger+\kappa} \ \emm_*^\otimes\left(\operatorname{Sh}(\underline{\bfa};\underline{\bfb})^{(3)}, 1\otimes1, \operatorname{Sh}(\underline{\bfa};\underline{\bfb})^{(1)}\right) \ |\ \operatorname{Sh}(\underline{\bfa};\underline{\bfb})^{(2)}\label{eq:label}\\
\text{where }& \ <\ =1+|a^{(1)}|'+|b^{(1)}|', \  \dagger=|a^{(2)}|'|b^{(1)}|'+|a^{(3)}|'|b^{(1)}|'+|a^{(3)}|'|b^{(2)}|', \notag\\
\text{ and } &\ \kappa=(|a^{(3)}|'+|b^{(3)}|')(|a^{(2)}|'+|b^{(2)}|'+|a^{(1)}|'+|b^{(1)}|')\notag\\
= \sum (-1)&^{<+\dagger+|a^{(2)}|'(|a^{(1)}|'+|a_0|)+|b^{(2)}|'(|b^{(1)}|'+|b_0|)}1\otimes 1 | \operatorname{sh}(\underline{a}^{(2)};\underline{b}^{(2)}), \emm_2^\otimes(\barb_0,\bara_0), \operatorname{sh}(\underline{a}^{(1)};\underline{b}^{(1)})\label{eq:17}\\
+ \sum& (-1)^{<+\dagger}\ 1\otimes 1 | \ \operatorname{sh}(\underline{a}^{(3)};\barb^{(4)},\barb_0,\barb^{(1)}), \emm_*^\otimes\left(\operatorname{sh}(\bara^{(4)},\bara_0,\bara^{(1)};\underline{b}^{(2)})\right), \operatorname{sh}(\underline{a}^{(2)};\underline{b}^{(3)})  \label{eq:16}\\
+ \sum& (-1)^{<+\dagger}\ 1\otimes 1 | \ \operatorname{sh}(\underline{a}^{(2)};\underline{b}^{(3)}), \emm_*^\otimes\left(\operatorname{sh}(\underline{a}^{(3)};\barb^{(4)},\barb_0,\barb^{(1)})\right), \operatorname{sh}(\bara^{(4)},\bara_0,\bara^{(1)};\underline{b}^{(2)})  \label{eq:15}\\
+ \sum_{\barb_0 < \bara_0}& (-1)^{<+\dagger}\ 1\otimes 1 | \ \operatorname{sh}_0(\bara^{(4)},\bara_0,\bara^{(1)};\barb^{(4)},\barb_0,\barb^{(1)}), \emm_*^\otimes\left(\operatorname{sh}(\underline{a}^{(2)};\barb^{(2)})\right), \operatorname{sh}(\underline{a}^{(3)};\underline{b}^{(3)})  \label{eq:20}\\
+ \sum_{\barb_0 < \bara_0}& (-1)^{<+\dagger}\ 1\otimes 1 | \ \operatorname{sh}(\underline{a}^{(2)};\underline{b}^{(2)}), \emm_*^\otimes\left(\operatorname{sh}(\underline{a}^{(3)};\barb^{(3)})\right), \operatorname{sh}_0(\bara^{(4)},\bara_0,\bara^{(1)};\barb^{(4)},\barb_0,\barb^{(1)}) \label{eq:18}\\
+ \sum& (-1)^{<+\dagger}\ 1\otimes 1 | \ \operatorname{sh}(\underline{a}^{(2)};\barb^{(4)},\barb_0,\barb^{(1)}), \emm_*^\otimes\left(\operatorname{sh}(\underline{a}^{(3)};\barb^{(2)})\right), \operatorname{sh}(\bara^{(4)},\bara_0,\bara^{(1)};\underline{b}^{(3)}) \label{eq:19}\\
+ \sum &(-1)^{\dagger+\kappa} \ \emm_2^\otimes\left(a_0\otimes1, 1\otimes1\right) \ |\ \operatorname{sh}(\underline{a};N\underline{\bfb}) \ \text{ (by strict unitality on \ref{eq:label})}\label{eq:21a}\\
+ \sum &(-1)^{\dagger+\kappa} \ \emm_2^\otimes\left(1\otimes1, 1\otimes b_0\right) \ |\ \operatorname{sh}(N\underline{\bfa};\underline{b}) \ \text{ (by strict unitality on \ref{eq:label})}\label{eq:21}
\end{align}

Equation (\ref{eq:17}) cancels out with equation (\ref{eq:14}), with $\dagger$ in (\ref{eq:17}) adjusted to correspond with the renumbering. For example, in this case $\dagger= |a_0||b^{(2)}|'+|a^{(1)}|'|b^{(2)}|'+|a^{(1)}|'|b_0|$). 

The strict unitality condition on equation (\ref{eq:label}) yields (\ref{eq:21a} and \ref{eq:21}), but also terms of the form 
$\pm\ \emm_2^\otimes\left(a_i\otimes1, 1\otimes1\right) |\operatorname{sh}(N\underline{\bfa};N\underline{\bfb})$ and 
$\pm \ \emm_2^\otimes\left(1\otimes b_j, 1\otimes1\right) | \operatorname{sh}(N\underline{\bfa};N\underline{\bfb})$ $i,j>0$. These however come in commuting pairs, namely, 
$\pm\ \emm_2^\otimes\left(1\otimes1,a_i\otimes1\right) |\operatorname{sh}(N\underline{\bfa};N\underline{\bfb})$ and \\
$\pm\ \emm_2^\otimes\left(1\otimes 1,1\otimes b_j\right) | \operatorname{sh}(N\underline{\bfa};N\underline{\bfb})$ respectively, and thus cancel each other.

Also, note that from the commuting subalgebras criterion (Definition \ref{commutingsubalgebradefn}), the entries in $\emm_*^\otimes$ in equation (\ref{eq:15}) have to be all $\barb's$ except for when $*=2$ in which case $\emm_2^\otimes(\barb_0, \bara_i)$, $i>0$, shows up in a commuting pair with $\emm_2^\otimes(\bara_i,\barb_0)$, and thus cancel each other. Similarly, the entries in $\emm_*^\otimes$ in equation (\ref{eq:16}) have to be all $\bara$'s except for when $*=2$ in which case $\emm_2^\otimes(\bara_0, \barb_i)$, $i>0$, shows up in a commuting pair with $\emm_2^\otimes(\barb_i,\bara_0)$, and thus cancel each other.

Applying a similar argument in equations (\ref{eq:20}, \ref{eq:18}, \ref{eq:19}), we find that the entries in $\emm_*^\otimes$ have  to be all $\bara$'s or all $\barb$'s except when $*=2$ in which case $\emm_2^\otimes(\bara_i, \barb_j)$ shows up in commuting pairs with $\emm_2^\otimes(\barb_j, \bara_i)$, $i,j>0$, and thus cancel each other. Hence, each of equations (\ref{eq:20}, \ref{eq:18}, \ref{eq:19}) splits into two sums: one with $\emm_*^\otimes(...)$ replaced by $\emm_*^\AAA(...)$, and the other by $\emm^\BB_*(...)$ respectively according to Definition (\ref{commutingsubalgebradefn}). The left hand side now becomes
\begin{align}
    \Bee \circ k(\bfa; \bfb) &+\bee\circ K(\bfa;\bfb) = \notag\\
\sum& (-1)^{<+\dagger}\ 1\otimes 1 | \ \operatorname{sh}(\underline{a}^{(3)};\barb^{(3)},\barb_0,\barb^{(1)}), \emm_*^A\left(\bara^{(4)},\bara_0,\bara^{(1)}\right), \operatorname{sh}(\underline{a}^{(2)};\underline{b}^{(2)})  \label{eqAA.4}\\
+ \sum& (-1)^{<+\dagger}\ 1\otimes 1 | \ \operatorname{sh}(\underline{a}^{(2)};\underline{b}^{(3)}), \emm_*^B\left(\barb^{(4)},\barb_0,\barb^{(1)}\right), \operatorname{sh}(\bara^{(3)},\bara_0,\bara^{(1)};\underline{b}^{(2)})  \label{eqA.5}\\
+ \sum_{\barb_0 < \bara_0}& (-1)^{<+\dagger}\ 1\otimes 1 | \ \operatorname{sh}_0(\bara^{(4)},\bara_0,\bara^{(1)};\barb^{(3)},\barb_0,\barb^{(1)}), \emm_*^A\left(\underline{a}^{(2)}\right), \operatorname{sh}(\underline{a}^{(3)};\underline{b}^{(2)})  \label{eqA.6a}\\
+ \sum_{\barb_0 < \bara_0}& (-1)^{<+\dagger}\ 1\otimes 1 | \ \operatorname{sh}_0(\bara^{(3)},\bara_0,\bara^{(1)};\barb^{(4)},\barb_0,\barb^{(1)}), \emm_*^B\left(\underline{b}^{(2)}\right), \operatorname{sh}(\underline{a}^{(2)};\underline{b}^{(3)})  \label{eqA.6b}\\
+ \sum_{\barb_0 < \bara_0}& (-1)^{<+\dagger}\ 1\otimes 1 | \ \operatorname{sh}(\underline{a}^{(2)};\underline{b}^{(2)}), \emm_*^A\left(\underline{a}^{(3)}\right), \operatorname{sh}_0(\bara^{(4)},\bara_0,\bara^{(1)};\barb^{(3)},\barb_0,\barb^{(1)}) \label{eqA.7a}\\
+ \sum_{\barb_0 < \bara_0}& (-1)^{<+\dagger}\ 1\otimes 1 | \ \operatorname{sh}(\underline{a}^{(2)};\underline{b}^{(2)}), \emm_*^B\left(\underline{b}^{(3)}\right), \operatorname{sh}_0(\bara^{(3)},\bara_0,\bara^{(1)};\barb^{(4)},\barb_0,\barb^{(1)}) \label{eqA.7b}\\
+ \sum& (-1)^{<+\dagger}\ 1\otimes 1 | \ \operatorname{sh}(\underline{a}^{(2)};\barb^{(3)},\barb_0,\barb^{(1)}), \emm_*^A\left(\underline{a}^{(3)}\right), \operatorname{sh}(\bara^{(4)},\bara_0,\bara^{(1)};\underline{b}^{(2)}) \label{eqA.8a}\\
+ \sum& (-1)^{<+\dagger}\ 1\otimes 1 | \ \operatorname{sh}(\underline{a}^{(2)};\barb^{(4)},\barb_0,\barb^{(1)}), \emm_*^B\left(\underline{b}^{(2)}\right), \operatorname{sh}(\bara^{(3)},\bara_0,\bara^{(1)};\underline{b}^{(3)}) \label{eqA.8b}\\
+ \sum &(-1)^{\dagger+\kappa} \ a_0\otimes1 \ |\ \operatorname{sh}(\underline{a};N\underline{\bfb}) \label{eqA.9}\\
+ \sum &(-1)^{\dagger+\kappa+|b_0|} \ 1\otimes b_0 \ |\ \operatorname{sh}(N\underline{\bfa};\underline{b}) \label{eqA.10}
\end{align}

We now compute the right hand side term by term and show how they cancel out with the left hand side computed above.
\begin{align}
    k(\Bee \otimes \text{id})(\bfa; \bfb) &=  k\left((\sum (-1)^{@_\AAA}\ 1\ |\ a^{(2)},a_0, a^{(1)}); b_0 \ |\ b\right) \notag \\
&= \sum (-1)^{@_\AAA+ |b_0|(|a_0|+|a|')}\ 1\otimes b_0 \ |\ \operatorname{sh}(\bara^{(2)},\bara_0, \bara^{(1)}; \barb) \notag\\
&= \sum (-1)^{@_\AAA+ |b_0|(|a_0|+|a|')} 1\otimes b_0 \ |\ \operatorname{sh}(N\underline{\bfa};\underline{b})\label{EqA.15}\\
(-1)^{|\bfa|} k(\text{id}\otimes \Bee)(\mathbf{a}; \mathbf{b}) &= k\left(a_0 \ |\ a;(\sum (-1)^{@_\BB+|\bfa|}\ 1\ |\  b^{(2)},b_0, b^{(1)})\right) \notag \\
&= \sum (-1)^{@_\BB+ |\bfa|}\ a_0\otimes 1 \ |\ \operatorname{sh}(\bara; \barb^{(2)},\barb_0, \barb^{(1)}) \notag\\
&= \sum (-1)^{@_\BB+ |\bfa|}\ a_0\otimes 1 \ |\ \operatorname{sh}(\underline{a};N\underline{\bfb}) \label{EqA.16}
\end{align}
Equations (\ref{EqA.15}) and (\ref{EqA.16}) cancel out with  (\ref{eqA.10}) and (\ref{eqA.9}) respectively, where $@_A= |a^{(2)}|(|a_0|+| a^{(1)}|')$ and $@_B= |b^{(2)}|(|b_0|+| b^{(1)}|')$.
\begin{align}
K(\bee \otimes \text{id})(\mathbf{a}; \mathbf{b}) &= \sum_{\barb_0<\bara_0} (-1)^{|a_0|+|a^{(1)}|'} \ 1\otimes 1 |\operatorname{sh}(\bara^{(4)},\bara_0, \bara^{(1)}, \emm_*^A(\bara^{(2)}), \bara^{(3)};N\underline{\bfb}) \label{eqA.23}\\
+ \sum_{\barb_0<\emm_*^A(\bara^{(4)},\bara_0,\bara^{(1)})} &(-1)^{|a^{(3)}|'(|a_0|+|a^{(1)}|'+|a^{(2)}|')}\ 1\otimes 1|\operatorname{sh}(\bara^{(3)},\emm_*^A(\bara^{(4)},\bara_0,\bara^{(1)}), \bara^{(2)}; N\underline{\bfb})\label{eqA.24}
\end{align}
We break up the shuffle in (\ref{eqA.23}) according to whether $\emm_*^A(\bara^{(2)})$ is to the right, left, or middle of $\barb_0$ and $\bara_0$. Similarly, we break up the shuffle in (\ref{eqA.24}) rewriting it more explicitly with $\barb_0$ to the left of $\bara_0$ (and adjusting the signs appropriately) to obtain
\begin{align}
    K(\bee \otimes \text{id})(\mathbf{a}; \mathbf{b}) &= \sum_{\barb_0<\bara_0} (-1)^{|a_0|+|a^{(1)}|'} 1\otimes 1 | \ \operatorname{sh}_0(\bara^{(4)},\bara_0,\bara^{(1)};\barb^{(3)},\barb_0,\barb^{(1)}), \emm_*^A\left(\underline{a}^{(2)}\right), \operatorname{sh}(\underline{a}^{(3)};\underline{b}^{(2)}) \\
    + &\sum_{\barb_0<\bara_0} (-1)^{|a_0|+|a^{(1)}|'} 1\otimes 1 | \ \operatorname{sh}(\underline{a}^{(2)};\underline{b}^{(2)}), \emm_*^A\left(\underline{a}^{(3)}\right), \operatorname{sh}_0(\bara^{(4)},\bara_0,\bara^{(1)};\barb^{(3)},\barb_0,\barb^{(1)})\\
    + &\sum (-1)^{|a_0|+|a^{(1)}|'} 1\otimes 1 | \ \operatorname{sh}(\underline{a}^{(2)};\barb^{(3)},\barb_0,\barb^{(1)}), \emm_*^A\left(\underline{a}^{(3)}\right), \operatorname{sh}(\bara^{(4)},\bara_0,\bara^{(1)};\underline{b}^{(2)})\\
    + \sum (-1)&^{|a^{(3)}|'(|a_0|+|a^{(1)}|'+|a^{(2)}|')} 1\otimes 1 | \ \operatorname{sh}(\underline{a}^{(3)};\barb^{(3)},\barb_0,\barb^{(1)}), \emm_*^A\left(\bara^{(4)},\bara_0,\bara^{(1)}\right), \operatorname{sh}(\underline{a}^{(2)};\underline{b}^{(2)}) 
\end{align}
which cancel out respectively with (\ref{eqA.6a}), (\ref{eqA.7a}), (\ref{eqA.8a}), and (\ref{eqAA.4}). \\
Lastly, we compute
\begin{align}
    (-1)^{|\bfa|} K(\text{id} \otimes \bee)(\mathbf{a}; \mathbf{b}) &= \sum_{\barb_0<\bara_0} (-1)^{|b_0|+|b^{(1)}|' + |\bfa|} \ 1\otimes 1 |\operatorname{sh}_0(N\underline{\bfa};\barb^{(4)},\barb_0, \barb^{(1)}, \emm_*^B(\barb^{(2)}),\barb^{(3)})\label{eqA.29}\\
    + \sum_{\emm_*^B(\barb^{(4)},\barb_0,\barb^{(1)})<\bara_0} &(-1)^{|b^{(3)}|'(|a_0|+|a^{(1)}|'+|a^{(2)}|') + |\bfa|}\ 1\otimes 1|\operatorname{sh}_0(N\underline{\bfa};\barb^{(3)},\emm_*^B(\barb^{(4)},\barb_0,\barb^{(1)}), \barb^{(2)})\label{eqA.30}
\end{align}
Again, we break up the shuffle in (\ref{eqA.29}) according to whether $\emm_*^B(\bara^{(2)})$ is to the right, left, or middle of $\barb_0$ and $\bara_0$. Similarly, we break up the shuffle in (\ref{eqA.30}) rewriting it with $\barb_0$ to the left of $\bara_0$ (and adjusting the signs appropriately) to obtain
\begin{align}
    &\sum_{\barb_0<\bara_0} (-1)^{|b_0|+|b^{(1)}|' + |\bfa|} 1\otimes 1 |  \operatorname{sh}_0(\bara^{(3)},\bara_0,\bara^{(1)};\barb^{(4)},\barb_0,\barb^{(1)}), \emm_*^B\left(\underline{b}^{(2)}\right), \operatorname{sh}(\underline{a}^{(2)};\underline{b}^{(3)}) \\
    +& \sum_{\barb_0<\bara_0} (-1)^{|b_0|+|b^{(1)}|' + |\bfa|} \ 1\otimes 1 |  \operatorname{sh}(\underline{a}^{(2)};\underline{b}^{(2)}), \emm_*^B\left(\underline{b}^{(3)}\right), \operatorname{sh}_0(\bara^{(3)},\bara_0,\bara^{(1)};\barb^{(4)},\barb_0,\barb^{(1)}) \\
    +& \sum (-1)^{|b_0|+|b^{(1)}|' + |\bfa|} 1\otimes 1 | \ \operatorname{sh}(\underline{a}^{(2)};\barb^{(4)},\barb_0,\barb^{(1)}), \emm_*^B\left(\underline{b}^{(2)}\right), \operatorname{sh}(\bara^{(3)},\bara_0,\bara^{(1)};\underline{b}^{(3)}) \\
    + \sum (-1)&^{|b^{(3)}|'(|a_0|+|a^{(1)}|'+|a^{(2)}|') + |\bfa|}\ 1\otimes 1 | \operatorname{sh}(\underline{a}^{(2)};\underline{b}^{(3)}), \emm_*^B\left(\barb^{(4)},\barb_0,\barb^{(1)}\right), \operatorname{sh}(\bara^{(3)},\bara_0,\bara^{(1)} ;\underline{b}^{(2)})
\end{align}
which cancel out respectively with (\ref{eqA.6b}), (\ref{eqA.7b}), (\ref{eqA.8b}), and (\ref{eqA.5}). $\hfill\qed$
\newpage
\section{Appendix B}

\subsection{Proof of Lemma \ref{appendixlemma}}\label{secA.2}

We compute the LHS.
\vspace{-0.5cm}
\begin{align}
    \Bee\{\emm'\} \circ k (\bfa; \bfb)&= \Bee\{\emm'\}\left(\sum(-1)^\bullet \ a_0\otimes b_0 \ |\ \bigotimes_{i=1}^4\operatorname{sh}(\bara^{(i)};\barb^{(i)})\right)\notag\\
    =\sum(-1)^{\bullet+@+<} \ 1\otimes 1|\operatorname{sh}(\bara^{(2)};\barb^{(2)}),& \emm'_\otimes\left(\operatorname{sh}(\bara^{(3)};\barb^{(3)})\right), \operatorname{sh}(\bara^{(4)};\barb^{(4)}),a_0\otimes b_0, \operatorname{sh}(\bara^{(1)};\barb^{(1)}) \label{eqB}\\
    \text{where } \ \bullet=|b_0||\bfa|'&+(\sum_{i=1,\ j>i}^{3} |b^{(i)}|'|a^{(j)}|'), \ <\ =|a^{(2)}|'|b^{(2)}|'\notag\\
    \text{and } \ @ &=(|a_0|+|b_0|+|a^{(1)}|'+|b^{(1)}|')(\sum_{i=2}^4|a^{(i)}|'+|b^{(i)})|')\notag
\end{align}
Since $\emm'_2=0$ and by the commuting subalgebra relations (\ref{commutingsubalgebradefn}), we must have $\operatorname{sh}(\bara^{(3)};\barb^{(3)})= \bara^{(3)}$ or $\barb^{(3)}$ but not both, in equation (\ref{eqB}), so we split it as
\begin{align}
    \Bee\{\emm'\} \circ k (\bfa; \bfb)&= \notag\\
    \sum(-1)^{\bullet+@+<} \ 1\otimes 1|\operatorname{sh}(\bara^{(2)};\barb^{(2)}),& \emm'_A\left(\bara^{(3)}\right), \operatorname{sh}(\bara^{(4)};\barb^{(3)}),a_0\otimes b_0, \operatorname{sh}(\bara^{(1)};\barb^{(1)}) \label{eqB.2}\\
    \sum(-1)^{\bullet+@+<} \ 1\otimes 1|\operatorname{sh}(\bara^{(2)};\barb^{(2)}),& \emm'_B\left(\barb^{(3)}\right), \operatorname{sh}(\bara^{(3)};\barb^{(4)}),a_0\otimes b_0, \operatorname{sh}(\bara^{(1)};\barb^{(1)})\label{eqB.3}
\end{align}
\begin{align}
    k\circ (\Bee\{\emm'\} \bfa\otimes \bfb) &= \sum(-1)^{@_A+<_A+|b_0||a|'+1}\ 1\otimes b_0|\operatorname{sh}\left(\bara^{(2)}, \emm'_A(\bara^{(3)}), \bara^{(4)}, a_0, \bara^{(1)}; \barb\right)\label{eqB.4}\\
    (-1)^{|\bfa|}k(\bfa;  \Bee\{\emm'\}\bfb)&= \sum(-1)^{@_B+<_B+|\bfa|}\ a_0\otimes 1|\operatorname{sh}\left(\bara;\barb^{(2)}, \emm'_B(\barb^{(3)}), \barb^{(4)}, b_0, \barb^{(1)}\right)\label{eqB.5}\\
    \text{where } <_A=&|a^{(2)}|' ;\ \ <_B= |b^{(2)}|'\notag
\end{align}
\begin{align}
    \bee\{\emm'\} \circ K(\bfa; \bfb)&= \sum_{\barb_0<\bara_0} (-1)^{\bullet+@+<} \notag\\
    \emm_*^\otimes(\operatorname{Sh}(\bfa;\bfb)^{(3)},& \emm'_\otimes(\operatorname{Sh}(\bfa;\bfb)^{(4)}), \operatorname{Sh}(\bfa;\bfb)^{(5)}, 1\otimes 1, \operatorname{Sh}(\bfa;\bfb)^{(1)})|\operatorname{Sh}(\bfa;\bfb)^{(2)}\label{eqB.6}\\
    \text{where } \ \bullet=&(\sum_{i=1,\ j>i}^{4} |b^{(i)}|'|a^{(j)}|'), \ < \ = |a^{(3)}|'+|b^{(3)}|' \notag
\end{align}
Again, since $\emm'_2=0$ and by the commuting subalgebra relations, the cyclic shuffle term above $\operatorname{Sh}(\bfa;\bfb)^{(4)}= \bfa^{(4)}$ or $\bfb^{(4)}$ but not both. Also, by strict unitality, $*=2$. We also break up the sum (\ref{eqB.6}) into whether $\bara_0$ is an input of $ \emm'$, or not. Observe that the restriction $\barb_0< \bara_0$ means $\barb_0$ is not an input of $ \emm'$. So we obtain
\begin{align}
    \bee\{\emm'\} \circ K(\bfa; \bfb)&= \sum_{\barb_0<\bara_0} (-1)^{\bullet+@} \emm_2^\otimes\left(\emm'_A(\bara^{(2)}), 1\otimes1\right)|\operatorname{sh}_0(\bara^{(3)},\bara_0,\bara^{(1)},N\underline{\bfb})\label{eqB.7a}\\
    &+ \sum_{\barb_0<\bara_0} (-1)^{\bullet+@} \emm_2^\otimes\left(\emm'_B(\barb^{(2)}), 1\otimes1\right)|\operatorname{sh}_0(N\underline{\bfa},\barb^{(3)},\barb_0,\barb^{(1)})\label{eqB.8a}\\
    &+ \sum (-1)^{\bullet+@} \emm_2^\otimes\left(\emm'_A(\bara^{(3)},\bara_0,\bara^{(1)}), 1\otimes1\right)|\operatorname{sh}(\bara^{(2)},N\underline{\bfb}) \label{eqB.7} 
\end{align}

\footnotesize{
\begin{align}
    K\circ (\bee\{\emm'\} \bfa; \bfb)&= \sum_{\barb_0<\emm_*^A(\bara^{(3)}, \emm'_A(\bara^{(4)}), \bara^{(5)},\bara_0,\bara^{(1)}}(-1)^{@_A+<_A} \ 1\otimes1|\operatorname{sh}_0\left(\emm_*^A(\bara^{(3)}, \emm'_A(\bara^{(4)}), \bara^{(5)},\bara_0,\bara^{(1)})| \bara^{(2)}; N\bfb\right)\label{eqB.10}\\
    (-1)^{|\bfa|}K(\bfa ;  \bee\{\emm'\}\bfb)&= \sum_{\emm_*^B(\barb^{(3)}, \emm'_B(\barb^{(4)}), \barb^{(5)},\barb_0,\barb^{(1)}<\bara_0}(-1)^{@_A+@_B+|\bfa|} \ 1\otimes1|\operatorname{sh}_0\left(N\bfa;\emm_*^B(\barb^{(3)}, \emm'_B(\barb^{(4)}), \barb^{(5)},\barb_0,\barb^{(1)})| \barb^{(2)}\right)\label{eqB.11}
\end{align}
}
\normalsize
\begin{align}
    -k (\mathcal{L}_{\emm'} \otimes \Bee)(\bfa \otimes \bfb)&=\sum(-1)^{1+<_A+@_B}\ a_0\otimes1|\operatorname{sh}(\bara^{(1)}, \emm'_A(\bara^{(2)}), \bara^{(3)}; \barb^{(2)}, \barb_0, \barb^{(1)}) \label{eqB.12}\\
    +&\sum(-1)^{1+@_A+@_B} \ \emm'_A(a^{(3)},a_0,a^{(1)})\otimes1|\operatorname{sh}(\bara^{(2)}; \barb^{(2)}, \barb_0, \barb^{(1)})\label{eqB.8}
\end{align}
Since the shuffle term in equation (\ref{eqB.8}) cyclically permutes $\underline{\bfb}$,
it is exactly the shuffle term in  (\ref{eqB.7}). A computation of the signs show that they cancel each other. We will show that the rest of the above terms are canceled by the terms in $\bee\circ H$.
\begin{align}
    \bee\circ H(\bfa; &\bfb)=\bee\left(\sum_{\barb_0<\bara_0}(-1)^{@_A+@_B+|a^{(2)}|'} \ 1\otimes 1|\operatorname{sh}(\bara^{(2)}, \emm'_A(\bara^{(3)}),\bara^{(4)},\bara_0,\bara^{(1)}; \barb^{(2)},\barb_0,\barb^{(1)})\right) \label{eqB.14}\\
    +\bee&\left(\sum_{\barb_0<\bara_0}(-1)^{@_A+@_B+|\bfa|+|b^{(2)}|'} \ 1\otimes 1|\operatorname{sh}(\bara^{(2)},\bara_0,\bara^{(1)};\barb^{(2)}, \emm'_B(\barb^{(3)}),\barb^{(4)},\barb_0,\barb^{(1)})\right)\label{eqB.15}
\end{align}
Equations (\ref{eqB.2}) and (\ref{eqB.3}) are canceled in (\ref{eqB.14}) and (\ref{eqB.15}) when $\bee$ acts by applying $\emm_2^\otimes$ to $(\barb_0,\bara_0)$ (with the correct signs). The restriction $\barb_0<\bara_0$ guarantees these do not come in commuting pairs.

Equation (\ref{eqB.4}) is canceled in (\ref{eqB.14}) when $\bee$ acts by applying $\emm_2^\otimes$ before "$|$" to $(1\otimes 1, 1\otimes \barb_0)$. Equations (\ref{eqB.5}) and (\ref{eqB.12}) are canceled in (\ref{eqB.15}) and (\ref{eqB.14}) respectively, when $\bee$ acts by applying $\emm_2^\otimes$ before "$|$" to $(\bara_0\otimes 1, 1\otimes 1)$.

Similarly, (\ref{eqB.7a}) and (\ref{eqB.8a}) are canceled in (\ref{eqB.14}) and (\ref{eqB.15}) when $\bee$ acts by applying $\emm_2^\otimes$ to $(1\otimes1,\emm'_A(a^{(2)}))$, and to $(1\otimes1,\emm'_B(b^{(2)}))$ respectively. By strict unitality, and the restriction $\barb_0<\bara_0$, these are the only possibilities for applying $\emm_2^\otimes$ on the "outside" before "$|$", with $1\otimes 1.$ All the other cases $\emm_2^\otimes(a_i\otimes1, 1\otimes1)$ and $\emm_2^\otimes(1\otimes1, 1\otimes b_i)$ come in commuting pairs and thus cancel each other.

Equations (\ref{eqB.10}) and (\ref{eqB.11}) are canceled in (\ref{eqB.14}) and (\ref{eqB.15}) when $\bee$ acts by applying $\emm_*^\otimes$ to $(\bara^{(3)}, \emm'_A(\bara^{(4)}), \bara^{(5)}, \bara_0,\bara^{(1)})$, and on $(\barb^{(3)}, \emm'_B(\barb^{(4)}), \barb^{(5)}, \barb_0,\barb^{(1)})$.

The remaining terms in $\bee\circ H$ unaccounted for are:
\begin{align}
    \sum_{\barb_0<\bara_0} (-1)&^{@_A+@_B+|\bfa|'+|b_0|+|b^{(1)}|'+|b^{(4)}|'} \ 1\otimes 1|\operatorname{sh}(\bara^{(2)}, \emm'_A(\bara^{(3)}),\bara^{(4)},\bara_0,\bara^{(1)}; \barb^{(4)},\barb_0,\barb^{(1)},\emm_*^B(\barb^{(2)}),\barb^{(3)}) \label{eqB.17}\\
   +\sum_{\barb_0<\bara_0} &(-1)^{@_A+@_B+|b^{(2)}|'+|a^{(2)}|'+|a^{(3)}|'} 1\otimes1|\operatorname{sh}(\bara^{(4)}, \bara_0, \bara^{(1)}, \emm_*^A(\bara^{(2)}), \bara^{(3)}; \barb^{(2)}, \emm'_B(\barb^{(3)}), \barb^{(4)}, \barb_0,  \barb^{(1)})\label{eqBB.14}\\
   +\sum_{\barb_0<\bara_0}& (-1)^{@_A+@_B+|\bfa|'} 1\otimes 1|\operatorname{sh}(\bara^{(3)}, \emm'_A(\bara^{(4)}),\bara^{(5)},\emm_*(\bara^{(6)}, \bara_0,\bara^{(1)}),\bara^{(2)}; \barb^{(2)},\barb_0,\barb^{(1)})\label{eqBB.15}\\
   + \sum_{\barb_0<\bara_0}& (-1)^{@_A+@_B+|\bfa|'+|b^{(3)}|'+|a^{(2)}|'} \ 1\otimes 1|\operatorname{sh}(\bara^{(3)}, \emm_*^A(\bara^{(4)},\bara_0,\bara^{(1)}),\bara^{(2)}; \barb^{(2)}, \emm'_B(\barb^{(3)}), \barb^{(4)},\barb_0,\barb^{(1)})\label{eqB.16}\\
   + \sum_{\barb_0<\bara_0}& (-1)^{@_A+@_B+|\bfa|'+|a^{(2)}|'+|b^{(3)}|'} \ 1\otimes 1|\operatorname{sh}(\bara^{(2)}, \emm'_A(\bara^{(3)}), \bara^{(4)},\bara_0,\bara^{(1)};\barb^{(3)}, \emm_*^B(\barb^{(4)},\barb_0,\barb^{(1)}),\barb^{(2)})\label{eqB.22}\\
   + \sum_{\barb_0<\bara_0}& (-1)^{@_A+@_B+|b^{(2)}|'+|b^{(3)}|'} \ 1\otimes 1|\operatorname{sh}(\bara^{(2)},\bara_0,\bara^{(1)};\barb^{(3)}, \emm'_B(\barb^{(4)}),\barb^{(5)},\emm_*^\otimes(\barb^{(6)}, \barb_0,\barb^{(1)}),\barb^{(2)})\label{eqB21}\\
   + \sum_{\barb_0<\bara_0}& (-1)^{@_A+@_B} \ 1\otimes 1|\operatorname{sh}(\bara^{(2)}, \emm_*^A(\bara^{(3)},\emm'_A(\bara^{(4)}), \bara^{(5)}),\bara^{(6)},\bara_0,\bara^{(1)}; \barb^{(2)},\barb_0,\barb^{(1)})\label{eqB.29-30}\\
   + \sum_{\barb_0<\bara_0}& (-1)^{@_A+@_B} \ 1\otimes 1|\operatorname{sh}(\bara^{(2)},\bara_0,\bara^{(1)}; \barb^{(2)},\emm_*^B(\barb^{(3)},\emm'_B(\barb^{(4)}),\barb^{(5)}),\barb^{(6)},\barb_0,\barb^{(1)})\label{eqB31}
\end{align}
We will show that these are all canceled by terms in $H\circ\bee$.

Now $H\circ\bee(\bfa; \bfb) = H(\bee(\bfa); \bfb) + (-1)^{|\bfa|} H(\bfa; \bee(\bfb))$, so we compute term by term.
\begin{align}
    H(\bee(\bfa); \bfb) & =\notag\\
    \sum_{\barb_0<\bara_0}(-1)^{@_A+@_B+|a^{(2)}|'+<_A} \ &1\otimes 1|\operatorname{sh}(\bara^{(2)}, \emm'_A(\bara^{(3)}),\bara^{(4)},\emm_*(\bara^{(5)}), \bara^{(6)},\bara_0,\bara^{(1)}; \barb^{(2)},\barb_0,\barb^{(1)})\label{eqB.11X}\\
    +\sum_{\barb_0<\bara_0}(-1)^{@_A+@_B+|a^{(2)}|'+<_A} \ &1\otimes 1|\operatorname{sh}(\bara^{(2)}, \emm'_A(\bara^{(3)},\emm_*(\bara^{(4)}), \bara^{(5)}),\bara^{(6)},\bara_0,\bara^{(1)}; \barb^{(2)},\barb_0,\barb^{(1)})\label{eqB.29-30X}\\
    +\sum_{\barb_0<\bara_0}(-1)^{@_A+@_B+|a^{(2)}|'+<_A} \ &1\otimes 1|\operatorname{sh}(\bara^{(2)}, \emm_*(\bara^{(3)}), \bara^{(4)},\emm'_A(\bara^{(5)}), \bara^{(6)},\bara_0,\bara^{(1)}; \barb^{(2)},\barb_0,\barb^{(1)})\label{eqB.13X}\\
    +\sum_{\barb_0<\bara_0}(-1)^{@_A+@_B+<_A+|\bfa|'+|b^{(2)}|'}\ &1\otimes1|\operatorname{sh}(\bara^{(4)}, \bara_0, \bara^{(1)}, \emm_*^A(\bara^{(2)}), \bara^{(3)}; \barb^{(2)}, \emm'_B(\barb^{(3)}), \barb^{(4)}, \barb_0,  \barb^{(1)})\label{eqB.14X}\\
    +\sum_{\barb_0<\bara_0}(-1)^{@_A+|a^{(3)}|'} \ 1\otimes 1|\operatorname{sh}(&\bara^{(3)}, \emm'_A(\bara^{(4)}),\bara^{(5)},\emm_*(\bara^{(6)}, \bara_0,\bara^{(1)}),\bara^{(2)}; \barb^{(2)},\barb_0,\barb^{(1)})\label{eqB.15X}\\
    +\sum_{\barb_0<\bara_0}(-1)^{@_A+|\bfa|+|b^{(2)}|'} \ 1\otimes 1|\operatorname{sh}(&\bara^{(3)}, \emm_*^A(\bara^{(4)},\bara_0,\bara^{(1)}),\bara^{(2)}; \barb^{(2)}, \emm'_B(\barb^{(3)}), \barb^{(4)},\barb_0,\barb^{(1)})\label{eqB.16X}
\end{align}
\begin{align}
    (-1)^{|\bfa|}&H(\bfa; \bee(\bfb))= \notag\\
    \sum_{\barb_0<\bara_0}(-1)^{|\bfa|+@_A+@_B+|a^{(2)}|'+<_B} \ 1\otimes 1|&\operatorname{sh}(\bara^{(2)}, \emm'_A(\bara^{(3)}),\bara^{(4)},\bara_0,\bara^{(1)}; \barb^{(4)},\barb_0,\barb^{(1)},\emm_*^B(\barb^{(2)}),\barb^{(3)}) \label{eqB.17X}\\
    +\sum_{\barb_0<\bara_0}(-1)^{|\bfa|+@_A+@_B+|b^{(2)}|'+<_B} \ 1\otimes 1|&\operatorname{sh}(\bara^{(2)},\bara_0,\bara^{(1)}; \barb^{(2)},\emm'_B(\barb^{(3)}),\barb^{(4)},\emm_*^B(\barb^{(5)}),\barb^{(6)},\barb_0,\barb^{(1)}) \label{eqB.18X}\\
    +\sum_{\barb_0<\bara_0}(-1)^{|\bfa|+@_A+@_B+|b^{(2)}|'+<_B} \ 1\otimes 1|&\operatorname{sh}(\bara^{(2)},\bara_0,\bara^{(1)}; \barb^{(2)},\emm'_B(\barb^{(3)},\emm_*^B(\barb^{(4)}),\barb^{(5)}),\barb^{(6)},\barb_0,\barb^{(1)})\label{eqB31X}\\
    +\sum_{\barb_0<\bara_0}(-1)^{|\bfa|+@_A+@_B+|b^{(2)}|'+<_B} \ 1\otimes 1|&\operatorname{sh}(\bara^{(2)},\bara_0,\bara^{(1)}; \barb^{(2)},\emm_*^B(\barb^{(3)}),\barb^{(4)},\emm'_B(\barb^{(5)}),\barb^{(6)},\barb_0,\barb^{(1)})\label{eqB.20X}\\
    +\sum_{\barb_0<\bara_0}(-1)^{@_A+@_B+|a^{(2)}|'+|\bfa|} \ 1\otimes 1|\operatorname{sh}(&\bara^{(2)},\bara_0,\bara^{(1)};\barb^{(3)}, \emm'_B(\barb^{(4)}),\barb^{(5)},\emm_*^\otimes(\barb^{(6)}, \barb_0,\barb^{(1)}),\barb^{(2)})\label{eqB.21X}\\
    +\sum_{\barb_0<\bara_0}(-1)^{@_A+@_B+|b^{(3)}|'+|\bfa|} \ 1\otimes 1|\operatorname{sh}(&\bara^{(2)}, \emm'_A(\bara^{(3)}), \bara^{(4)},\bara_0,\bara^{(1)};\barb^{(3)}, \emm_*^B(\barb^{(4)},\barb_0,\barb^{(1)}),\barb^{(2)}) \label{eqB.22X}
    \end{align}
We find that (\ref{eqB.17}) cancels with (\ref{eqB.17X}); (\ref{eqBB.14}) cancels with (\ref{eqB.14X}); (\ref{eqBB.15}) cancels with (\ref{eqB.15X}); (\ref{eqB.16}) cancels with (\ref{eqB.16X}); (\ref{eqB.22}) cancels with (\ref{eqB.22X}); (\ref{eqB21}) cancels with (\ref{eqB.21X}). Further, (\ref{eqB.18X}) and (\ref{eqB.20X}) cancel each other, as well as (\ref{eqB.11X}) and (\ref{eqB.13X}). Lastly, (\ref{eqB.29-30}) and (\ref{eqB31}) cancel with (\ref{eqB.29-30X}) and (\ref{eqB31X}) respectively since $[\emm_*, \emm']=0$. $\hfill\qed$

\printbibliography


\end{document}